\documentclass{article}
\RequirePackage{amsthm,amsmath,amsbsy,amssymb,bm,color,graphicx,hyperref}
\usepackage{ dsfont }

\def\R{\mathbb{R}}
\newcommand{\sigMod}{s}
\newcommand{\muMod}{b}

\newcommand{\errorModtotal}{\epsilon}
\newcommand{\errorMod}{\epsilon^{(1)}}
\newcommand{\errorModTwo}{\epsilon^{(2)}}
\newcommand{\errorModThree}{\epsilon^{(3)}}
\newcommand{\errorModFour}{\epsilon^{(4)}}
\newcommand{\errorModFive}{\epsilon^{(5)}}
\newcommand{\ingroupPara}{\left\|\bm{\pi}\right\|_1^{g(j),j}}
\newcommand{\ingroupParaI}{\left\|\bm{\pi}\right\|_1^{g(i),i}}
\newcommand{\ingroupParaL}{\left\|\bm{\pi}\right\|_1^{g(l),l}}

\newcommand{\E}{\ensuremath{\operatorname{\mathbb{E}}}}
\newcommand{\V}{\ensuremath{\operatorname{Var}}}

\def\No{\operatorname{Normal}}
\def\pij{\widehat{\E A}_{ij}}
\def\N{\mathbb{N}}
\def\diag{\operatorname{diag}}
\def\ingroup{ \delta_{g(i)= g(j)}}
\def\NOTingroup{ \delta_{g(i) \neq g(j)}}

\def\ingrouplm{ \delta_{g(l) = g(m)}}

\let\originalleft\left
\let\originalright\right
\renewcommand{\left}{\mathopen{}\mathclose\bgroup\originalleft}
\renewcommand{\right}{\aftergroup\egroup\originalright}

\theoremstyle{plain}
\newtheorem{Theorem}{Theorem}

\newtheorem*{corollary}{Corollary}
\newtheorem{definition}{Definition}

\newtheorem{appxtheorem}{Theorem}[section]
\newtheorem{appxcorollary}{Corollary}[section]
\newtheorem{appxlemma}{Lemma}[section]

\usepackage{tikz}
\usepackage{grffile}
\usepackage{pgfplotstable}

\pgfplotsset{compat=newest}

\begin{document}

\title{Network modularity in the presence of covariates}

\author{Beate Franke$^1$ \and Patrick J.\ Wolfe$^{1,2}$}

\date{$^1$Department of Statistical Science, University College London \newline%
 $^2$Department of Computer Science, University College London}

\maketitle

\begin{abstract}
We characterize the large-sample properties of network modularity in the presence of covariates, under a natural and flexible nonparametric null model.  This provides for the first time an objective measure of whether or not a particular value of modularity is meaningful. In particular, our results quantify the strength of the relation between observed community structure and the interactions in a network.  Our technical contribution is to provide limit theorems for modularity when a community assignment is given by nodal features or covariates. These theorems hold for a broad class of network models over a range of sparsity regimes, as well as weighted, multi-edge, and power-law networks.  This allows us to assign $p$-values to observed community structure, which we validate using several benchmark examples in the literature.  We conclude by applying this methodology to investigate a multi-edge network of corporate email interactions.

\vspace{\baselineskip}%
\end{abstract}

A fundamental challenge in modern science is to understand and explain network structure: in particular, the tendency of nodes in a network to connect in \emph{communities} based on shared characteristics or function. Scientists inevitably observe not only network nodes and their connections, but also additional information in the form of covariates. Most analysis methods fail to exploit this information when attempting to explain network structure, and instead assign communities based solely on the network itself.  This leads to a loss of interpretability and presents a barrier to understanding.  We solve this problem, by showing how to decide whether communities defined by covariates lead to a valid summary of network structure.  In the student friendship network shown in Fig.\ \ref{Fig:Addhealth}, for example, this means we can evaluate whether communities based on common gender, race, or year in school can explain the observed structure of the friendships.

The strength of community structure in networks is most often measured by modularity \cite{PhysRevE.69.026113}, which is intuitive and practically effective but until now has lacked a sound theoretical basis. We derive modularity from first principles, give it a formal statistical interpretation, and show why it works in practice.  Moreover, by acknowledging that different community assignments may explain different aspects of a network's observed structure, we extend the applicability of modularity beyond its typical use to find a single ``best'' community assignment.

We use covariates to define community assignments, and then prove that modularity quantifies how well these covariates explain network structure. We show a fundamental limit theorem for modularity in this context: in the presence of covariates, it behaves like a Normal random variable for large networks whenever there is a lack of community structure. This allows us to translate modularity into a probability (a $p$-value), enabling for the first time its use to draw defensible, repeatable conclusions from network analysis.

Our main technical contribution is a flexible, nonparametric approach to quantify the strength of observed community structure. Most work assumes a single \emph{unobserved} or latent community assignment (e.g., stochastic block models \cite{holland1983stochastic} and latent space models \cite{hoff2002latent}). Hoff et al.\ \cite{hoff2002latent} and Zhang et al.\ \cite{Levina2015} both estimate latent community structure, while adjusting for the varying effects of covariates. Fosdick and Hoff \cite{Fosdick2015} simultaneously model covariates and latent structure, providing a test for independence.  In contrast, we derive limit theorems to evaluate \emph{observed} community structure implied by the covariates themselves.

The existing statistical literature on modularity has focused on more basic parametric approaches.  For example, the authors of \cite{Arias-castro2009} and \cite{bickel2015hypothesis} model all edges as equally likely Bernoulli random variables. In contrast, we take a nonparametric approach: using a single parameter per node, we model only the expectation of each edge \cite{Chung2002}. This allows for individual node-specific differences but avoids specific distributional assumptions on the edges. Our results apply to a broad class of network models, allowing us to treat (among others) power-law networks, weighted networks, and those with multiple edges.

\begin{figure}[t!]
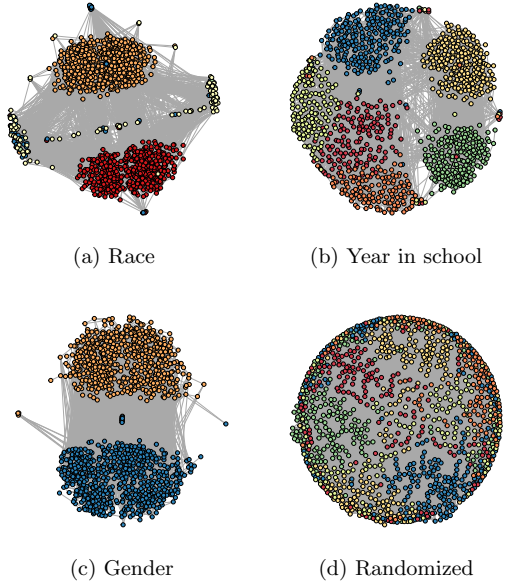

\centering
\scalebox{0.9}{\input{Addhealth_race}}	\hspace{0.4cm}
\scalebox{0.9}{\input{Addhealth_grade}} \\[0.4cm] 	
\scalebox{0.9}{\input{Addhealth_gender}}	\hspace{0.4cm}
\scalebox{0.9}{\input{Addhealth_random}}				
\caption{\label{Fig:Addhealth}A student friendship network illustrated for four different community assignments, each defined by a covariate \cite{Fosdick2015, resnick1997protecting}.\vspace{-0.5\baselineskip}}			
\end{figure}
\section{Network modularity in the presence of covariates}

Two essential ingredients are necessary to understand modularity in the presence of covariates: first, a framework to allow for a formal interpretation of modularity as a measure of statistical significance; and second, the use of this framework to evaluate a covariate-based community assignment. We now describe each of these ingredients in turn.

First, to interpret modularity as a measure of statistical significance, we must recognize it as an estimator of a population quantity. Let $g(\cdot)$ denote an assignment of nodes into groups (i.e., communities), and write $\ingroup = 1$ when nodes $i$ and $j$ are assigned to the same group, and $0$ otherwise.  Denote by $A_{ij}$ the strength of an edge (e.g., a count or a weight) between nodes $i$ and $j$, and by $d_i = \sum_{j\neq i} A_{ij}$ the degree of the $i$th node. Then, modularity as defined in \cite{PhysRevE.69.026113} is
\begin{align}\label{Def:modularityempirical}
	 \widehat{Q} = \sum_{j=1}^n \sum_{i<j} \left[A_{ij} - \frac{d_i d_j}{\sum_{l=1}^n d_l} \right] \ingroup.
\end{align}

Modularity contrasts an observed edge $A_{ij}$ with the ratio $d_i d_j/\sum_l d_l$ whenever nodes $i$ and $j$ are in the same community. Now consider replacing $d_i d_j/\sum_l d_l$ by $\E A_{ij}$, the expected value of an edge under a given model:
\begin{align}\label{Def:modularity}
		Q = \sum_{j=1}^n \sum_{i<j} \left[A_{ij} - \E A_{ij} \right] \ingroup.
\end{align}
We recognize $Q$ in Eq.\ \eqref{Def:modularity} as a sum of signed residuals (observed minus expected values) $A_{ij} - \E A_{ij}$.  If the model for each $\E A_{ij}$ posits the \emph{absence} of community structure, then a large positive value of $Q$ indicates the \emph{presence} of such structure (more within-group edges than expected). Figure \ref{Fig:Addhealth} illustrates this effect: the visible community structure in Figs.\ \ref{Fig:Addhealth}a--c is obscured in Fig.\ \ref{Fig:Addhealth}d when communities are assigned at random.  Moreover, using $d_i d_j/\sum_l d_l$ as a proxy for $\E A_{ij}$, we see that modularity $\smash{\widehat{Q}}$ as defined in Eq.\ \eqref{Def:modularityempirical} is an estimator of $Q$ in Eq.\ \eqref{Def:modularity}. We will return to this point in the next section.

Second, to interpret covariate-based community structure, we must recognize that different community assignments reveal different structural aspects of a network.  Figures\ \ref{Fig:Addhealth}a--c illustrate this point using a student friendship network grouped by gender, race, and year in school.  Covariates such as these define distinct community assignments, each of which relates the covariate in question to the observed network structure.

A key insight is that rather than maximizing modularity to obtain a single ``best'' community assignment, we may instead use modularity to measure the strength of an observed community structure. If a particular community assignment is given by a covariate, then modularity allows us to quantify the explanatory value of this covariate for the observed structure of the network.
			
\section{Main result: A limit theorem for modularity}

Our main result is a practical tool to understand objectively whether a covariate captures the structure of the interactions in a network. Technically, we derive a theorem quantifying the large-sample behavior of modularity in the setting above. In particular, if the null model of Definition \ref{Def:DegreebasedModel} below is in force, then modularity in the presence of covariates behaves like a Normal random variable. This enables us to associate a $p$-value with any observed community structure, quantifying how unlikely it is (under the null) to observe a community structure \emph{at least as extreme as} the one we observe.

\begin{Theorem}[Central limit theorem for modularity]
\label{Thm:DensityM}
Suppose the null model of Definition \ref{Def:DegreebasedModel} below is in force, and consider a sequence of networks where for each $n$ we observe a fixed (non-random) group assignment $g(1), g(2), \ldots, g(n)$.  Then as long as the number of groups grows strictly more slowly than $n$, there exist constants $\muMod$ and $\sigMod$ for each $n$ such that as $n \to \infty$,
\begin{equation*}
	\frac{\widehat{Q}- \muMod}{\sigMod} \stackrel{d}{\rightarrow} \No(0,1).
\end{equation*}
\end{Theorem}

\begin{proof}
	Proofs of all results are given in the Appendices.
\end{proof}

Thus, when appropriately shifted and scaled, modularity converges in distribution to a standard Normal random variable. In the sequel we explain this result and give explicit formulations for $\muMod$ and $\sigMod^2$ (Eqs.\ \eqref{muMod} and \eqref{sigMod} below).

\section{The network model underlying modularity} \label{DegreeBasedNull}

To understand Theorem \ref{Thm:DensityM}, we must establish a technical foundation for modularity in the presence of covariates. Different models for the network edges $A_{ij}$ will imply different estimators for $Q$ in Eq.\ \eqref{Def:modularity}. Estimating $Q$ using $\smash{\widehat{Q}}$ in Eq.\ \eqref{Def:modularityempirical}, we indirectly assume a model for the absence of community structure, where nodes connect independently based on the product of their individual propensities to form connections \cite{Chung2002, Wolfe2012, Olhede2012}.

\begin{definition}[The network model underlying modularity]
\label{Def:DegreebasedModel}
Consider an undirected, random graph on $n$ nodes without self-loops. We model its (possibly weighted) edges $A_{ij} \geq 0$ as independent random variables with expectations given by the product of node-specific parameters $\pi_1, \pi_2, \ldots, \pi_n >0$:
\begin{align*}
	\E A_{ij}= \pi_i \pi_j, \quad 1 \leq i< j\leq n.
\end{align*}
Furthermore, considering a sequence of such networks as $n$ grows, we assume they are well behaved asymptotically:
\begin{enumerate}
	\item No single node dominates the network: $\max_i \pi_i/ \bar{\pi}$, with $\bar{\pi}= \frac{1}{n} \sum_{l=1}^n \pi_l$, is bounded asymptotically; \label{nonodeincontrol}
	\item The network is not too sparse: $\min_i \pi_i \cdot \sqrt{n}$ diverges as $n$ grows;\label{sparse}
	\item The expectation of each edge $\E A_{ij}$ does not diverge too quickly as $n$ grows: $\max_i \pi_i/\sqrt{n}$ goes to 0; \label{notcomplete}
	\item The variance of each edge does not vary too much from its expectation: $\V A_{ij}/ \E A_{ij}$ is bounded from above and away from 0 asymptotically; and%
 \label{over-dispersed}
		\item The skewness of each edge $A_{ij}$ is controlled: the third central moment $ \smash{\E \left[\left(A_{ij} - \E A_{ij}\right)^3\right]}$ divided by the variance $ \V A_{ij}$ is bounded asymptotically.  \label{skewed}
\end{enumerate}
\end{definition}

We make no further assumptions on the distribution of $A_{ij}$, and so our results apply in many settings, including weighted networks and those with multiple edges. Assumptions \ref{nonodeincontrol}--\ref{notcomplete} are structural: the first excludes star-like networks; the second ensures that the network is not too sparse; and the third controls the growth of $\E A_{ij}$ with $n$ in the weighted or multi-edge setting. Assumptions \ref{over-dispersed} and \ref{skewed} are technical; they exclude extreme behavior of the edge variables. For instance, both are fulfilled whenever $A_{ij} \sim \operatorname{Bernoulli}\left(\pi_i \pi_j\right)$ or $A_{ij} \sim \operatorname{Poisson}\left(\pi_i \pi_j\right)$.

Each parameter $\pi_i$ describes the relative popularity of node $i$. Thus, to fit the degree-based model of Definition \ref{Def:DegreebasedModel} to a network, we estimate the parameters $\pi_i$ using the node's degrees $d_i$ as follows \cite{Chung2002, Wolfe2012, Olhede2012}:
\begin{align}
\hat{\pi}_i = \frac{d_i}{\sqrt{\sum_{l=1}^n d_l}}, \quad 1 \leq i \leq n.
\label{pihat}
\end{align}

The estimator $\hat{\pi}_i$ is both more natural and more computationally efficient than the corresponding maximum-likelihood estimator for $\pi_i$, which follows from the theory of generalized linear models and cannot be written explicitly in closed form.  In many settings the difference between these estimators is provably small \cite{Wolfe2012}, and so properties of maximum likelihood estimation can also be expected to hold for Eq.\ \eqref{pihat}.

Most importantly, we show that any finite collection of estimators defined by Eq.\ \eqref{pihat} tends toward a multivariate Normal distribution when $n$ is large and Definition \ref{Def:DegreebasedModel} is in force.  This generalizes a univariate result in \cite{Olhede2012} which assumes $\operatorname{Bernoulli}\left(\pi_i \pi_j\right)$ edges and a power law degree distribution.

\begin{Theorem} [Multivariate central limit theorem for Eq.\ \eqref{pihat}] \label{GeneralCLTPoisson}
Assume the model of Definition \ref{Def:DegreebasedModel} and any finite set of estimators from Eq.\ \eqref{pihat}.
Relabeling the indices of these estimators from 1 to $r$ without loss of generality, we have that as $n \to \infty$,
	 \begin{align*}
	\sqrt{\sum_{l=1}^n \E d_l} \left(\frac{\hat{\pi}_1-\pi_1}{\sqrt{\V d_1}}, \ldots,  \frac{\hat{\pi}_r-\pi_r}{\sqrt{\V d_r}}\right)
	 		 \stackrel{d}{\rightarrow}\No(0, I_{r}).
	 \end{align*}
Furthermore, $\sqrt{n \V d_i/\sum_{l=1}^n \E d_l}$ is bounded asymptotically, and can be consistently estimated if $A_{ij} \sim \operatorname{Bernoulli}\left(\pi_i \pi_j\right)$ or $\operatorname{Poisson}\left(\pi_i \pi_j\right)$ by substituting $\hat{\pi}$ for $\pi$ in $\V d_i$ $\!$and $\E d_i$.
\end{Theorem}

From Definition \ref{Def:DegreebasedModel} and Eq.\ \eqref{pihat}, it is natural to define
\begin{align*}
 \pij = \hat{\pi}_i \hat{\pi}_j= \frac{d_i d_j}{\sum_{l=1}^n d_l}, \quad 1 \leq i< j\leq n.
\end{align*}
Substituting $\pij$ for $\E A_{ij}$ in Eq.\ \eqref{Def:modularity}, we immediately recognize modularity $\smash{\widehat{Q}}$ as defined in Eq.\ \eqref{Def:modularityempirical}.  Thus, modularity implicitly assumes the degree-based model of Definition \ref{Def:DegreebasedModel}.

Moreover, $\smash{\pij} - \E A_{ij}$ converges in probability to zero under the model of Definition \ref{Def:DegreebasedModel} (see Appendices).  As a consequence of Theorem \ref{GeneralCLTPoisson}, we then obtain a central limit theorem for $\smash{\pij}$.

\begin{corollary}
As $n \to \infty$ under the model of Definition \ref{Def:DegreebasedModel},
\begin{align*}								
		\frac{\pij-\E A_{ij}}{\sqrt{\left(\pi_j^2 \V d_i + \pi_i^2 \V d_j\right)/ \sum_{l=1}^n \E d_l}} \stackrel{d}{\rightarrow} \No(0,1).
\end{align*}
Furthermore, $\smash{ \sqrt{ \left[n / \E A_{ij}\right] \cdot\left(\pi_j^2 \V d_i + \pi_i^2 \V d_j\right)/ \sum_{l=1}^n \E d_l} }$ is bounded asymptotically, and can be consistently estimated if $A_{ij} \sim \operatorname{Bernoulli}\left(\pi_i \pi_j\right)$ or $A_{ij} \sim \operatorname{Poisson}\left(\pi_i \pi_j\right)$ by substituting $\hat{\pi}$ for~$\pi$.			
\end{corollary}

This result leads to the first of two key insights as to why modularity, when appropriately shifted and scaled, behaves like a $\No(0,1)$ random variable.  Recall that $\smash{\widehat{Q}}$ (Eq.\ \eqref{Def:modularityempirical}) is an estimator for its population counterpart $Q$ (Eq.\ \eqref{Def:modularity}), in which $\smash{\pij}$ estimates $\E A_{ij}$. Comparing Eqs.\ \eqref{Def:modularityempirical} and \eqref{Def:modularity}, and approximating $\smash{\pij}$ by $\E d_i d_j / \sum_{l=1}^n \E d_l$, we obtain:
\begin{align*}
\E ( \widehat{Q} - Q ) \approx \sum_{j=1}^n  \sum_{i<j} \left(\E A_{ij} - \frac{\E d_i d_j}{ \sum_{l=1}^n \E d_l} \right) \ingroup.
\end{align*}
Under the model of Definition \ref{Def:DegreebasedModel}, this difference cancels to first order (see Appendices), yielding an approximate bias term of
\begin{align}
\label{muMod}	\muMod &= \sum_{j=1}^n \sum_{i<j}  \frac{\E A_{ij} \left(\E d_i + \E d_j - \sum_{l=1}^n \pi_l^2 \right)}{\sum_{l=1}^n \E d_l} \ingroup.
\end{align}
This is precisely the shift term appearing in Theorem \ref{Thm:DensityM}.

\begin{figure}[t!]
\centering
\includegraphics[scale=0.21,trim = 2.3cm 2.5cm 2.6cm 3.8cm, clip]{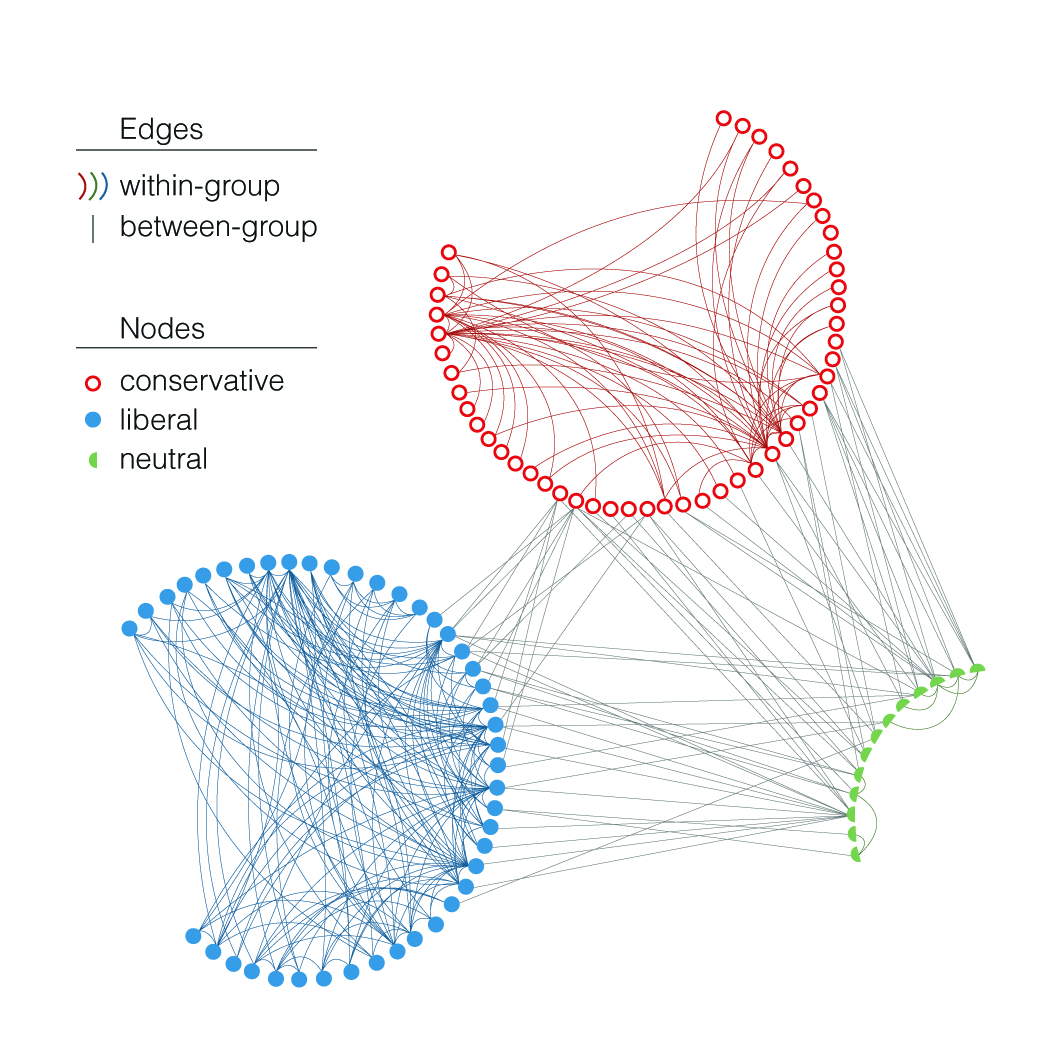}
\caption{\label{FigWithinBetweenPolbooks}Within- and between-group edges in a network of political books frequently purchased together, where groups are defined by political alignment \cite{newman2006modularity}.  Note that only within-group edges appear in $Q$ (Eq.\ \eqref{Def:modularity}); by contrast, both types of edges contribute to modularity $\smash{\widehat{Q}}$ (Eq.\ \eqref{Def:modularityempirical}).\vspace{-0.75\baselineskip}}
\end{figure}

\section{Modularity reflects within- and between-group edges}
\label{Significance}

Figure \ref{FigWithinBetweenPolbooks} illustrates the second main insight into the limiting behavior of modularity: its variability reduces asymptotically to that of a centered sum of within- and between-group edges.

More specifically, every network degree $d_i = \smash{\sum_{j\neq i} A_{ij}}$ decomposes into within- and between-group components:
\begin{align*}
    &d_i = d_i^w + d_i^b;
\\ 	d_i^w = \sum_{j \neq i} A_{ij} & \ingroup, \,\, \quad
 	d_i^b = \sum_{j \neq i} A_{ij} \NOTingroup.
\end{align*}

This decomposition is surprisingly powerful, in part because the model of Definition \ref{Def:DegreebasedModel} asserts that $d_i^w$ and $d_i^b$ are statistically independent for any fixed group assignment $g(1), g(2), \ldots, g(n)$. After separating the systematic bias term $b$ in modularity from its random variation, we obtain the following decomposition.

\begin{Theorem}[Bias--variance decomposition for modularity]
\label{modaswithinandbetween}
Under the null model of Definition \ref{Def:DegreebasedModel} and for a fixed (non-random) group assignment $g(1), g(2), $ $ \ldots, g(n)$, it holds that
\begin{equation*}
\widehat{Q} - \muMod = \sum_{i=1}^n \alpha_i \left[d_i^w - \E d_i^w\right]  + \sum_{i=1}^n \beta_i \left[d_i^b - \E d_i^b\right] + \epsilon,
\end{equation*}
where $\epsilon$ is a random error term, $\alpha_i = 1/2 + \beta_i$, and
\begin{equation*}
\beta_i = \left[ \frac{1}{2} \frac{\sum_{l=1}^n \E d_l^w}{\sum_{l=1}^n \E d_l} - \frac{ \E d_i^w}{\E d_i}\right], \quad 1 \leq i \leq n.
\end{equation*}
\end{Theorem}

Theorem \ref{modaswithinandbetween} quantifies the random variability inherent in modularity under the model of Definition \ref{Def:DegreebasedModel}. It establishes that a main term contributing to the variability of $\smash{\widehat{Q}} - \muMod$ in this setting is a linear combination of centered within- and between-group degrees ($d_i^w,d_i^b$), which for each $i$ are statistically independent.  The weights $\alpha_i$ and $\beta_i$ associated with this linear combination are determined by the global proportion of expected within-group edges in the network, 	relative to the local proportion of expected within-group edges specific to node~$i$.

Combining these two insights, we first shift modularity $\widehat{Q}$ by its approximate bias $\muMod$ and then scale it by the standard deviation $\sigMod$ of $\sum_{i=1}^n \alpha_i \left[d_i^w - \E d_i^w\right] +  \sum_{i=1}^n \beta_i \left[d_i^b - \E d_i^b\right] $, with
\begin{equation}
\label{sigMod}
\sigMod^2 = \sum_{j=1}^n \sum_{i < j} \left[  \ingroup  +  \beta_i + \beta_j \right]^2 \V A_{ij}.
\end{equation}
Recalling Theorem \ref{modaswithinandbetween}, we then know that we are left with a linear combination of centered within- and between-group degrees that are now also scaled by $s$. This leads directly to a central limit theorem for modularity $\smash{\widehat{Q}}$ as stated in Theorem~\ref{Thm:DensityM}:
\begin{equation*}
	\frac{\widehat{Q}- \muMod}{\sigMod} \stackrel{d}{\rightarrow} \No(0,1).
\end{equation*}

\section{Applying the limit theorem to benchmark examples}

\begin{table*}[t!]
	\vspace{0.5\baselineskip}%
\hskip-2.25cm
{\footnotesize
		\begin{tabular}{l l rrr rrrr r r  rr}
			Dataset (no.\ nodes) & Covariate (no.\ groups) &  \multicolumn{3}{c}{Degree percentiles}
			& \multicolumn{4}{c}{Simulated under the null}			
			& \multicolumn{2}{c}{Data as observed}
\\		& & \multicolumn{1}{c}{$25 \%$} & \multicolumn{1}{c}{$50 \%$} & \multicolumn{1}{c}{$75 \%$}
			&\multicolumn{2}{c}{$(\widehat{Q} - \hat{\muMod} ) / \hat{\sigMod}$}
			& \multicolumn{2}{c}{$p$-value}
			& $(\widehat{Q} - \hat{\muMod} ) / \hat{\sigMod}$ & $p$-value
\\		& &&&
		  &\multicolumn{1}{r}{mean}& std. &  \multicolumn{1}{r}{mean}& std.
			&&
\\[0.15cm]		
			Books ($105$) \cite{newman2006modularity} & Political alignment (3)
			& $5$	& $6$ & $9$ &$0.02$ &$1.01$ &$0.51$& $0.29$ & $21$  & $< 10^{-6}$  \\
			Jazz bands ($198$) \cite{gleiser2003community} &Recording location (17)
			& $16$ &	$25$ &	$39$
			& $0.01$ &$1.02$ &$0.51$& $0.29$ & $29$ & $< 10^{-6}$ \\
			Weblogs ($1224$) \cite{adamic2005political} &Political alignment (2)
			& $3$ &	$13$ &	$36$
			& $0.01$&  $1.04$& $0.50$ &$0.30$ & $118 $&$< 10^{-6}$ \\
			Co-authors ($36297$) \cite{newman2001structure}& Subject category (7)
			& $2$ & $5$ & $10$
			& $0.00$ & $1.00$& $0.50$ & $0.29$ & $472$& $< 10^{-6}$
\end{tabular}}
		\caption{\label{Tab: ModularityBinary} Analysis of four benchmark network datasets, using modularity derived from covariate-based community assignments.}
\end{table*}

Having established a central limit theorem for modularity in the presence of covariates, we now show how to apply this result in practice. To turn our theory into a methodology suitable for a specific network dataset, we first need to elicit a model for the data based on Definition \ref{Def:DegreebasedModel}. We then fit this model, leading ultimately to a $p$-value based on Theorem \ref{Thm:DensityM}.  We now illustrate the complete analysis procedure for four binary networks which, along with their covariates, frequently serve as benchmarks for community detection \cite{newman2006modularity, duch2005community}. Table~\ref{Tab: ModularityBinary} summarizes all data and results.
\begin{enumerate}
	\item First, we must further specify the null model of Definition~\ref{Def:DegreebasedModel}, so that the parameter $\sigMod^2$ in Eq.\ \eqref{sigMod} can be estimated. This can be done either by assuming sets of the variances $\V A_{ij}$ to be equal, or by assuming a distribution for the edges $A_{ij}$.  Since the benchmark networks we consider here are binary ($A_{ij}\in \left\{0, 1\right\}$), we model their edges as
\begin{align*}
A_{ij} \sim \operatorname{Bernoulli}(\pi_i \pi_j).
\end{align*}
	\item Second, we must assess whether the five asymptotic assumptions of Definition \ref{Def:DegreebasedModel} appear to hold for our data. Assumptions \ref{notcomplete}--\ref{skewed} are automatically satisfied for Bernoulli edges, and so we are left to assess Assumptions \ref{nonodeincontrol} ($\max_i \pi_i/\bar{\pi}$ bounded) and \ref{sparse} ($\min_i \pi_i \cdot \sqrt{n} $ growing). We do this by substituting $\hat{\pi}_i$ for $\pi_i$, noting that $\max_i \hat{\pi}_i / \bar{\hat{\pi}} = \max_i d_i / \bar{d}$ and $\min_i \hat{\pi}_i \cdot \sqrt{n}  = \min_i d_i / \smash{\sqrt{\bar{d}}}$. Replacing $\min_i d_i$, $\bar{d}$, and $\max_i d_i$ respectively by the first, second and third degree quartiles as shown in Table \ref{Tab: ModularityBinary}, we observe that for all four benchmark networks, these ratios are of order one. This indicates that these networks are neither too star-like nor too sparse for Theorem \ref{Thm:DensityM} to apply.
	\item Third, we estimate the parameters $\muMod$ and $\sigMod$ necessary to shift and scale $\smash{\widehat{Q}}$ in accordance with Theorem \ref{Thm:DensityM}. To obtain an estimator $\smash{\hat{\muMod}}$, we substitute $\hat{\pi}$ for $\pi$ in Eq.\ \eqref{muMod}. The estimator $\hat{\sigMod}$ depends on the assumption added in Step 1 above. Here, with $A_{ij} \sim \operatorname{Bernoulli}\left(\pi_i \pi_j\right)$, we have
\begin{equation*}
\V A_{ij} = \pi_i \pi_j \left(1 - \pi_i \pi_j\right).
\end{equation*}
Then, $\hat{\sigMod}$ follows directly by substituting $\hat{\pi}$ for $\pi$ in Eq.\ \eqref{sigMod}.
	\item Finally, we compute and interpret the resulting approximate $p$-value. We first define community assignments $g(1), g(2), \ldots, g(n)$ based on a covariate, and calculate $\smash{\widehat{Q}}$ as per Eq.\ \eqref{Def:modularityempirical}. We next estimate $\smash{(\widehat{Q}-\muMod) / \sigMod}$ using $\smash{\hat{\muMod}}$ and $\hat{\sigMod}$. Then, by Theorem \ref{Thm:DensityM}, we compute an approximate one-sided $p$-value as follows:
\begin{align}\label{Eq:pval}
\Pr \left(Z \geq \frac{\widehat{Q}- \hat{\muMod}}{\hat{\sigMod}}\right), \quad Z \sim \No(0,1).
\end{align}
A small $p$-value implies that the observed value of modularity (or any larger value) is unlikely under the null.
\end{enumerate}

Table \ref{Tab: ModularityBinary} shows the results of applying this procedure to four benchmark datasets: a network of books \cite{newman2006modularity} where books are connected if they have frequently been purchased together, categorized by political affiliation (Fig.\;\ref{FigWithinBetweenPolbooks}); a network of jazz bands \cite{gleiser2003community} where bands are connected if they have at least one band member in common, categorized by recording location; a network of political commentary websites (weblogs) \cite{adamic2005political} where weblogs are connected if they refer to each other, categorized by political affiliation; and a network of physicists \cite{newman2001structure} where physicists are connected if they have co-authored a manuscript, categorized by manuscript subject category.

The first conclusion of our benchmark analysis is as follows: when we fit the null model of Definition \ref{Def:DegreebasedModel} to each of these four networks, and then simulate from the fitted model (parametric bootstrap), each simulated network results in (via Eq.$\!$~\eqref{Eq:pval}) a $p$-value with empirical mean near $1/2$ and standard deviation near $1/\smash{\sqrt{12}}$. This empirical result aligns with Theorem \ref{Thm:DensityM}, which predicts the $p$-values to be uniformly distributed with exactly that mean and standard deviation in the limit.

Our second conclusion is that, when using the observed data rather than simulated data under the null, each of the covariates leads (again via Eq.$\!$~\eqref{Eq:pval}) to a very small $p$-value ($< 10^{-6}$; see Table \ref{Tab: ModularityBinary}). This suggests that the data as observed are extremely unlikely under the null.  Furthermore, since the null itself cannot explain any community structure, the conclusion we obtain agrees with the use of these covariates by other researchers as ground truth in community detection settings.

\section{Evaluating communities in a multi-edge email network}

We now illustrate how our methodology can identify covariates that reflect a network's community structure. This analysis goes beyond the four benchmark examples considered above, where we validated our methodology but did not reach any new data-analytic conclusions. Here we evaluate the effects of employee \emph{seniority}, \emph{gender}, and \emph{company department} on community structure in a multi-edge corporate email network (see Fig.\ \ref{Fig:EnronAdjMats}).  Table \ref{Tab: ModularityEnron} summarizes all results, showing that each of these covariates results in a small $p$-value, while covariates based on grouping the \emph{first-} or \emph{last-name initials} of the employees do not.  We will return to
this analysis in more detail below, after describing the data and eliciting a suitable model.

\begin{figure}[t!]
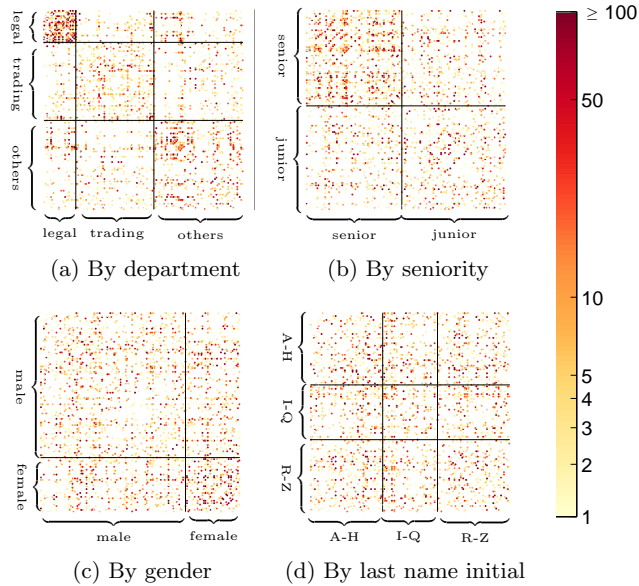

\centering
\hskip-1cm\scalebox{0.95}{\input{log_Enron_department_tex}}	
\scalebox{0.95}{\input{log_Enron_seniority_tex}} \\[-3.9cm]	 	
\scalebox{0.95}{\input{log_Enron_gender_tex}}	
\scalebox{0.95}{\input{log_Enron_initial_second_tex}}				
\caption{\label{Fig:EnronAdjMats} Multi-edges $A_{ij}$ in the Enron corporate email dataset (153 employees, 32261 pairwise email exchanges), grouped according to four different covariate-based community assignments. Shading indicates the number of emails exchanged.}
\end{figure}
\begin{table}[t!]
	\centering\vspace{0.5\baselineskip}%
		\begin{tabular}{l r r r}
			Covariate (no.\ groups) & \multicolumn{1}{c}{$\underline{\widehat{Q} - \hat{\muMod} }$} & \multicolumn{2}{c}{$p$-value} \\
			 & \multicolumn{1}{c}{$\hat{\sigMod}$}  & Eq.\ \eqref{Eq:pval} & Bootstrap \\[0.15cm] \hline \\
			Department (3) &  $6.17$ & $< 10^{-6}$ & $< 10^{-6}$ \\
			Seniority (3) &   $3.14$ & $9 \times 10^{-4}$ & $8 \times 10^{-6}$ \\
			Gender (2) &   $2.36$ & $9 \times 10^{-3}$ & $2 \times 10^{-3}$ \\			
			First name initial (17) &  $0.74$& $2 \times 10^{-1}$ & $2 \times 10^{-1}$	\\
			Last name initial (3)&$ -0.46$& $7\times 10^{-1}$ & $7 \times 10^{-1}$
\end{tabular}
		\caption{\label{Tab: ModularityEnron} Analysis of the data of Fig.\ \ref{Fig:EnronAdjMats}, using modularity derived from multiple covariate-based community assignments.}
\end{table}

This network and its covariates form a substantially richer dataset than those treated above. The data come from the Enron corporation \cite{perry2013point}: as part of a U.S.\ government investigation following allegations of fraud, the email activities of senior employees from 1998--2002 were made public. Following the analysis in \cite{perry2013point}, we exclude all emails that have been sent \emph{en masse} (to more than five recipients), leading to 32261 pairwise email exchanges between 153 employees.  To model this network we will use the full flexibility afforded by Definition \ref{Def:DegreebasedModel}, following the four steps described in the previous section to determine a $p$-value corresponding to each covariate.

Step 1: To construct a suitable model for the observed multi-edges $A_{ij}$, we compare four different distributions satisfying the assumptions of Definition \ref{Def:DegreebasedModel}: $\operatorname{Poisson}(\pi_i \pi_j)$, $\operatorname{NegativeBinomial}(\pi_i \pi_j,r)$ with common shape parameter $r$, and zero-inflated versions of both. Figure \ref{Fig:ModlFit} shows how well these distributions model the multi-edges. Even without zero-inflation, the negative Binomial distribution yields a good fit, particularly in the right tail. A formal model comparison via suitable likelihood ratio tests \cite{cameron1986econometric} confirms this: as Table \ref{Tab:ModlFit} shows, the negative Binomial achieves the best balance between fitting the observed data (residual deviance) and model complexity (degrees of freedom). We thus choose the model
\begin{align} \label{25022016}
A_{ij} \sim \operatorname{NegativeBinomial}(\pi_i \pi_j,r).
\end{align}
Step 2: To verify the assumptions of Definition~\ref{Def:DegreebasedModel} for our data, we first assess Assumptions \ref{nonodeincontrol} and \ref{sparse} exactly as before. Computing quartiles $Q_1$--$Q_3$ of the degrees---$68, 200, 564$---we see that $Q_3/Q_2$ and $Q_1/\sqrt{Q_2}$ are both of order one.  Assumption \ref{notcomplete} ($\max_i \pi_i / \sqrt{n} $ shrinking) can be analogously assessed via $Q_3/(n\sqrt{Q_2})$. Assumptions \ref{over-dispersed} and \ref{skewed} require $\V A_{ij} / \E A_{ij} = 1 + \pi_i \pi_j/r$ and $ \smash{\E \left[\left(A_{ij} - \E A_{ij}\right)^3\right]} / \V A_{ij} = 1 + 2 \pi_i \pi_j / r$ to be bounded. To assess this, we observe that a maximum-likelihood estimate of $r$ \cite{cameron1986econometric} yields $\hat{r} = 0.047$, while the first three quartiles of $\smash{\pij}$ are respectively $0.16, 0.59, 2.1$.

\begin{figure}[t!]
\centering\scalebox{0.93}{\input{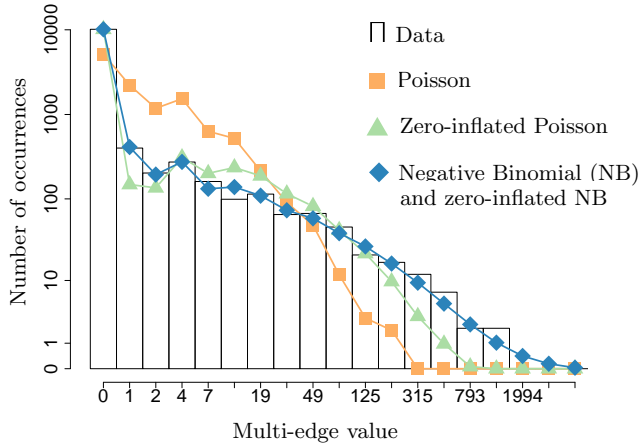}}
			\caption{\label{Fig:ModlFit} Observed versus expected email counts for maximum-likelihood fits of four different models satisfying Definition~\ref{Def:DegreebasedModel}.}
\end{figure}
\begin{table}[t!]
	\centering\vspace{0.5\baselineskip}%
		\begin{tabular}{ l r r r}
			\multicolumn{1}{l}{Model for the} & \multicolumn{1}{c}{Degrees}  &  \multicolumn{1}{c}{Residual} & \multicolumn{1}{c}{Relative} \\
			\multicolumn{1}{l}{multi-edges $A_{ij}$} &\multicolumn{1}{c}{of freedom}    &  \multicolumn{1}{c}{deviance} & \multicolumn{1}{c}{change}  \\[0.15cm] \hline \\
		    Poisson & $153$ & $142031$ & $-39 \%$\\[0.05cm]
			Zero-inflated Poisson & $154$ & $57070$ & $-37 \%$\\[0.05cm]
			Negative Binomial (NB)$\!\!\!\!\!\!\!\!\!\!\!\!$ & $154$ &$12671$ & $-19 \%$ \\[0.05cm]
			Zero-inflated NB & $155$  &  $12671$ & $0 \%$ 	
\end{tabular}
		\caption{\label{Tab:ModlFit}Goodness-of-fit versus model complexity for the models in Fig.\ \ref{Fig:ModlFit} (starting from the $1$-parameter model $\operatorname{Poisson}\left(\lambda\right)$, relative to a saturated negative Binomial model with $r \to \infty$).}
\end{table}

Step 3: To estimate $\muMod$ and $\sigMod$ in Theorem \ref{Thm:DensityM}, we substitute $\hat{\pi}_i$ for $\pi_i$ in Eqs.\ \eqref{muMod} and \eqref{sigMod} exactly as before. Recall, however, that to estimate $s$ we also require an estimate of $\V A_{ij}$ in Eq.\ \eqref{sigMod}. Under the parametrization of Eq.\ \eqref{25022016}, it follows that
\begin{align} \label{VarNegBin}
  \V A_{ij} = \pi_i \pi_j \left(1 + \pi_i \pi_j/r\right).
\end{align}
Thus, $\V A_{ij}$ can be estimated by substituting $\hat{\pi}_i$ for $\pi_i$ and $\hat{r}$ for $r$ in \eqref{VarNegBin}. This yields the required estimators $\smash{\hat{\muMod}}$ and $\smash{\hat{\sigMod}}$.

Step 4: To calculate $p$-values, we must first compute $(\smash{\widehat{Q}} - \smash{\hat{\muMod}} ) / \hat{\sigMod}$ for each covariate. In advance of our analysis, we would expect that employee gender, seniority, and department might reflect aspects of community structure in email interactions. In contrast, we would expect covariates based on the first or last name of each individual to be non-informative. Figure \ref{Fig:EnronAdjMats} illustrates, in decreasing order of $(\smash{\widehat{Q}} - \smash{\hat{\muMod}} ) / \hat{\sigMod}$, the observed structure of our data when grouped by covariate.

Table \ref{Tab: ModularityEnron} reports two approximate $p$-values per covariate, in contrast to the previous section.  The first of these derives (via Eq.\ \eqref{Eq:pval}) from Theorem \ref{Thm:DensityM}, which shows the limiting distribution of $(\smash{\widehat{Q}} - \smash{\hat{\muMod}} ) / \hat{\sigMod}$ under the assumed model to be a standard Normal.  The second is based on $10^7$ replicates of the parametric bootstrap, whereby we fit a negative Binomial model to the data and then simulate from the fitted values to obtain an empirical finite-sample distribution.  Table~\ref{Tab: ModularityEnron} indicates that our asymptotic theory is somewhat conservative in this setting, leading as it does here to larger $p$-values than the bootstrap.

Finally, considering these $p$-values in more detail, we see from Table \ref{Tab: ModularityEnron} that for the covariates of department, gender, and seniority, all $p$-values fall below $1 \%$ (leading to a corrected total of 5\% after adjusting for multiple comparisons).  In contrast, we obtain large $p$-values for first- and last-name covariates. This matches our expectations that department, gender, and seniority are likely to have an impact on email interactions, while there is no obvious reason why this should hold for name-related covariates.

\section{Discussion}

Networks have richer and more varied structure than can be described by a single ``best'' community assignment. To reflect this, we have introduced an approach which exploits the structural information captured by covariates, each of which may describe different aspects of community structure in the data. In contrast to community \emph{detection} per se, this approach allows us to assess the significance of a given, interpretable community assignment with respect to the observed network structure. As described in the data analysis examples above, our method leads to the identification of structurally significant community assignments, ultimately yielding a better understanding of the network under study.

In technical terms, we have established a central limit theorem for modularity under a nonparametric null model, yielding $p$-values to assess the significance of observed community structure.  The model we introduce shows explicitly how modularity measures variability in the data that cannot be explained solely by node-specific propensities for connection.  What is more, modularity has more explanatory power than a classical (chi-squared) goodness-of-fit statistic: by aggregating the estimated \emph{signed} residuals $A_{ij}-d_i d_j/\sum_l d_l$ within every network community, it measures the global tendency of a given community assignment to explain the observed network structure.

To advance the state of the art in network analysis, we as a research community must use this explanatory power to understand the effects of multiple observed communities on network structure. Our work here represents a first step in this direction: we use the explanatory power of modularity to assess the significance of observed community structure relative to a null model.  This opens the door to more advanced uses of multiple observed community assignments within formal statistical modeling frameworks.  This is an important next step, since we see clear evidence here that multiple groupings may explain different aspects of a network's community structure.

\section*{Acknowledgments}
The authors thank Dr.\ Leon Danon for sharing the data on jazz musicians from \cite{gleiser2003community} and Mar\'{\i}a Dolores Alfaro Cuevas for producing Fig.\ \ref{FigWithinBetweenPolbooks}. This work was supported in part by the US Army Research Office under Multidisciplinary University Research Initiative Award 58153-MA-MUR; by the US Office of Naval Research under Award N00014-14-1-0819; by the UK Engineering and Physical Sciences Research Council under Mathematical Sciences Established Career Fellowship EP/K005413/1; by the UK Royal Society under a Wolfson Research Merit Award; and by Marie Curie FP7 Integration Grant PCIG12-GA-2012-334622 within the 7th European Union Framework Program.

\appendix

\section{Notation and assumptions}
For the following proofs we will always consider an undirected random graph on $n$ nodes with no self-loops. We model the edges $A_{ij}$ as independent random variables with expectation
						\begin{align*}
								\E A_{ij}= \pi_i \pi_j, \quad A_{ij}\geq 0; \quad 1 \leq i< j\leq n
						\end{align*}
where $\bm{\pi}=\left(\pi_1, \ldots, \pi_n\right) \in \R^n_{>0}$.  We will denote the degree of node $i$ as $d_i$; i.e., $d_i=\sum_{j \neq i} A_{ij}$. The remaining five assumptions of Definition \ref{Def:modularityempirical} of the degree-based model are not all needed at all times and will therefore be mentioned explicitly. For convenience we restate the assumptions below,
all of which reference a sequence of networks where $n \rightarrow \infty$.

\begin{enumerate}
	\item No node dominates the network; i.e, $n \max_i \pi_i/\left\|\bm{\pi}\right\|_1 = \mathcal{O}(1)$;
	\item The network is not too sparse; i.e., $\min_i \pi_i = \omega\left(1/ \sqrt{n}\right)$;
	\item The expectation of each edge does not diverge too quickly; i.e., $\max_i \pi_i = o\left( \sqrt{n}\right)$;
	\item The ratio of variance to expectation of each edge is controlled; i.e., $\forall i,j: \V A_{ij}/ \E A_{ij} = \Theta(1)$; and 	
\item The skewness of each edge $A_{ij}$ is controlled; i.e., $\forall i,j:$\\ $ \E \left[\left(A_{ij} - \E A_{ij}\right)^3\right]/ \V \left(A_{ij}\right)= \mathcal{O}(1)$.
\end{enumerate}

We use bold letters to denote vectors.

\section{Proof of Theorem 2}
We first show a univariate central limit theorem for the scalar estimator $\hat{\pi}_i=d_i/\sqrt{\| \bm{d} \|_1}$. We then extend this result to the multivariate case, applying the Cram\'er--Wold theorem.\\

Preliminaries:
Since the edges $A_{ij}, i<j$ are independent, it follows as shown in \cite{Olhede2012} that for finite $n$
\begin{align}
\E d_i & = \pi_i \left( \| \bm{\pi} \|_1 - \pi_i \right) , \label{Exp_degree}
\\ \V d_i & = \sum_{i \neq j} \V A_{ij} , \label{Var_degree}
\\ \operatorname{cov} \left( d_i , d_j \right) & = \begin{cases}\V  A_{ij}, &\quad i \neq j\\
\V d_i, &\quad i = j \end{cases}\label{Cov_degrees}\\
	 \E \left\|\bm{d}\right\|_1 & = \| \bm{\pi} \|_1^2 - \| \bm{\pi} \|_2^2, \label{Exp_1norm_degree}
\\ \V \left\|\bm{d}\right\|_1 & = 2 \sum_{i=1}^n \V d_i. \label{Var_1norm_degree}
\end{align}

\begin{appxtheorem}[Central limit theorem for $\hat{\pi}_i$] \label{App:GeneralCLTPoisson}
Consider Assumptions~\ref{nonodeincontrol}--\ref{skewed}. Define $\hat{\pi}_i=d_i/\sqrt{\| \bm{d} \|_1}$ as an estimator of $\pi_i$. Then as $n \rightarrow \infty$,
\begin{align*}
								\frac{\hat{\pi}_i-\pi_i}{\sqrt{\V d_i/ \E \left\|\bm{d}\right\|_1}} \stackrel{d}{\rightarrow} \No(0,1).
\end{align*}								
Furthermore, $\sqrt{ \V d_i/\E \left\|\bm{d}\right\|_1} = \mathcal{O}\left(1/ \sqrt{n}\right)$, and can be consistently estimated  using a plug-in estimator for $A_{ij} \sim \operatorname{Bernoulli}\left(\pi_i \pi_j\right)$ and $A_{ij} \sim \operatorname{Poisson}\left(\pi_i \pi_j\right)$.				
\end{appxtheorem}

\begin{proof}
			The proof is a generalization of the proof of Theorem 3.2 in \cite{Olhede2012}, which assumes Bernoulli edges and a power law degree distribution. We write%
\begin{align}\label{App:splitForCLT}
	\frac{\hat{\pi}_i-\pi_i}{\sqrt{ \V d_i/ \E \left\|\bm{d}\right\|_1}} 				
											& = \left[ \underbrace{\frac{d_i - \E d_i}{\sqrt{\V d_i }}}_{T_1} + \underbrace{\frac{\E d_i - \pi_i \sqrt{\| \bm{d} \|_1}}{\sqrt{\V d_i }}}_{T_2} \right] \underbrace{\sqrt{\frac{\E \left\|\bm{d}\right\|_1 }{\| \bm{d} \|_1}} }_{T_3}.
\end{align}
			To deduce the required result, we show that $T_1$ converges in distribution to a $\No(0,1)$ random variable and $T_2$ and $T_3$ go in probability to 0 and 1, respectively. Slutsky's theorem enables us to
			combine the results and to obtain the claimed convergence in distribution. \\
			
		 \underline{Term $T_1$}:
				Each degree $d_i=\sum_{j \neq i} A_{ij}$ is a sum of independent random variables. From Assumption~\ref{sparse} $\left( \Rightarrow \E d_i \rightarrow \infty\right)$ and Assumption~\ref{over-dispersed} ($\E A_{ij} = \Theta\left(\V A_{ij}\right)$), it follows that $\V d_i \rightarrow \infty$. Since in addition, the skewness of each edge $A_{ij}$ is asymptotically  bounded (Assumption~\ref{skewed}), the Lyapunov condition for exponent $1$ is satisfied; i.e.,
				\begin{align*}
					\frac{\sum_{j \neq i} \E \left[\left(A_{ij} - \E A_{ij}\right)^3\right]}{\left[\sum_{j \neq i} \V A_{ij} \right]^{3/2}} \rightarrow 0.
				\end{align*}
				
				 Hence, the Lindeberg--Feller Central Limit Theorem allows us to conclude that $T_1 \stackrel{d}{\rightarrow} \No(0,1)$.\\

	\underline{Term $T_2$}:
		We write
				\begin{align}
								T_2 &= \frac{\E d_i - \pi_i \sqrt{\| \bm{d} \|_1}}{\sqrt{\V d_i }} \notag
								    \\&= \underbrace{\frac{\E d_i - \pi_i \sqrt{\E \|\bm{d}\|_1 }}{\sqrt{\V d_i }}}_{a)} -\underbrace{\frac{ \pi_i \sqrt{\| \bm{d} \|_1}  - \pi_i \sqrt{\E \|\bm{d}\|_1 } }{\sqrt{\V d_i }}}_{b)}. \label{09092015}
				\end{align}
Term $T_2$ converges in probability to 0 since both a) the first ratio converges to 0 and b) the second ratio converges to 0 in probability.

				a) This convergence is driven by the fact that $\E d_i - \pi_i \sqrt{\E \|\bm{d}\|_1 } = \mathcal{O}(1)$  (see Eqs.\ \eqref{Exp_degree} and \eqref{Exp_1norm_degree}) while $\V d_i \rightarrow \infty$. More precisely,
				\begin{align}
							 \frac{\E d_i - \pi_i \sqrt{\E \|\bm{d}\|_1 }}{\sqrt{\V d_i }}
							 				&= \frac{\pi_i \left\|\bm{\pi}\right\|_1 \left[ 1 - \sqrt{1 - \left\|\bm{\pi}\right\|_2^2 / \left\|\bm{\pi}\right\|_1^2}\right]- \pi_i^2 }{\sqrt{\V d_i }}. \label{0809}
				\end{align}
Considering $\bm{\tilde{\pi}}=\bm{\pi}/ \max_{j} \pi_j$, we can conclude from $\| \bm{\tilde{\pi}} \|_2^2 \leq \| \bm{\tilde{\pi}} \|_1$ that
\begin{align} \label{08102015}
	\frac{\| \bm{\pi} \|_2^2 }{ \| \bm{\pi} \|_1^2}
	& = \frac{ \left(\max_{j} \pi_j\right)^2  \| \bm{\tilde{\pi}} \|_2^2 }{ \left(\max_{j} \pi_j\right)^2 \| \bm{\tilde{\pi}} \|_1^2}
 \leq \frac{1 }{ \| \bm{\tilde{\pi}} \|_1} =  \frac{\max_{j} \pi_j}{ \| \bm{\pi} \|_1}.
\end{align}
Assumption~\ref{nonodeincontrol} implies that
$\max_{j} \pi_j/ \| \bm{\pi} \|_1 = \mathcal{O}\left(1/ n\right)$, and thus we conclude
\begin{align}\label{twonormoveroneneormsqu}
	\frac{\| \bm{\pi} \|_2^2 }{ \| \bm{\pi} \|_1^2}
	& = \mathcal{O}\left(\frac{1}{ n}\right).
\end{align}
 This allows us to apply a convergent Taylor expansion of $\sqrt{1 - x}$ at 0  in Eq.~\eqref{0809}:
				\begin{align}
					\label{08092015}
						& \frac{\E d_i - \pi_i \sqrt{\E \|\bm{d}\|_1 }}{\sqrt{\V d_i }}
\\ & \quad = \frac{\pi_i \left\|\bm{\pi}\right\|_1 \left[ 1 - \left(1 - \left\|\bm{\pi}\right\|_2^2 / 2 \left\|\bm{\pi}\right\|_1^2 + o\left(\left\|\bm{\pi}\right\|_2^2 / \left\|\bm{\pi}\right\|_1^2\right)\right)\right]- \pi_i^2 }{\sqrt{\V d_i }}	\notag
\\ &\quad  = \frac{\pi_i \left[ \left\|\bm{\pi}\right\|_2^2 / 2 \left\|\bm{\pi}\right\|_1 + o\left(\left\|\bm{\pi}\right\|_2^2 / \left\|\bm{\pi}\right\|_1\right)\right]- \pi_i^2 }{\sqrt{\V d_i }}	\notag	
\\ &\quad  \leq \frac{\pi_i \left[ \max_{j} \pi_j / 2+ o\left(\max_{j} \pi_j\right)\right]- \pi_i^2 }{\sqrt{\V d_i }}	\notag	
\qquad \text{(see Eq.\ \eqref{08102015})}
\\ &\quad  = \Theta\left(\frac{\pi_i  \left(\max_{j} \pi_j  - \pi_i\right) }{\sqrt{\E d_i }}	\notag\right)	
\qquad \text{(Assumption~\ref{over-dispersed})}
\\ &\quad  = \Theta\left(\frac{\sqrt{\pi_i}  \left(\max_{j} \pi_j  - \pi_i\right) }{\sqrt{\left\|\pi\right\|_1 -\pi_i}}	\notag\right)	
\\ &\quad  = \mathcal{O}\left(\frac{  \max_{j} \pi_j  - \pi_i }{\sqrt{n }}	\notag\right). \notag		
\qquad \text{(Assumption~\ref{nonodeincontrol})}
				\end{align}				
				Since $\pi_j = o\left(\sqrt{n}\right)$ for all $j$ (Assumption~\ref{notcomplete}), it follows that the left-hand side of Eq.\ \eqref{08092015} converges to 0 in $n$.\\
			  b) We show below that the second ratio $\left( \pi_i \sqrt{\| \bm{d} \|_1}  - \pi_i \sqrt{\E \|\bm{d}\|_1 }\right) / \sqrt{\V d_i }$  in Eq.\ \eqref{09092015} converges in probability to 0; this follows since $\pi_i / \sqrt{\V d_i} \rightarrow 0$ under Assumptions~\ref{nonodeincontrol} and \ref{over-dispersed} (see c) below) and $ \sqrt{\| \bm{d} \|_1}  - \sqrt{\E \|\bm{d}\|_1 } =\mathcal{O}_P(1)$ (see Lemma~\ref{09042014} below).

c) From Assumption~\ref{over-dispersed} it follows that		
\begin{align} \notag
		\frac{\pi_i}{\sqrt{\V d_i }} &= \Theta\left(\frac{\pi_i}{\sqrt{\E d_i }}\right) = \Theta\left(\sqrt{\frac{\pi_i}{\left\|\bm{\pi}\right\|_1-\pi_i}}\right)
		\\ &= \mathcal{O}\left(1/ \sqrt{n}\right). \qquad \text{(Assumption~\ref{nonodeincontrol})} \label{10112015b}
\end{align}		
\begin{appxlemma} 	\label{09042014}
	Consider Assumptions~\ref{sparse}--\ref{skewed}. Then, $\sqrt{\| \bm{d} \|_1}  - \sqrt{\E \|\bm{d}\|_1 }=\mathcal{O}_P(1)$.
\end{appxlemma}
				
				\begin{proof}
						Observe that the square root function has one continuous derivative at 1. A Taylor expansion in probability of $\sqrt{\|\bm{d}\|_1 / \E \|\bm{d}\|_1 }$ about 1 requires in addition \cite[p.~201]{brockwell1991time} that
					
\begin{enumerate}
	\item [I.]$\exists a \in \mathds{R}: \|\bm{d}\|_1 / \E \|\bm{d}\|_1  = a + \mathcal{O}_P(r_n)$; with
	\item [II.]$r_n \rightarrow 0$ as $n \rightarrow \infty$.
\end{enumerate}						
						
I.~It follows from Chebyshev's inequality that
\begin{align} \label{normd1concentrates}
\|\bm{d}\|_1 / \E \|\bm{d}\|_1 	= 1 + \mathcal{O}_P \left(\sqrt{\V \|\bm{d}\|_1 } / \E \|\bm{d}\|_1 \right).
\end{align}
						
II. As a consequence of I., $r_n = \sqrt{\V \|\bm{d}\|_1 } / \E \|\bm{d}\|_1$. From Eq.\ \eqref{Exp_1norm_degree} and Assumption \ref{sparse} $\left(\Rightarrow \E d_i \rightarrow \infty\right)$ it follows that $\E \left\|\bm{d}\right\|_1 \rightarrow \infty$. Since $A_{ij}$ are independent for $i<j$, and since we assume $\V A_{ij} / \E A_{ij} = \Theta(1)$ (Assumption~\ref{over-dispersed}), it holds that
\begin{align}
 \frac{\V \|\bm{d}\|_1 }{ \E \|\bm{d}\|_1} &= \frac{\V\left( 2 \sum_{j=1}^n \sum_{i<j} A_{ij}\right)}{ \E \left( 2 \sum_{j=1}^n \sum_{i<j} A_{ij}\right)} \notag
 \\ &= \frac{4 \sum_{j=1}^n \sum_{i<j} \V\left(  A_{ij}\right)}{ 2 \sum_{j=1}^n \sum_{i<j}\E \left(  A_{ij}\right)}  \notag
 \\ & =\Theta(1). \label{ratio of normsorder1}
\end{align}
It follows that the ratio $\sqrt{\V \|\bm{d}\|_1 } / \E \|\bm{d}\|_1 \rightarrow 0$.

We now can apply a convergent Taylor expansion in probability:
\begin{align} \notag
								\sqrt{\frac{\|\bm{d}\|_1}{\E \|\bm{d}\|_1 }} &= 1+ \frac{1}{2} \left(\frac{\|\bm{d}\|_1}{\E \|\bm{d}\|_1 } -1\right) + o_P \left(\frac{\sqrt{\V \|\bm{d}\|_1 }}{\E \|\bm{d}\|_1 }\right)
\\								\Leftrightarrow \quad \sqrt{\|\bm{d}\|_1} - \sqrt{\E \|\bm{d}\|_1 } &=    \frac{\sqrt{\V \|\bm{d}\|_1}}{\sqrt{\E \|\bm{d}\|_1 }} \left[ \frac{1}{2} \left(\frac{\|\bm{d}\|_1-\E \|\bm{d}\|_1  }{\sqrt{ \V \|\bm{d}\|_1}} \right) + o_P \left(1\right)\right]. \label{10112015}
\end{align}
				Since the term $\|\bm{d}\|_1/2= \sum_{j =1}^n \sum_{i<j} A_{ij}$ is a sum of independent random variables, we apply the Lindeberg--Feller central limit theorem analogously to Term $T_1$: From Assumptions~\ref{sparse}--\ref{skewed}, it follows that
\begin{align*}		
	\frac{\|\bm{d}\|_1 - \E \|\bm{d}\|_1}{ \sqrt{\V \|\bm{d}\|_1 }} \stackrel{d}{\rightarrow} \No(0,1).
\end{align*}
						Since $\V \|\bm{d}\|_1 / \E \|\bm{d}\|_1 =\Theta(1)$ by Eq.\ \eqref{ratio of normsorder1}, we conclude from Eq.\ \eqref{10112015} the result of Lemma~\ref{09042014}; i.e., $\sqrt{\| \bm{d} \|_1}  - \sqrt{\E \|\bm{d}\|_1 }=\mathcal{O}_P(1)$.
		\end{proof}
		
 As a consequence of Lemma~\ref{09042014}, we now know	that the numerator of term~b) in Eq.\ \eqref{09092015} is bounded in probability. Since we show in Eq.\ \eqref{10112015b} that $\pi_i / \sqrt{\V d_i} = \mathcal{O}\left(1/ \sqrt{n}\right)$, it follows that
 \begin{align*}
 		b) = \frac{ \pi_i \sqrt{\| \bm{d} \|_1}  - \pi_i \sqrt{\E \|\bm{d}\|_1 } }{\sqrt{\V d_i }} \stackrel{P}{\rightarrow} 0.
 \end{align*} In turn, this completes the proof of the convergence of Term 2 (see Eq.\ \eqref{09092015}); i.e.,
\begin{align} \label{proofTermT2}
								T_2 = \underbrace{\frac{\E d_i - \pi_i \sqrt{\E \|\bm{d}\|_1 }}{\sqrt{\V d_i }}}_{a)} -\underbrace{\frac{ \pi_i \sqrt{\| \bm{d} \|_1}  - \pi_i \sqrt{\E \|\bm{d}\|_1 } }{\sqrt{\V d_i }}}_{b)} \stackrel{P}{\rightarrow} 0.
\end{align}			

\underline{Term $T_3$}\\
Combining Eqs.\ \eqref{normd1concentrates} and~\eqref{ratio of normsorder1}, we know that
\begin{align*}
	\frac{\| \bm{d} \|_1}{\E \left\|\bm{d}\right\|_1} &
	  = 1 + \mathcal{O}_P\left(\frac{1}{ \sqrt{\E \left\|\bm{d}\right\|_1 }}\right).
\end{align*}
This converges in probability to 1 because of Assumption~\ref{sparse} $\left(\Rightarrow  \E \left\|\bm{d}\right\|_1 \rightarrow \infty\right)$.

Applying the continuous mapping theorem, leads to   $\sqrt{\|\bm{d}\|_1 / \E \|\bm{d}\|_1} \stackrel{P}{\rightarrow} 1$. The inverse of a random variable which converges in probability to a constant $c$, must in turn converge to $1/c$, as long as $c \neq 0$ \cite[Theorem 2.1.3]{Lehmann199812}. Thus,
\begin{align} \label{proofTermT3}
	T_3= \sqrt{\frac{\E \left\|\bm{d}\right\|_1}{ \|\bm{d}\|_1}} \stackrel{P}{\rightarrow} 1.
\end{align}

Slutsky's Theorem enables us to combine the results on the convergence of terms $T_1$--$T_3$ to obtain that
\begin{align*}
	\frac{\hat{\pi}_i -\pi_i}{ \sqrt{\V d_i/ \E \|\bm{d}\|_1}} \rightarrow \No(0,1).
\end{align*}
To complete the proof of Theorem~\ref{App:GeneralCLTPoisson}, it remains to show that $ \V d_i /\E \|\bm{d}\|_1= \mathcal{O}(1/n)$, and that it can be consistently estimated  using a plug-in estimator for $A_{ij} \sim \operatorname{Bernoulli}\left(\pi_i \pi_j\right)$ and $A_{ij} \sim \operatorname{Poisson}\left(\pi_i \pi_j\right)$.

Since $\V A_{ij}/ \E A_{ij} = \Theta(1)$ (Assumption~\ref{over-dispersed}), we know that
\begin{align*}
	\sqrt{\frac{n \V d_i}{\E \|\bm{d}\|_1}} 	
&= \sqrt{\frac{n \; \Theta \left(\E d_i\right)}{\E \|\bm{d}\|_1}}
\\	&	= \sqrt{\frac{n \; \Theta\left(\pi_i \left( \| \bm{\pi} \|_1 - \pi_i \right)\right)}{\| \bm{\pi} \|_1^2 - \| \bm{\pi} \|_2^2}}
\\	&	= \sqrt{\frac{n \pi_i}{\| \bm{\pi} \|_1} \Theta \left(\frac{ 1 - \pi_i/\| \bm{\pi} \|_1 }{1 - \| \bm{\pi} \|_2^2/\| \bm{\pi} \|_1^2}\right)} .
\end{align*}
We know that $n \pi_i / \| \bm{\pi} \|_1 = \mathcal{O}(1)$ (Assumption~\ref{nonodeincontrol}) and we have seen in Eq.\ \eqref{twonormoveroneneormsqu} that $\| \bm{\pi} \|_2^2/\| \bm{\pi} \|_1^2 = \mathcal{O}(1/n)$ (also from Assumption~\ref{nonodeincontrol}). Hence, $\sqrt{\V d_i/\E \|\bm{d}\|_1}$ $= \mathcal{O}(1/ \sqrt{n})$.

We defer the proof of consistency of the plug-in estimator of $\V d_i /\E \|\bm{d}\|_1$ for $A_{ij} \sim \operatorname{Bernoulli}\left(\pi_i \pi_j\right)$ and $A_{ij} \sim \operatorname{Poisson}\left(\pi_i \pi_j\right)$ to Theorem~\ref{ThmConsistency}, where we show a more general statement.
\end{proof}

Having shown a univariate central limit theorem for each $\hat{\pi}_i$, we are now ready to extend this result to the multivariate case. The Corollary below is identical to Theorem~2 in the main text.

\begin{appxcorollary} [Multivariate central limit theorem for $\hat{\pi}_i$s] \label{multiCLTp}
Consider Assumptions~\ref{nonodeincontrol}--\ref{skewed}. Estimate $\pi_i$ by $\hat{\pi}_i=d_i / \sqrt{\left\|\bm{d}\right\|_1}$ for all $i$ and fix a set of $r$ positive integers as indices, with $r$ finite. Relabeling the indices from 1 to $r$ without loss of generality,
	 \begin{align*}
	\sqrt{\E \left\|\bm{d}\right\|_1} \left(\frac{\hat{\pi}_1-\pi_1}{\sqrt{\V d_1}}, \ldots,  \frac{\hat{\pi}_r-\pi_r}{\sqrt{\V d_r}}\right)'
	 		 \stackrel{d}{\rightarrow}\No(\bm{0}, \bm{I_{r}}).
	 \end{align*}
Furthermore for all $i$, $\sqrt{ \V d_i/\E \left\|\bm{d}\right\|_1} = \mathcal{O}\left(1/ \sqrt{n}\right)$, and can be consistently estimated for $A_{ij} \sim \operatorname{Bernoulli}\left(\pi_i \pi_j\right)$ and $A_{ij} \sim \operatorname{Poisson}\left(\pi_i \pi_j\right)$ using a plug-in estimator.
\end{appxcorollary}

\begin{proof}
		This proof is the multidimensional equivalent of the proof of Theorem~\ref{App:GeneralCLTPoisson}. It is analogously driven by the fact that the vector
		\begin{align*}
		\bm{m_1} = \left(\frac{d_1 - \E d_1}{\sqrt{\V d_1 }}, \ldots,  \frac{d_r - \E d_r}{\sqrt{\V d_r}}\right)'
		\end{align*}
		can be reduced to a sum of independent but not identically distributed random vectors. These in turn converge in distribution to a multivariate standard $\No$ random vector; as we now show. In direct analogy to the univariate case of Eq.\ \eqref{App:splitForCLT},
		\begin{align}
		& \sqrt{\E \left\|\bm{d}\right\|_1} \left(\frac{\hat{\pi}_1-\pi_1}{\sqrt{\V d_1}}, \ldots,  \frac{\hat{\pi}_r-\pi_r}{\sqrt{\V d_r}}\right)' \notag
\\		 &= \sqrt{\E \left\|\bm{d}\right\|_1} \left(\frac{1}{\sqrt{\V d_1}}\left( \frac{d_1}{\sqrt{\left\|\bm{d}\right\|_1}}-\pi_1\right), \ldots,  \frac{1}{\sqrt{\V d_r}}\left( \frac{d_r}{\sqrt{\left\|\bm{d}\right\|_1}}-\pi_r\right)\right)' \notag
		\\
		\begin{split} \label{12102015a}
		& =\underbrace{\sqrt{\frac{ \E \left\|\bm{d}\right\|_1}{\left\|\bm{d}\right\|_1}}}_{m_3} \cdot
		 \Biggl(\underbrace{
							 \left(\frac{d_1 - \E d_1}{\sqrt{\V d_1 }}, \ldots,
							 			 \frac{d_r - \E d_r}{\sqrt{\V d_r}}\right)'}_{\bm{m_1}} \\
		& \qquad							 			
							  +\underbrace{\left( \frac{\E d_1 - \pi_1 \sqrt{\left\|\bm{d}\right\|_1}}{\sqrt{\V d_1 }}
							  , \ldots,
							   \frac{\E d_r - \pi_r\sqrt{\left\|\bm{d}\right\|_1}}{\sqrt{\V d_r}}
							 \right)'}_{\bm{m_2}}\Biggr).	
		\end{split}					 			
		\end{align}
		Each component of the vector $\bm{m_2}$ converges in probability to 0 (see Eq.\ \eqref{proofTermT2} in the proof of Theorem~\ref{App:GeneralCLTPoisson}). It follows that the vector $\bm{m_2} \stackrel{P}{\rightarrow} \bm{0}$. In addition, the scalar $m_3$ converges in probability to 1 (see Eq.\ \eqref{proofTermT3} in the proof of Theorem~\ref{App:GeneralCLTPoisson}).
		
		We now prove that $\bm{m_1}\stackrel{d}{\rightarrow} \No\left(\bm{0}, \bm{I_r}\right)$. In order to apply a multivariate central limit theorem, we rearrange $\bm{m_1}$ such that we extract a sum of independent random vectors ($\bm{m_{12}}$):
			\begin{align}
	\notag			\bm{m_1}
		&=	 \left(\frac{d_1 - \E d_1}{\sqrt{\V d_1 }}, \ldots,
							 			 \frac{d_r - \E d_r}{\sqrt{\V d_r}}\right)'
 		\\ \begin{split}	\label{12102015b}	  &=\underbrace{\diag \left(\frac{\sqrt{\V(\sum_{l= r+1}^n A_{l1} )}}{\sqrt{\V d_1 }}, \ldots,
								          \frac{\sqrt{\V(\sum_{l=r+1}^n A_{lr} )}}{\sqrt{\V d_r }}\right) }_{\bm{D_{11}}}\\
			  	& \quad \cdot \underbrace{
			  	 \left(
						 \frac{\sum_{l =r+1}^n \left(A_{l1} - \E A_{l1}\right)}{\sqrt{\V(\sum_{l=r+1}^n A_{l1} ) }}, \ldots,
				 		 \frac{\sum_{l =r+1}^n \left(A_{lr} -\E A_{lr}\right)}{\sqrt{\V(\sum_{l=r+1}^n A_{lr} ) }}
			  	\right)'
			  	 }_{\bm{m_{12}}}
\\		  	& \quad
+\underbrace{
						\left(
						 \frac{\sum_{l =1}^r \left(A_{l1} - \E A_{l1}\right)}{\sqrt{\V d_1 }}, \ldots,
				 		 \frac{\sum_{l=1}^r  \left(A_{lr} -\E A_{lr}\right)}{\sqrt{\V d_r }}
			  	\right)'
			  	}_{\bm{m_{13}}}.
			 \end{split} 	
		\end{align}		
We will show three things: that the matrix $\bm{D_{11}}$ converges to the identity matrix $\bm{I_{r}}$; that $\bm{m_{12}} \stackrel{d}{\rightarrow}\No(\bm{0}, \bm{I_{r}})$; and that the term $\bm{m_{13}} \stackrel{P}{\rightarrow}\bm{0}$.

For the term $\bm{D_{11}}$, it holds for all $i$ that
\begin{align*}
	\sqrt{\frac{\V(\sum_{l= r+1}^n A_{li} )}{\V d_i }}
	& = \sqrt{1 - \frac{\V(\sum_{l= 1}^r A_{li} )}{\V \left(\sum_{l=1}^n A_{li}\right)}}.
\end{align*}
Furthermore, from Assumption~\ref{over-dispersed} $\left(\V A_{ij}=\Theta\left(\E A_{ij}\right)\right)$ we conclude  for all $i$  that
\begin{align*}
  \frac{\V(\sum_{l= 1}^r A_{li} )}{\V \left(\sum_{l=1}^n A_{li}\right)}
&  = \Theta \left(\frac{\sum_{l= 1}^r \E A_{li} }{\sum_{l=1}^n \E A_{li}}\right)
\\& = \Theta \left(\frac{\pi_i \sum_{l= 1}^r \pi_l }{\pi_i \left\|\bm{\pi}\right\|_1 }\right)
\\& = \Theta \left(\sum_{l= 1}^r \frac{ \pi_l }{\left\|\bm{\pi}\right\|_1 }\right).
\end{align*}
It follows further from Assumption~\ref{nonodeincontrol} that
\begin{align}
  \frac{\V(\sum_{l= 1}^r A_{li} )}{\V \left(\sum_{l=1}^n A_{li}\right)}
& = \mathcal{O} \left(\frac{r}{n}\right)
 \rightarrow 0. \label{07092015}
\end{align}

In turn, $\sqrt{\V(\sum_{l= r+1}^n A_{li} )}/\sqrt{\V d_i } \rightarrow 1$ for all $i$. Hence, the diagonal matrix $\bm{D_{11}}$ converges to the identity matrix $\bm{I_r}$ in the operator norm.

 The term  $\bm{m_{12}} \stackrel{d}{\rightarrow}\No(\bm{0}, \bm{I_{r}})$, as we will now show by applying  the Cram\'er--Wold theorem.
 The term $\bm{m_{12}}$ is a random vector depending on $n$, where each component is a sum of independent random variables. We will show now that, as a consequence, each component converges marginally in distribution to a $\No(0,1)$ random variable (by the same argument as in Theorem~\ref{App:GeneralCLTPoisson} for Term $T_1$). From Assumption~\ref{sparse} $\left(\Rightarrow  \E d_i \rightarrow \infty\right)$ and Assumption~\ref{over-dispersed} ($\V A_{ij}/ \E A_{ij} = \Theta(1)$), it follows that $\V d_i \rightarrow \infty$. Since in addition we assume the skewness of each edge $A_{ij}$ to be bounded asymptotically (Assumption~\ref{skewed}), the Lyapunov condition (for $\delta = 1$) is satisfied for each component. Hence, the Lindeberg--Feller central limit theorem lets us conclude that each component converges marginally in distribution to a $\No\left(0,1\right)$ random variable \cite[p.~362]{billingsley1995}.

 Furthermore, the components of $\bm{m_{12}}$ are independent. It follows that for each $\left(c_1, \ldots, c_r\right) \in \mathds{R}^r$ and $Y_u \stackrel{\text{iid}}{\sim} \No\left(0,1\right)$ for $u=1, \ldots, r$, it holds  that
		\begin{align*}
					\sum_{u=1}^r c_u  \frac{\sum_{l=r+1}^n \left(A_{lu} -\E A_{lu}\right)}{\sqrt{\V(\sum_{l=r+1}^n A_{lu} ) }} \quad \stackrel{d}{\rightarrow} \quad \sum_{u=1}^r c_u Y_u.
		\end{align*}		
	Applying the Cram\'er--Wold theorem, we conclude that $\bm{m_{12}} \stackrel{d}{\rightarrow}\No(\bm{0}, \bm{I_{r}})$.

Finally, term $\bm{m_{13}} \stackrel{P}{\rightarrow} \bm{0}$, since by Chebyshev's inequality
\begin{align*}
\frac{\sum_{l =1}^r \left(A_{li} - \E A_{li}\right)}{\sqrt{\V d_i }} = \mathcal{O}_P \left(\sqrt{\frac{  \V\left(\sum_{l =1}^r A_{li}\right)}{\V d_i }} \right),
\end{align*}
which in turn goes to 0 for all $i$, as seen in Eq.\ \eqref{07092015}.

 By Slutsky's theorem, we can combine the results on the convergence of $\bm{D_{11}}$, $\bm{m_{12}}$, and $\bm{m_{13}}$ to conclude (see Eq.\ \eqref{12102015b}) that
 \begin{align*}
 \bm{m_1}= \bm{D_{11}} \; \bm{m_{12}} + \bm{m_{13}}\stackrel{d}{\rightarrow} \No\left(\bm{0}, \bm{I_r}\right).
 \end{align*}
  In turn, we deduce the required result (see Eq.\ \eqref{12102015a}) that
  \begin{align*}
		\sqrt{\E \left\|\bm{d}\right\|_1} \left(\frac{\hat{\pi}_1-\pi_1}{\sqrt{\V d_1}}, \ldots,  \frac{\hat{\pi}_r-\pi_r}{\sqrt{\V d_r}}\right)' = m_3  \bm{m_1} + \bm{m_2} \stackrel{d}{\rightarrow} \No\left(\bm{0}, \bm{I_r}\right).
	\end{align*}
	
	To complete the proof we need to show consistency of the plug-in estimator of $\V d_i /\E \|\bm{d}\|_1$ for $A_{ij} \sim \operatorname{Bernoulli}\left(\pi_i \pi_j\right)$ and $A_{ij} \sim \operatorname{Poisson}\left(\pi_i \pi_j\right)$. We defer this to Theorem~\ref{ThmConsistency}, where we show a more general statement.
\end{proof}

\section{Proof of the Corollary of Theorem~2}
\label{Sectionstrength}

As a reminder to the reader, the Corollary in the main text is as follows.
\begin{corollary}[Central limit theorem for $\pij$] \label{CLTP}
Consider Assumptions~\ref{nonodeincontrol}--\ref{skewed}. Define the estimator $\pij=d_i d_j/ \left\|\bm{d}\right\|_1$ for $\E A_{ij}$. Then as $n \rightarrow \infty$, %
\begin{align*}								
		\frac{\pij-\E A_{ij}}{\sqrt{\left(\pi_j^2 \V d_i + \pi_i^2 \V d_j\right)/ \E \left\|\bm{d}\right\|_1}} \stackrel{d}{\rightarrow} \No(0,1).
\end{align*}								
Furthermore for all $i,j$, $\sqrt{ \left(\pi_j^2 \V d_i + \pi_i^2 \V d_j\right)/ \E \left\|\bm{d}\right\|_1}= \mathcal{O}\left(\sqrt{\E A_{ij}/n}\right)$, and can be consistently estimated using a plug-in estimator for $A_{ij} \sim \operatorname{Bernoulli}\left(\pi_i \pi_j\right)$ and $A_{ij} \sim \operatorname{Poisson}\left(\pi_i \pi_j\right)$.			
\end{corollary}

\begin{proof}
We show that $\widehat{\E A}_{ij}=\hat \pi_i \hat \pi_j$, once appropriately standardized,  converges in distribution to a $\No(0,1)$ random variable. It can easily be seen that
\begin{align}
&\pij
= \pi_i \pi_j + \pi_j (\hat \pi_i - \pi_i) + \pi_i (\hat \pi_j - \pi_j) + (\hat \pi_i - \pi_i)(\hat \pi_j - \pi_j). \label{09082015b}
\end{align}
Under the hypothesis that $(\hat \pi_i - \pi_i)(\hat \pi_j - \pi_j)$ is asymptotically negligible, the asymptotic behavior of $\pij - \pi_i \pi_j$ will be dominated by $\pi_j (\hat \pi_i - \pi_i) + \pi_i (\hat \pi_j - \pi_j)$. As a consequence, we standardize all quantities in Eq.\ \eqref{09082015b} by the factor
$\sqrt{\E\left\|\bm{d}\right\|_1/\left(\pi_j^2 \V d_i + \pi_i^2 \V d_j\right)} $, which can be interpreted as an approximation of the standard deviation of $\pi_j (\hat \pi_i - \pi_i) + \pi_i (\hat \pi_j - \pi_j)$. Then, we can use Eq.~\eqref{09082015b} to write
\begin{align*}
& \sqrt{\E\left\|\bm{d}\right\|_1} \frac{\pij - \pi_i \pi_j}{\sqrt{\pi_j^2 \V d_i + \pi_i^2 \V d_j}}
\\ & = \underbrace{\sqrt{\frac{\E\left\|\bm{d}\right\|_1}{\pi_j^2 \V d_i + \pi_i^2 \V d_j}} \left[  \pi_j \sqrt{\V d_i} \left(\frac{\hat \pi_i - \pi_i}{\sqrt{\V d_i}}\right)
 + \pi_i \sqrt{\V d_j} \left(\frac{\hat \pi_j - \pi_j}{\sqrt{\V d_j}}\right) \right]}_{T_1}
\\ & \quad+ \underbrace{\sqrt{\frac{\E\left\|\bm{d}\right\|_1}{\pi_j^2 \V d_i + \pi_i^2 \V d_j}} \cdot (\hat \pi_i - \pi_i)(\hat \pi_j - \pi_j)}_{T_2} .
\end{align*}

To deduce the required result, we will show that $T_1 \stackrel{d}{\rightarrow} \No\left(0,1\right)$ and $T_2 = o_P\left(T_1\right)$. Slutsky's theorem will then enable us to combine these results and obtain the claimed convergence in distribution.

$\underline{\text{Term } T_1}$:
Recall from Corollary~\ref{multiCLTp} that under Assumptions~\ref{nonodeincontrol}--\ref{skewed} it holds that $\sqrt{\E\left\|\bm{d}\right\|_1} \left(\frac{\hat \pi_i - \pi_i}{\sqrt{\V d_i}},\frac{\hat \pi_j - \pi_j}{ \sqrt{\V d_j}}\right)' \stackrel{d}{\rightarrow}\No\left(\bm{0}, \bm{I_2}\right)$.
Applying the Cram\'er--Wold theorem
and Slutsky's theorem, we can conclude that
\begin{align*}
	T_1 \stackrel{d}{\rightarrow} \No(0,1).
\end{align*}

$\underline{\text{Term } T_2}$:
It remains to show that $T_2 = o_P\left(T_1\right)$; i.e., that
\begin{align*}
	(\hat \pi_i - \pi_i)(\hat \pi_j - \pi_j) = o\left(\pi_j (\hat \pi_i - \pi_i) + \pi_i (\hat \pi_j - \pi_j)\right).
\end{align*}

We now use Lemma \ref{15122015} that we will show immediately below.
\begin{align}
\frac{T_2}{T_1}
&= \frac{(\hat \pi_i - \pi_i)(\hat \pi_j - \pi_j) }{\pi_j (\hat \pi_i - \pi_i) + \pi_i (\hat \pi_j - \pi_j)} \notag
\\ &= \left[\frac{\pi_j (\hat \pi_i - \pi_i) + \pi_i (\hat \pi_j - \pi_j)}{(\hat \pi_i - \pi_i)(\hat \pi_j - \pi_j) }\right]^{-1} \notag
\\ &= \left[\frac{\pi_j}{\hat \pi_j - \pi_j } + \frac{\pi_i }{ \hat \pi_i - \pi_i}\right]^{-1} \notag
\\ &= \left[ \Omega \left(\sqrt{ \E d_i}\right) + \Omega \left(\sqrt{ \E d_j}\right)\right]^{-1}  \qquad \text{(see Lemma \ref{15122015})} \notag
\\ &= \mathcal{O}_P \left( \frac{1}{\sqrt{ \E d_i} + \sqrt{ \E d_j}}\right)  \label{11112015}
\end{align}

From Assumption~\ref{sparse} $\left(\pi_i = \omega\left(1/ \sqrt{n}\right)\right)$, it follows that $\min_i \E d_i $ diverges, and hence that $T_2/T_1\stackrel{P}{\rightarrow} 0$.

 \begin{appxlemma} \label{15122015}
 		Consider Assumptions~\ref{nonodeincontrol},~\ref{sparse} and \ref{over-dispersed}. Then,
 		\begin{align*}
 		\hat \pi_i - \pi_i
			 &= \mathcal{O}_P\left( \frac{ \pi_i }{ \sqrt{ \E d_i }} \right).
\end{align*}
 \end{appxlemma}

 \begin{proof}
 First, we appeal to a Taylor expansion in probability of $\hat{\pi}_i= d_i / \sqrt{\|\bm{d}\|_1}$. Let $A = d_i/\E d_i$ and $B=\left(\| \bm{d} \|_1 - 2 d_i\right)/ \E \left(\| \bm{d} \|_1 - 2 d_i\right)$. Observe that the function
 \begin{align} \label{10112015c}
 		\hat{\pi}_i=f(A , B)=\frac{\E d_i \; A}{\sqrt{2 \E d_i \; A + \E \left(\| \bm{d} \|_1 - 2 d_i\right) \; B }}
 \end{align}		
 has continuous partial derivatives at $(1,1)'$. A Taylor expansion in probability \cite[p.~201]{brockwell1991time} of $f$ requires in addition that $ \sqrt{(A - 1)^2 + (B - 1)^2} \stackrel{P}{\rightarrow} 0 $. By Chebyshev's inequality, we know that
\begin{align*}
	\sqrt{(A - 1)^2 + (B - 1)^2} & = \sqrt{\left(\frac{d_i}{\E d_i} - 1\right)^2 + \left(\frac{\| \bm{d} \|_1 - 2 d_i}{ \E \left(\| \bm{d} \|_1 - 2 d_i\right)} - 1\right)^2}
	\\ & = \sqrt{\mathcal{O}_p\left[\V \left(\frac{d_i}{\E d_i}\right)\right] + \mathcal{O}_p\left[\V \left(\frac{\| \bm{d} \|_1 - 2 d_i}{ \E \left(\| \bm{d} \|_1 - 2 d_i\right)}\right)\right]}
		\\ & = \sqrt{\mathcal{O}_p\left[ \frac{\V d_i}{\left(\E d_i\right)^2}\right] + \mathcal{O}_p\left[ \frac{\V \left(\| \bm{d} \|_1 - 2 d_i \right)}{ \left(\E \left(\| \bm{d} \|_1 - 2 d_i\right) \right)^2}\right]}.
\end{align*}
From Assumptions~\ref{sparse} and~\ref{over-dispersed} ($\Rightarrow  \E d_i \rightarrow \infty$, $\V A_{ij}/ \E A_{ij} = \Theta(1)$), it follows that $\sqrt{(A - 1)^2 + (B - 1)^2}  \stackrel{P}{\rightarrow} 0$.

We now can expand the function $f(A,B)$ in Eq.\ \eqref{10112015c} in a convergent Taylor series around $(1,1)'$. In combination with Assumptions~\ref{sparse} and~\ref{over-dispersed} we obtain
\begin{align}
\frac{d_i}{ \sqrt{ \| \bm{d} \|_1 }}
& = \frac{ \E d_i }{ \sqrt{ \E \| \bm{d} \|_1 } } \left[ 1 + \mathcal{O}_P\left( \frac{ 1 }{ \sqrt{ \E d_i } } \right) \right]. \label{12102015d}
\end{align}
Furthermore, we conclude that
\begin{align}
\frac{ \E d_i }{ \sqrt{ \E \| \bm{d} \|_1 } }
&= \frac{ \pi_i \left( 1 - \pi_i / \| \bm{\pi} \|_1 \right) }{ \sqrt{ 1 - \| \bm{\pi} \|_2^2 / \| \bm{\pi} \|_1^2 } }
\\&=  \pi_i \left[ 1 +\mathcal{O}\left(\frac{1}{n}\right)\right] \left[ 1 - \frac{\| \bm{\pi} \|_2^2 }{ \| \bm{\pi} \|_1^2 }\right]^{-1/2} \notag\qquad \text{(Assumption~\ref{nonodeincontrol})}
\\ & = \pi_i \left[ 1 + \mathcal{O}\left(\frac{1}{n}\right)\right]\left[ 1 + \mathcal{O}\left( \frac{ \| \bm{\pi} \|_2^2 }{ \| \bm{\pi} \|_1 }\right)\right] \qquad \text{(Taylor expansion)} \notag
\\  & = \pi_i \left[ 1 + \mathcal{O}\left( \frac{ 1 }{n} \right) \right] . \qquad \text{(see Eq.\ \eqref{twonormoveroneneormsqu})} \label{12102015c}
\end{align}
Combining Eqs.\ \eqref{12102015d} and~\eqref{12102015c}, it follows that
\begin{align*}
\hat \pi_i=\frac{d_i}{ \sqrt{ \| \bm{d} \|_1 }}
  = \pi_i \left[ 1 + \mathcal{O}_P\left( \frac{ 1 }{ \sqrt{ \E d_i } } \right) \right].
\end{align*}
We conclude immediately the result of Lemma~\ref{15122015}; i.e.,
\begin{align}
\hat \pi_i - \pi_i
&= \mathcal{O}_P\left( \frac{ \pi_i }{ \sqrt{ \E d_i } } \right).  \label{12102015e}
\end{align}
\end{proof}

Having established the claimed central limit theorem, we now show that $\sqrt{ \left(\pi_j^2 \V d_i + \pi_i^2 \V d_j\right)/ \E \left\|\bm{d}\right\|_1}= \mathcal{O}\left(\sqrt{\E A_{ij}/n}\right) = \mathcal{O}\left(\sqrt{ \pi_i \pi_j /n}\right)$:
\begin{align*}
	&\sqrt{\frac{n}{\pi_i \pi_j} \cdot \frac{\pi_i^2 \V d_j + \pi_j^2 \V d_i}{\E \|\bm{d}\|_1}}
\\&= \Theta \left(\sqrt{\frac{n}{\pi_i \pi_j} \cdot \frac{\pi_i^2 \E d_j + \pi_j^2 \E d_i}{\E \|\bm{d}\|_1}}\right) \quad \text{(Assumption~\ref{over-dispersed})}
\\&= \Theta \left(\sqrt{\frac{n}{\pi_i \pi_j} \cdot \frac{\pi_i^2 \pi_j + \pi_j^2 \pi_i}{ \|\bm{\pi}\|_1}}\right) \quad \text{(Assumption~\ref{nonodeincontrol})}
\\&= \Theta \left(\sqrt{n \cdot \frac{\pi_i  + \pi_j}{ \|\bm{\pi}\|_1}}\right)
\\&= \mathcal{O} \left(1\right). \quad \text{(Assumption~\ref{nonodeincontrol})}
\end{align*}

To complete the proof of the Corollary, we need to show consistency of the plug-in estimator of $\sqrt{ n \left(\pi_j^2 \V d_i + \pi_i^2 \V d_j\right)/ \sum_{l=1}^n \E d_l}$ for networks with edges $A_{ij} \sim \operatorname{Bernoulli}\left(\pi_i \pi_j\right)$ or $A_{ij} \sim \operatorname{Poisson}\left(\pi_i \pi_j\right)$. We defer this to Theorem~\ref{ThmConsistency}, where we show a more general statement.
\end{proof}

Recall that modularity $\widehat{Q}$ (Eq.~[1] in main text) is an empirical quantity that  estimates its population counterpart $Q$ (Eq.~[2] in main text), in the sense that $\E A_{ij}$ is estimated using $\pij$.  For each individual $\pij$, we show now that $\pij - \E A_{ij} \stackrel{P}{\rightarrow} 0$ at a rate no slower than $\left(\pi_i + \pi_j\right)/\sqrt{n}$ (Assumption~\ref{notcomplete}). More precisely we have the following.
\begin{appxlemma}
 		Consider Assumptions \ref{nonodeincontrol}, \ref{sparse}, and \ref{over-dispersed}. Then,
 		\begin{align*}
 	\pij - \E A_{ij} &= \mathcal{O}_P\left(\frac{\pi_i + \pi_j}{\sqrt{n}}\right).
\end{align*}
From Assumption \ref{notcomplete}, we know that $\left(\pi_i + \pi_j\right)/\sqrt{n}=o_P\left(1\right)$.
 \end{appxlemma}

\begin{proof}
Recall from Eq.\ \eqref{09082015b} that
\begin{align*}
\pij - \E A_{ij}
&=   \pi_j (\hat \pi_i - \pi_i) + \pi_i (\hat \pi_j - \pi_j) + (\hat \pi_i - \pi_i)(\hat \pi_j - \pi_j).
\intertext{Furthermore, we know from Eq.\ \eqref{11112015} that}
&=   \left(\pi_j (\hat \pi_i - \pi_i) + \pi_i (\hat \pi_j - \pi_j)\right)
\left[1+ \mathcal{O}_P \left( \frac{1}{\sqrt{ \E d_i} + \sqrt{ \E d_j}}\right) \right]
\intertext{From Lemma \ref{15122015} and Assumptions \ref{nonodeincontrol}, \ref{sparse} and \ref{over-dispersed}, it follows that}
&=   \mathcal{O}_P \left( \frac{\pi_i \pi_j}{\sqrt{\E d_i}} + \frac{\pi_i \pi_j}{\sqrt{\E d_j}}\right)
\\&=  \mathcal{O}_P \left( \sqrt{\frac{\pi_i}{\left\|\bm{\pi}\right\|_1}} \pi_j + \sqrt{\frac{\pi_j}{\left\|\bm{\pi}\right\|_1}} \pi_i\right) \quad \text{(Assumption \ref{nonodeincontrol})}
\\&=  \mathcal{O}_P \left( \frac{\pi_j}{\sqrt{n}} + \frac{\pi_i}{\sqrt{n}} \right) \quad \text{(Assumption \ref{nonodeincontrol})}
\\&=  o_P \left(1\right) \quad \text{(Assumption \ref{notcomplete})}.
\end{align*}
\end{proof}

\section{Consistency of the plug-in estimator for \texorpdfstring{\newline $\sqrt{\V d_i/ \E \| \bm{d} \|_1}$}{sqrt(Var di/E||d||1)}} \label{SecConsistency}

Throughout the Theorem and Corollaries in the main text (and above), we state that $\sqrt{\V d_i/ \E \| \bm{d} \|_1}$ can be consistently estimated using a plug-in estimator for $A_{ij} \sim \operatorname{Bernoulli}\left(\pi_i \pi_j\right)$ and $A_{ij} \sim \operatorname{Poisson}\left(\pi_i \pi_j\right)$. In fact, this is true more generally, as we show below.

Each edge distribution leads to a different variance $\V d_i$, each of which is $\Theta\left( \E d_i \right)$ by Assumption \ref{over-dispersed}. We now show that the term $\sqrt{\V d_i/ \E \| \bm{d} \|_1}$ can be consistently estimated by a plug-in estimator, as long as $\V d_i$ can be consistently estimated by a plug-in estimator.  More precisely, we have the following.

\begin{appxtheorem}[Consistency of plug-in estimator for {$\sqrt{\V d_i/ \E \| \bm{d} \|_1}$}] \label{ThmConsistency}
 	Consider Assumptions \ref{nonodeincontrol}, \ref{sparse} and \ref{over-dispersed}.  Define plug-in estimators $\widehat{\V d_i}$ and $\widehat{\E \| \bm{d} \|_1}$ by exchanging each $\pi_i$ in $\V d_i$ and $\E \| \bm{d} \|_1$ by $\hat{\pi}_i=d_i/\sqrt{\| \bm{d} \|_1}$. In addition, assume that
\begin{equation*}
\frac{ \widehat{\V d}_i }{ \V d_i } \stackrel{P}{\rightarrow} 1.
\end{equation*}
 	 Then,  $\sqrt{\V d_i/ \E \| \bm{d} \|_1}$ can be estimated consistently using the plug-in estimator $\sqrt{\widehat{\V d_i} / \widehat{\E \| \bm{d} \|_1}}$; i.e.,
\begin{align*}
 \frac{\sqrt{\widehat{\V d_i}/ \widehat{\E \| \bm{d} \|_1}}}{\sqrt{\V d_i/ \E \| \bm{d} \|_1}} \stackrel{P}{\rightarrow} 1.
\end{align*}	
\end{appxtheorem}

\begin{proof}
We first write
\begin{align}
\frac{\sqrt{\widehat{\V d_i}/ \widehat{\E \| \bm{d} \|_1}}}{\sqrt{\V d_i/ \E \| \bm{d} \|_1}} & = \sqrt{\frac{\widehat{\V d}_i}{\V d_i}}
		 \sqrt{\frac{\E \left\|\bm{d}\right\|_1}{\| \hat{\bm{\pi}} \|_1^2 - \| \hat{\bm{\pi}} \|_2^2}} \notag
\\& = \sqrt{\frac{\widehat{\V d}_i}{\V d_i}}
		 \sqrt{\frac{\E \left\|\bm{d}\right\|_1}{\| \bm{d} \|_1^2 - \| \bm{d} \|_2^2} \left\|\bm{d}\right\|_1} \notag
\\ &		 = \sqrt{\frac{\widehat{\V d}_i}{\V d_i}}  \sqrt{\frac{\E \left\|\bm{d}\right\|_1}{\left\|\bm{d}\right\|_1}} \left[1-\frac{\left\|\bm{d}\right\|_2^2}{\left\|\bm{d}\right\|_1^2}\right]^{-\frac{1}{2}}. \label{12102015i}
\end{align}

From term $T_3$ (Eq.\ \eqref{proofTermT3}) in the proof of Theorem \ref{App:GeneralCLTPoisson}, we know that under Assumption \ref{over-dispersed} $\left(\V A_{ij} = \Theta\left(\E A_{ij}\right)\right)$ it holds that $\sqrt{\E \left\|\bm{d}\right\|_1/\left\|\bm{d}\right\|_1} \stackrel{P}{\rightarrow} 1$ . Since we assume $\widehat{\V d}_i/ \V d_i \stackrel{P}{\rightarrow} 1$, it remains to show that
\begin{align*}
\frac{\left\|\bm{d}\right\|_2^2}{\left\|\bm{d}\right\|_1^2} \stackrel{P}{\rightarrow} 0.
\end{align*}

First, from Chebyshev's inequality, and from Assumption \ref{over-dispersed}, we know that
\begin{align}
\frac{\left\|\bm{d}\right\|_2^2}{\E \left\|\bm{d}\right\|_2^2} = 1 + \mathcal{O}_P\left(\frac{1}{\sqrt{\E \left\|\bm{d}\right\|_2^2}}\right)
\quad \text{and} \quad
\frac{\left\|\bm{d}\right\|_1^2}{\E \left\|\bm{d}\right\|_1^2} = 1 + \mathcal{O}_P\left(\frac{1}{\sqrt{\E \left\|\bm{d}\right\|_1^2}}\right). \label{27102015a}
\end{align}
In return, it follows that
\begin{align}
\frac{\left\|\bm{d}\right\|_2^2}{\left\|\bm{d}\right\|_1^2}
&= \frac{\E \left\|\bm{d}\right\|_2^2}{\E \left\|\bm{d}\right\|_1^2} \left[1 + \mathcal{O}_P\left(\frac{1}{\sqrt{\E \left\|\bm{d}\right\|_2^2}}\right)\right] \left[1 + \mathcal{O}_P\left(\frac{1}{\sqrt{\E \left\|\bm{d}\right\|_1^2}}\right)\right]^{-1}. \notag
\intertext{We may apply a convergent Taylor expansion of $f(x) = (1 + x)^{-1}$ at $1$, since $x=1/\sqrt{\E \left\|\bm{d}\right\|_1^2} = o(1)$. It follows that}
&= \frac{\E \left\|\bm{d}\right\|_2^2}{\E \left\|\bm{d}\right\|_1^2} \left[1 + \mathcal{O}_P\left(\frac{1}{\sqrt{\E \left\|\bm{d}\right\|_2^2}}\right)\right] \left[1 + \mathcal{O}_P\left(\frac{1}{\sqrt{\E \left\|\bm{d}\right\|_1^2}}\right)\right]  \notag
\\&= \frac{\E \left\|\bm{d}\right\|_2^2}{\E \left\|\bm{d}\right\|_1^2} \left[1 + \mathcal{O}_P\left(\frac{1}{\sqrt{\E \left\|\bm{d}\right\|_2^2}}\right)\right].  \qquad \left(\text{since } \left\|\bm{d}\right\|_2^2 \leq \left\|\bm{d}\right\|_1^2  \right)\label{12102015g}
\end{align}
 Via straightforward algebraic computations, we obtain
 \begin{align}
 	\E \left\|\bm{d}\right\|_2^2&= \sum_i \sum_{j \neq i} \sum_{l \neq i} \E \left(A_{ij} A_{il}\right) \notag
\\ 	&= \sum_i  \E d_i \E d_i \cdot \left(1+ o(1)\right) \notag
\\ 	&=  \left\|\bm{\pi}\right\|_1^2 \left\|\bm{\pi}\right\|_2^2 \cdot \left(1+ o(1)\right), \qquad \text{(Assumption \ref{nonodeincontrol})} \label{12102015f}
 \intertext{and}
 	\E \left\|\bm{d}\right\|_1^2&= \V \left\|\bm{d}\right\|_1 + \left(\E \left\|\bm{d}\right\|_1\right)^2 \label{29102015b}
 	\\ &= \Theta \left(\E \left\|\bm{d}\right\|_1\right) + \left(\E \left\|\bm{d}\right\|_1\right)^2 \quad \text{(Assumption \ref{over-dispersed})} \notag
\\ & = \Theta \left[\left(\E \left\|\bm{d}\right\|_1\right)^2\right]. \label{29102015}
\quad \text{(Assumption \ref{sparse})}
 	\end{align}
We know from Eq.\ \eqref{12102015g} that
 \begin{align} \notag
\frac{\left\|\bm{d}\right\|_2^2}{\left\|\bm{d}\right\|_1^2}
&= \frac{\E \left\|\bm{d}\right\|_2^2}{\E \left\|\bm{d}\right\|_1^2} \left[1 + \mathcal{O}_P\left(\frac{1}{\sqrt{\E \left\|\bm{d}\right\|_2^2}}\right)\right].
\intertext{Combining Eqs.\ \eqref{12102015f} and~\eqref{29102015} and applying Assumption \ref{nonodeincontrol}, it then follows that}
&= \frac{\left\|\bm{\pi}\right\|_2^2}{\left\|\bm{\pi}\right\|_1^2} \left[1 + \mathcal{O}_P\left(\frac{1}{\left\|\bm{\pi}\right\|_1 \left\|\bm{\pi}\right\|_2}\right)\right] \notag
\\&= \mathcal{O}_P\left(\frac{1}{ n}\right). \qquad \text{(see Eq.\ \eqref{twonormoveroneneormsqu})} \notag
\end{align}

Finally, we know from Eq.\ \eqref{12102015i} that
 \begin{align*}
\frac{\sqrt{\widehat{\V d_i}/ \widehat{\E \| \bm{d} \|_1}}}{\sqrt{\V d_i/ \E \| \bm{d} \|_1}}
		 = \sqrt{\frac{\widehat{\V d}_i}{\V d_i}}  \sqrt{\frac{\E \left\|\bm{d}\right\|_1}{\left\|\bm{d}\right\|_1}} \left[1-\frac{\left\|\bm{d}\right\|_2^2}{\left\|\bm{d}\right\|_1^2}\right]^{-\frac{1}{2}}.
\end{align*}
The inverse of a random variable which converges in probability to a constant $c$ must in turn converge to $1/c$, as long as $c \neq 0$ \cite[Theorem 2.1.3]{Lehmann199812}. Applying this fact and the continuous mapping theorem, we obtain the claimed convergence in probability; i.e.,
  \begin{align*}
\frac{\sqrt{\widehat{\V d_i}/ \widehat{\E \| \bm{d} \|_1}}}{\sqrt{\V d_i/ \E \| \bm{d} \|_1}}
		  \stackrel{P}{\rightarrow} 1.
\end{align*}  	   		
\end{proof}

Having established Theorem \ref{ThmConsistency}, we now show for $A_{ij} \sim \operatorname{Bernoulli}\left(\pi_i \pi_j\right)$ and $A_{ij} \sim \operatorname{Poisson}\left(\pi_i \pi_j\right)$ that  $\widehat{\V d}_i/ \V d_i \stackrel{P}{\rightarrow} 1$.
This allows us to apply Theorem \ref{ThmConsistency} to conclude that $\sqrt{\V d_i/ \E \| \bm{d} \|_1}$ can be estimated consistently via its plug-in estimator. \\

\underline{$A_{ij} \sim \operatorname{Poisson}\left(\pi_i \pi_j\right)$}:
For Poisson-distributed edges, $\E A_{ij}= \V A_{ij}$ for all $i,j$. Hence, we obtain
\begin{align} \notag
	\frac{	\widehat{\V d}_i}{ \V d_i} &= \frac{	\widehat{\E d}_i}{ \E d_i}
	\\ \notag & = \frac{	\hat{\pi}_1 \left\|\bm{\hat{\pi}}\right\|_1 - \hat{\pi}_i^2}{ \E d_i}
	\\ \notag & = \frac{	\frac{d_i}{\sqrt{\left\| \bm{d}\right\|_1}} \frac{\left\| \bm{d} \right\|_1}{\sqrt{\left\| \bm{d} \right\|_1}} - \frac{d_i^2}{\left\| \bm{d}\right\|_1}}{ \E d_i}
	\\ \notag& = \frac{d_i}{ \E d_i} \left[1 - \frac{d_i}{\left\| \bm{d}\right\|_1} \right]
	\\ \notag& = \left[1 + \mathcal{O}_P\left(\sqrt{\frac{\V d_i}{\left(\E d_i\right)^2}}\right) \right] \left[1 - \frac{d_i}{\left\| \bm{d}\right\|_1} \right] \quad \text{(Chebyshev's inequality)}
	\\ \label{26102015}& = \left[1 + \mathcal{O}_P\left(\frac{1}{\sqrt{\E d_i}}\right) \right] \left[1 - \frac{d_i}{\left\| \bm{d}\right\|_1} \right]. \quad \text{(Assumption \ref{over-dispersed})}
\end{align}

Furthermore, from Assumptions \ref{nonodeincontrol} ($n \pi_i/\left\|\bm{\pi}\right\|_1 = \mathcal{O}(1) $), \ref{sparse} ($\Rightarrow \E d_i  \to \infty$), and \ref{over-dispersed} ($\V A_{ij} = \Theta \left(\E A_{ij}\right)$), it follows that $\frac{\left\|\bm{\pi}\right\|_1}{\pi_i} \, \frac{d_i}{\| \bm{d} \|_1} \stackrel{P}{\rightarrow} 1$, as we will now show.

We write
			\begin{equation} \label{11112015c}
						\frac{\left\|\bm{\pi}\right\|_1}{\pi_i} \, \frac{d_i}{\| \bm{d} \|_1}  = \underbrace{\left( \frac{\left\|\bm{\pi}\right\|_1}{\pi_i} \, \frac{\E d_i  }{\| \bm{\pi} \|_1^2}
			 		\right)	}_{c_n} { \underbrace{\left( \frac{\| \bm{d} \|_1}{\| \bm{\pi} \|_1^2} \right)}_{E_n} }^{-1} \underbrace{\left( \frac{ d_i }{\E d_i  }\right) }_{F_n} .
			\end{equation}
						By Chebyshev's inequality and from Assumptions \ref{sparse} and \ref{over-dispersed}, we know that
\begin{equation*}
						F_n=\frac{d_i }{ \E d_i} = 1 + \mathcal{O}_P\left( \frac{1}{\sqrt{\E d_i}}\right).
\end{equation*}						 For $E_n$, we will first establish the equivalence
			\begin{align*}
				&\frac{\| \bm{d} \|_1}{ \E \| \bm{d} \|_1}
				= \frac{\| \bm{d} \|_1 }{ \| \bm{\pi} \|_1^2 - \| \bm{\pi} \|_2^2}
				= \frac{\| \bm{d} \|_1 }{ \| \bm{\pi} \|_1^2}  \left[1 - \frac{\| \bm{\pi} \|_2^2}{\| \bm{\pi} \|_1^2}\right]^{-1}
				\\ \Leftrightarrow \quad  & \frac{\| \bm{d} \|_1 }{ \| \bm{\pi} \|_1^2} = \frac{\| \bm{d} \|_1}{ \E \| \bm{d} \|_1} \left[1 - \frac{\| \bm{\pi} \|_2^2}{\| \bm{\pi} \|_1^2}\right].
			\end{align*}
By Eq.\ \eqref{twonormoveroneneormsqu}, we know that from Assumption \ref{nonodeincontrol} it follows that $\| \bm{\pi} \|_2^2/\| \bm{\pi} \|_1^2 = \mathcal{O}\left(1/n\right)$.	Furthermore, by Chebyshev's inequality and from Assumptions \ref{sparse} and \ref{over-dispersed}, $\| \bm{d} \|_1 / \E \| \bm{d} \|_1  \overset{P}{\longrightarrow} 1$. Thus, it follows that
			\begin{align} \label{27102015}
	E_n=\frac{\| \bm{d} \|_1}{\| \bm{\pi} \|_1^2} = \frac{\| \bm{d} \|_1}{ \E \| \bm{d} \|_1} \left[1 - \frac{\| \bm{\pi} \|_2^2}{\| \bm{\pi} \|_1^2}\right]
	= 1 + \mathcal{O}_P\left(\frac{1}{\min\left(n, \sqrt{\E \| \bm{d} \|_1 }\right)}\right).
			\end{align}
For the non-random sequence $\{c_n; n \in \N\}$ in Eq.\ \eqref{11112015c} it holds that
\begin{align*}
	c_n &=  \frac{\left\|\bm{\pi}\right\|_1}{\pi_i} \, \frac{\E d_i  }{\| \bm{\pi} \|_1^2}
\\ &=  \frac{\left\|\bm{\pi}\right\|_1}{\pi_i} \, \frac{\pi_i \left\|\bm{\pi}\right\|_1 }{\| \bm{\pi} \|_1^2} \left[1- \frac{\pi_i}{\left\|\bm{\pi}\right\|_1 }\right]
\\ &=   \left[1 + \mathcal{O}\left(\frac{1}{n}\right)\right]. \quad \text{(Assumption \ref{nonodeincontrol})}
\end{align*}

 	  The inverse of a random variable which converges in probability to a constant $c$ must in turn converge to $1/c$, as long as $c \neq 0$ \cite[Theorem 2.1.3]{Lehmann199812}. Furthermore, the product of two random variables, converging in probability to a constant $c$ and a constant $d$ respectively, itself converges to the product of the constants $cd$ \cite[Theorem 2.1.3]{Lehmann199812}. Thus, it follows that
  \begin{align}
 	  \label{26102015g} &\frac{\left\|\bm{\pi}\right\|_1}{\pi_i} \, \frac{d_i}{\| \bm{d} \|_1} = c_n E_n^{-1} F_n =  1 + \mathcal{O}_P\left(\frac{1}{\min_i \left(\sqrt{\E d_i}, n\right)}\right).
 	   \\ \notag \Leftrightarrow \quad & \frac{d_i}{\| \bm{d} \|_1} = \mathcal{O}_P \left(\frac{1}{n}\right). \quad \text{(Assumption \ref{nonodeincontrol})}
 \end{align}
Recall from Eq.\ \eqref{26102015} that
 	\begin{align*}
	\frac{	\widehat{\V d}_i}{ \V d_i}
 & = \left[1 + \mathcal{O}_P\left(\frac{1}{\sqrt{\E d_i}}\right) \right] \left[1 - \frac{d_i}{\left\| \bm{d}\right\|_1} \right].
 \intertext{In turn, we obtain the required result; i.e.,}
 \frac{	\widehat{\V d}_i}{ \V d_i}  & = 1 + \mathcal{O}_P\left(\frac{1}{\min_i \left(\sqrt{\E d_i}, n\right)} \right).
\end{align*}
From Assumption \ref{sparse} $\left(\pi_i = \omega\left(1/ \sqrt{n}\right)\right)$, it follows that $\min_i \E d_i $ diverges. Hence, we have shown the required result that $\V d_i$ can be consistently estimated by its plug-in estimator $\widehat{\V d}_i$.  \\
 	
\underline{$A_{ij} \sim \operatorname{Bernoulli}\left(\pi_i \pi_j\right)$}:
For Bernoulli-distributed edges, we obtain $\V d_i = \E d_i - \pi_i^2 \left\|\bm{\pi}\right\|_2^2 + \pi_i^4$ \cite{Olhede2012}. We write
\begin{align}
\notag	\frac{	\widehat{\V d}_i}{ \V d_i}
 & = \frac{\hat{\pi}_i \left\|\bm{\hat{\pi}}\right\|_1 - \hat{\pi}_i^2 - \hat{\pi}_i^2 \left\|\bm{\hat{\pi}}\right\|_2^2 + \hat{\pi}_i^4}{ \pi_i \left\|\bm{\pi}\right\|_1 - \pi_i^2 - \pi_i^2 \left\|\bm{\pi}\right\|_2^2 + \pi_i^4}.
\end{align}
 It can easily been seen that $\hat{\pi}_i \left\|\bm{\hat{\pi}}\right\|_1 = d_i$ and $\left\|\bm{\hat{\pi}}\right\|_2^2 = \left\|\bm{d}\right\|_2^2 / \left\|\bm{d}\right\|_1$. It follows that
\begin{align}
  \notag & = \frac{d_i - d_i^2/\left\|\bm{d}\right\|_1 - d_i^2 \left\|\bm{d}\right\|_2^2/\left\|\bm{d}\right\|_1^2 + d_i^4/\left\|\bm{d}\right\|_1^2}{ \pi_i \left\|\bm{\pi}\right\|_1 - \pi_i^2 - \pi_i^2 \left\|\bm{\pi}\right\|_2^2 + \pi_i^4}
\\ \notag & = \frac{d_i - d_i^2/\left\|\bm{d}\right\|_1 - d_i^2 \left\|\bm{d}\right\|_2^2/\left\|\bm{d}\right\|_1^2 + d_i^4/\left\|\bm{d}\right\|_1^2}{ \pi_i \left\|\bm{\pi}\right\|_1 - \pi_i^2
 \left\|\bm{\pi}\right\|_2^2 } \cdot \left[1 + o(1)\right] \qquad \text{(Assumption \ref{nonodeincontrol})}
 \\ \label{26102015b} & = \frac{d_i \left[ 1 - d_i/\left\|\bm{d}\right\|_1\right] - d_i^2 \left[  \left\|\bm{d}\right\|_2^2/\left\|\bm{d}\right\|_1^2 + d_i^2/\left\|\bm{d}\right\|_1^2\right]}{ \pi_i \left\|\bm{\pi}\right\|_1 - \pi_i^2
 \left\|\bm{\pi}\right\|_2^2 } \cdot \left[1 + o(1)\right].
\end{align}
We have seen in Eq.\ \eqref{12102015d} that Assumptions \ref{sparse} and \ref{over-dispersed} imply that
\begin{align*}
\frac{d_i}{ \sqrt{ \| \bm{d} \|_1 }}
& = \frac{ \E d_i }{ \sqrt{ \E \| \bm{d} \|_1 } } \left[ 1 + \mathcal{O}_P\left( \frac{ 1 }{ \sqrt{ \E d_i } } \right) \right].
\end{align*}
It follows from identical arguments that
\begin{align}
\frac{d_i}{  \| \bm{d} \|_1 }
& = \frac{ \E d_i }{  \E \| \bm{d} \|_1 }  \left[ 1 + \mathcal{O}_P\left( \frac{ 1 }{ \sqrt{ \E d_i } } \right) \right]. \label{26102015c}
\end{align}
From Assumption \ref{nonodeincontrol}, we conclude that
\begin{align}
\frac{ \E d_i }{  \E \| \bm{d} \|_1 }
&= \frac{ \pi_i \left( 1 - \pi_i / \| \bm{\pi} \|_1 \right) }{  \| \bm{\pi} \|_1 \left(1 - \| \bm{\pi} \|_2^2 / \| \bm{\pi} \|_1^2\right) }  \notag
\\  & = \frac{\pi_i}{ \| \bm{\pi} \|_1} \left[ 1 + \mathcal{O}\left( \frac{ 1 }{n} \right) \right]  \qquad \text{(see Eq.\ \eqref{12102015c})} \notag
\\  & =  \mathcal{O}\left( \frac{ 1 }{n} \right) . \qquad \text{(Assumption \ref{nonodeincontrol})} \label{26102015d}
\end{align}
Combining Eqs.\ \eqref{26102015c} and~\eqref{26102015d}, it follows that
\begin{align*}
\frac{d_i}{ \| \bm{d} \|_1 }
  = \mathcal{O}_P\left( \frac{ 1 }{n} \right) .
\end{align*}

It follows in turn that in combination with Eq.\ \eqref{26102015b}, we obtain
\begin{align}
\frac{	\widehat{\V d}_i}{ \V d_i}
 & = \frac{d_i  - d_i^2   \left\|\bm{d}\right\|_2^2/\left\|\bm{d}\right\|_1^2 }{ \pi_i \left\|\bm{\pi}\right\|_1 - \pi_i^2
 \left\|\bm{\pi}\right\|_2^2 } \cdot \left[1 + o_P(1)\right] \notag
\\ & = \underbrace{\frac{d_i}{ \pi_i \left\|\bm{\pi}\right\|_1}}_{R_n} \cdot \underbrace{\frac{ 1  - d_i/\left\|\bm{d}\right\|_1   \left\|\bm{d}\right\|_2^2/\left\|\bm{d}\right\|_1 }{1 - \pi_i / \left\|\bm{\pi}\right\|_1 \;
 \left\|\bm{\pi}\right\|_2^2}}_{S_n} \cdot \left[1 + o_P(1)\right]. \label{15122015h}
\end{align}

\underline{Term $R_n$:}
\begin{align*}
	R_n &= \frac{d_i}{ \pi_i \left\|\bm{\pi}\right\|_1}
	\\ & = \frac{\E d_i}{ \pi_i \left\|\bm{\pi}\right\|_1} \left[1 + \mathcal{O}_P\left(\sqrt{\frac{\V d_i}{\left(\E d_i\right)^2}}\right)\right] \qquad \text{(Chebyshev's inequality)}
		\\ & = \frac{\E d_i}{ \pi_i \left\|\bm{\pi}\right\|_1} \left[1 + \mathcal{O}_P\left(\frac{1}{\sqrt{\E d_i}}\right)\right] \qquad \text{(Assumption \ref{over-dispersed})}
		\\ & = 1 + \mathcal{O}_P\left(\frac{1}{\sqrt{\E d_i}}\right) \qquad \text{(Assumption \ref{nonodeincontrol})}
		\\ & = 1 + o_P\left(1\right). \qquad \text{(Assumption \ref{sparse})}
\end{align*}

\underline{Term $S_n$:}
 We show the convergence of $S_n$ from Eq.\ \eqref{15122015h} in two steps:
\begin{enumerate}
	\item $						\frac{\left\|\bm{\pi}\right\|_1}{\pi_i} \, \frac{d_i}{\| \bm{d} \|_1} \overset{P}{\longrightarrow} 1
$;
\item $					\left(\left\|\bm{\pi}\right\|_2^2\right)^{-1} \frac{\| \bm{d}\|_2^2}{ \| \bm{d}\|_1}\quad \stackrel{P}{\rightarrow} \quad 1
$.
\end{enumerate}

\underline{Step 1:} This step follows analogously to Eq.\ \eqref{26102015g} for  $A_{ij} \sim \operatorname{Poisson}\left(\pi_i \pi_j\right)$.

\underline{Step 2:}	We write the ratio of interest as
					\begin{align*}
							&\left(\left\|\bm{\pi}\right\|_2^2\right)^{-1} \frac{\| \bm{d}\|_2^2}{ \| \bm{d}\|_1} = \left(\frac{\| \bm{d}\|_1}{ \left\|\bm{\pi}\right\|_1^2}\right)^{-1} \cdot	\left(\frac{\| \bm{d}\|_2^2}{ \E \| \bm{d}\|_2^2  }\right)
								\cdot \left( \frac{\E \| \bm{d}\|_2^2  }{\left\|\bm{\pi}\right\|_2^2 \left\|\bm{\pi}\right\|_1^2}\right) = L_n^{-1} M_n t_n.
					\end{align*}
 Now, we analyze $L_n$, $M_n$ and $t_n$ in consecutive order.	%
 Under Assumptions \ref{nonodeincontrol}, \ref{sparse} and \ref{over-dispersed}, we know that $L_n=\| \bm{d} \|_1 / \| \bm{\pi} \|_1^2
					\overset{P}{\longrightarrow} 1$  (see Eq.\ \eqref{27102015}).
 Furthermore, combining Eqs.\ \eqref{27102015a} and \eqref{12102015f} enables us to conclude that $M_n=\|\bm{d}\|_2^2/\E \|\bm{d}\|_2^2 \overset{P}{\longrightarrow} 1$ (under Assumptions \ref{nonodeincontrol} and \ref{over-dispersed}). %
 From Eq.\ \eqref{12102015f}, we know that under Assumption \ref{nonodeincontrol}, the sequence $\left\{t_n; n \in \N\right\}$ converges to 1.

  	  The inverse of a random variable which converges in probability to a constant $c$, must in turn converge to $1/c$, as long as $c \neq 0$ \cite[Theorem 2.1.3]{Lehmann199812}. Furthermore, the product of two random variables, converging in probability to a constant $c$ and a constant $d$ respectively, itself converges to the product of the constants $cd$ \cite[Theorem 2.1.3]{Lehmann199812}. Thus, Step 2 follows.
  	
  	  Returning now to Eq.\ \eqref{15122015h} and following the same argument, we conclude that $S_n \stackrel{P}{\rightarrow} 1$ and in turn, $	\widehat{\V d}_i/ \V d_i = R_n S_n \left[1+o_P(1)\right] \stackrel{P}{\rightarrow} 1$ for Bernoulli-distributed edges $(A_{ij} \sim \operatorname{Bernoulli}\left(\pi_i \pi_j\right))$.

\section{Proof of Theorem 1}
We now state and prove Theorem \ref{DensityM}, which is identical to Theorem~1 in the main text, except for the formulation of the weights $\beta_j$, $j=1, \ldots, n$.  In Corollary \ref{APP:LemmabetaStar} below, we introduce the formulation for $\beta_j$ used in Theorem~1 to improve interpretability and show that both formulations are asymptotically equivalent. The proof below expands on the proof sketch given in the main text.

\begin{appxtheorem}[Central limit theorem for modularity]
\label{DensityM}
In addition to  Assumptions \ref{nonodeincontrol}--\ref{skewed}, suppose that the number $K$ of communities grows strictly more slowly than $n$(; i.e., $K/n \rightarrow 0$). Then, as $n \rightarrow \infty$,
 	\begin{align*}
 			\frac{\widehat{Q}- \muMod}{\sigMod} \stackrel{d}{\rightarrow} \No(0,1),
 	\end{align*}
where
\begin{align*}
		\muMod &= \sum_{j=1}^n \sum_{i<j}  \frac{\E A_{ij} \left(\E d_i + \E d_j - \left\|\bm{\pi}\right\|_2^2\right)}{\E \left\|\bm{d}\right\|_1} \ingroup,
		\\ \sigMod^2 &= \sum_{j=1}^n \sum_{i < j} \left[  \ingroup  +  \beta_i + \beta_j \right]^2 \V \left( A_{ij} \right).
\end{align*}
\end{appxtheorem}
The $\beta_i$ are defined in Eq.\ \eqref{beta} in Lemma \ref{Modularity_reform} below and are non-random.

\begin{proof}
The proof consists of two main steps. First, in Lemma \ref{Modularity_reform}, we will relate modularity to a linear combination of within-group degrees ($d_i^w$ in Eq.\ \eqref{withinandbetween} below) and between-group degrees ($d_i^b$ in Eq.\ \eqref{withinandbetween} below). Second, in Lemma~\ref{billinslayformod}, we will show that this linear combination,  when appropriately standardized, converges in distribution to a $\No\left(0,1\right)$ random variable.

 Let us first note some preliminaries. Recall from the main text:
\begin{align}
d_j^w &= \sum_{i \neq j} A_{ij} \ingroup \quad  \text{and} \quad
d_j^b = \sum_{i \neq j} A_{ij} \NOTingroup. \label{withinandbetween}
\end{align}
 Let us denote
 \begin{align*}
 		  &\ingroupPara=\sum_{i \neq j} \pi_i \; \ingroup
 		  \quad  \text{and} \quad
 		\left\|\bm{\pi}\right\|_1^{\neg g(j)}=\sum_{i =1}^n \pi_i \; \NOTingroup.
 \end{align*}		
 We obtain
\begin{align}
		\E d_j^w = &\pi_j \ingroupPara
		\quad  \text{and} \quad
	\E d_j^b = \pi_j \|\bm{\pi}\|_1^{\neg g(j)} \label{Expdib}.
\end{align}

We are now ready to proceed with our analysis. The following Lemma is identical to Lemma~1 in the main document.
\begin{appxlemma} \label{Modularity_reform}
Consider Assumptions \ref{nonodeincontrol}--\ref{over-dispersed} ($\pi_i/ \left\|\bm{\pi}\right\|_1= \mathcal{O}\left(1/ n\right)$, $\pi_i = \omega\left(1/ \sqrt{n}\right)$, $\pi_i = o\left(\sqrt{n}\right), \E A_{ij}=\Theta\left(\V A_{ij}\right)$). Then, the following identity holds:
\begin{align}
 		\widehat{Q} &= \muMod + \left(\sum_{j=1}^n \alpha_j \left[d_j^w - \E d_j^w\right]  + \sum_{j=1}^n \beta_j \left[d_j^b - \E d_j^b\right]\right) + \mathcal{O}_P\left(\errorModtotal\right), \notag
\intertext{where the non-random quantities $\alpha_j$, $\beta_j$, and $\errorModtotal$ are defined as follows:}
\beta_j & = \left[ \frac{1}{2} \sum_{l=1}^n   \ingroupParaL  \frac{\E d_l}{\E \left\| \bm{d} \right\|_1} -  \ingroupPara\right]  \frac{1}{\sqrt{\E \left\| \bm{d} \right\|_1}}, \label{beta}
\\ \alpha_j & = \frac{1}{2}  + \beta_j, \label{alpha}
\\ \errorModtotal & = \frac{\sum_{j=1}^n \sum_{i<j} \pi_i \pi_j \ingroup}{\min\left(n, \left\|\bm{\pi}\right\|_1\right) \, \min_l \sqrt{ \E d_l }}. \label{errorModTotal}
\end{align}
\end{appxlemma}
\begin{proof}
Since $\pij = d_i d_j/ \left\|\bm{d}\right\|_1$, modularity can be written as
\begin{align} \label{12102015j}
	\widehat{Q}  & = \sum_{j=1}^n \sum_{i<j} A_{ij}\ingroup
 	- \sum_{j=1}^n \sum_{i<j} \pij \ingroup.
\end{align}

We will show this lemma in six steps. We
\begin{enumerate}
	\item Write $\pij$ in terms of $\hat \pi_j = d_j / \sqrt{\left\|\bm{d}\right\|_1}$; \label{Epandpij}
	\item Expand the denominator $\sqrt{\left\|\bm{d}\right\|_1}$ around its mean in a convergent Taylor series; \label{Taylord1noem}
	\item Substitute $d_j = \E d_j + \mathcal{O}_P\left(\sqrt{\E d_j}\right)$ into the lower-order terms of the Taylor expansion of Step \ref{Taylord1noem};  \label{Taylordj}
	\item Apply the decomposition $d_j = d_j^w + d_j^b$, and center $d_j^w$ and $d_j^b$ about their respective means $\E d_j^w$ and $\E d_j^b$; \\[-0.3cm]				\label{centerdj}
	\item Collect all higher-order non-random terms in $\widehat{Q}$ into $\muMod$; and  \label{nonrandomtems}
	\item Show that the remaining lower-order random and non-random terms can be absorbed into $\errorModtotal$.  \label{errorterm}
\end{enumerate}

\underline{Step \ref{Epandpij}:}
Recall from Eq.\ \eqref{09082015b} that
\begin{align}
\pij &= \hat{\pi}_i \hat{\pi}_j \notag
\\ & = \pi_i \pi_j + \pi_j (\hat \pi_i - \pi_i) + \pi_i (\hat \pi_j - \pi_j) + (\hat \pi_i - \pi_i)(\hat \pi_j - \pi_j), \notag
\end{align}
and from Eq.\ \eqref{11112015} that, given  Assumptions \ref{nonodeincontrol}, \ref{sparse}, and \ref{over-dispersed}, it holds that
\begin{align*}
\frac{(\hat \pi_i - \pi_i)(\hat \pi_j - \pi_j) }{\pi_j (\hat \pi_i - \pi_i) + \pi_i (\hat \pi_j - \pi_j)} = \mathcal{O}_P \left(\frac{1}{\sqrt{ \E d_i} + \sqrt{\E d_j}}\right).
\end{align*}
As a consequence, we may combine these two results to write
\begin{align}
\pij
= \pi_i \pi_j + \left[\pi_j (\hat \pi_i - \pi_i) + \pi_i (\hat \pi_j - \pi_j)\right] \cdot \left(1+\mathcal{O}_P\left(\frac{1}{\min_l \sqrt{\E d_l}}\right)\right). \label{12102015k}
\end{align}
Focusing on the rightmost sum in Eq.\ \eqref{12102015j}, we then obtain from Eq.\ \eqref{12102015k}
\begin{align*}
	&\sum_{j=1}^n \sum_{i<j} \pij \ingroup
	- \sum_{j=1}^n \sum_{i<j} \pi_i \pi_j \ingroup
	\\ & \quad =  \left[ \sum_{j=1}^n \sum_{i<j} \pi_j (\hat \pi_i - \pi_i) \ingroup + \sum_{j=1}^n \sum_{i<j}  \pi_i (\hat \pi_j - \pi_j) \ingroup  \right]
		\\ & \qquad \qquad \cdot \left(1+\mathcal{O}_P\left(\frac{1}{\min_l \sqrt{ \E d_l} }\right)\right).
		\intertext{Renaming the indices in the first summand from $i$ to $j$ and vice versa leads to}
	 & \quad=  \left[ \sum_{j=1}^n \sum_{i\neq j} \pi_i (\hat \pi_j - \pi_j) \ingroup \right] \cdot \left(1+\mathcal{O}_P\left(\frac{1}{\min_l \sqrt{ \E d_l} }\right)\right).
\end{align*}

Hence, $\sum_{j=1}^n \sum_{i<j} \pij \ingroup$ can be substituted into Eq.\ \eqref{12102015j} as follows:
\begin{align*}
	\widehat{Q}  & = \sum_{j=1}^n\sum_{i<j} A_{ij}\ingroup
 	- \sum_{j=1}^n \sum_{i<j} \pi_i \pi_j \ingroup
 \\& \qquad- \sum_{j=1}^n   \sum_{i\neq j} \pi_i  \left(\hat{\pi}_j - \pi_j\right) \ingroup \cdot \left(1+\mathcal{O}_P\left(\frac{1}{\min_l \sqrt{ \E d_l} }\right)\right).
 \intertext{We now change from a relative error term to an absolute error. In addition, we substitute $\sum_{j=1}^n\sum_{i<j} A_{ij}\ingroup = \frac{1}{2}\sum_{j=1}^n d_j^w$, \; $\hat{\pi}_j = d_j / \sqrt{\left\|\bm{d}\right\|_1}$ and $\sum_{i\neq j} \pi_i \ingroup = \ingroupPara$:}
 & = \frac{1}{2}\sum_{j=1}^n d_j^w - \sum_{j=1}^n \sum_{i<j} \pi_i \pi_j \ingroup
\\
& \qquad - \left[\sum_{j=1}^n   \ingroupPara  \left(\frac{d_j}{\sqrt{\left\| \bm{d} \right\|_1}} \right)
- \sum_{j=1}^n   \sum_{i\neq j} \pi_i  \pi_j \ingroup \right]
\\ & \qquad + \mathcal{O}_P\left(\frac{1}{\min_l \sqrt{ \E d_l} }\sum_{j=1}^n   \ingroupPara  \left(\hat{\pi}_j - \pi_j\right)\right).
\end{align*}

We will show in Step \ref{errorterm} below that
\begin{align} \label{IntroErrorMod}
	\frac{1}{\min_l \sqrt{ \E d_l} }\sum_{j=1}^n   \ingroupPara  \left(\hat{\pi}_j - \pi_j\right) = \mathcal{O}_P\left(\errorModtotal\right),
\end{align}
where $\errorModtotal$ is the error term defined in Eq.\ \eqref{errorModTotal}. Thus,
\begin{align} \label{15122015b}
\widehat{Q} & = \frac{1}{2}\sum_{j=1}^n d_j^w + \sum_{j=1}^n \sum_{i<j} \pi_i \pi_j \ingroup
 - \sum_{j=1}^n   \ingroupPara  \frac{d_j}{\sqrt{\left\| \bm{d} \right\|_1}}   + \mathcal{O}_P\left(\errorModtotal\right).
\end{align}

\underline{Step \ref{Taylord1noem}:}
In this step we focus on the penultimate term in Eq.\ \eqref{15122015b}.
We appeal to a Taylor expansion of $\left(\left\| \bm{d} \right\|_1/ \E \left\| \bm{d} \right\|_1\right)^{-1/2} = f(x)=x^{-1/2}$ at 1, and then control the remainder using Chebyshev's inequality. As a consequence, we obtain from Assumption \ref{over-dispersed} $\left(\V A_{ij} = \Theta\left(\E A_{ij}\right)\right)$ that
\begin{align}
\begin{split}
&\sum_{j=1}^n   \ingroupPara  \frac{d_j}{\sqrt{\left\| \bm{d} \right\|_1}}
\\ & =  \sum_{j=1}^n   \ingroupPara  \frac{d_j}{\sqrt{\E \left\| \bm{d} \right\|_1}}
 \cdot
 \left[1 - \frac{1}{2}\left(\frac{ \left\| \bm{d} \right\|_1}{\E \left\| \bm{d} \right\|_1} - 1 \right) + \mathcal{O}_P\left(\frac{1}{\E \left\| \bm{d} \right\|_1}\right)\right] .
 \end{split} \label{10122015b}
 \end{align}
We will show in Step \ref{errorterm} below that
\begin{align} \label{seconderror}
\sum_{j=1}^n   \ingroupPara  \frac{d_j}{\sqrt{\E \left\| \bm{d} \right\|_1}} \cdot \frac{1}{\E \left\| \bm{d} \right\|_1} = \mathcal{O}_P\left(\errorModtotal\right).
\end{align}
Continuing Eq.\ \eqref{10122015b}, we have that
 \begin{align} \label{15122015c}
 & =  \sum_{j=1}^n   \ingroupPara  \frac{d_j}{\sqrt{\E \left\| \bm{d} \right\|_1}}
 - \frac{1}{2} \sum_{j=1}^n   \ingroupPara  \frac{d_j}{\sqrt{\E \left\| \bm{d} \right\|_1}} \left(\frac{ \left\| \bm{d} \right\|_1}{\E \left\| \bm{d} \right\|_1} - 1 \right) + \mathcal{O}_P\left(\errorModtotal\right).
\end{align}

\underline{Step \ref{Taylordj}:} From Chebyshev's inequality and Assumption \ref{over-dispersed}, we know that $d_j = \E d_j \left[1+ \mathcal{O}_P \left(1/\sqrt{\E d_j}\right)\right]$. Inserting this result into the second (i.e., lower-order) term of the Taylor expansion in Eq.\ \eqref{15122015c}, we obtain
\begin{align}
\begin{split}
 &= \sum_{j=1}^n   \ingroupPara  \frac{d_j}{\sqrt{\E \left\| \bm{d} \right\|_1}}
\\ & \quad  - \frac{1}{2} \sum_{j=1}^n   \ingroupPara  \frac{\E d_j}{\sqrt{\E \left\| \bm{d} \right\|_1}} \left(\frac{ \left\| \bm{d} \right\|_1}{\E \left\| \bm{d} \right\|_1} - 1 \right) \left[1+ \mathcal{O}_P\left(\frac{1}{\sqrt{\E d_j}}\right)\right] + \mathcal{O}_P\left(\errorModtotal\right).
\end{split} \label{17112015}
\end{align}
Applying Chebyshev's inequality and then Assumption \ref{over-dispersed}, we next obtain
\begin{align} \notag
&\frac{1}{2} \sum_{j=1}^n   \ingroupPara  \frac{\E d_j }{\sqrt{\E \left\| \bm{d} \right\|_1}} \left(\frac{ \left\| \bm{d} \right\|_1}{\E \left\| \bm{d} \right\|_1} - 1 \right) \frac{ 1 }{ \sqrt{\E d_j} } \label{errorthree}
\\ & = \mathcal{O}_P \left(\errorModtotal\right). \quad \text{(Step \ref{errorterm} below)}
\end{align}
Applying Eq.\ \eqref{errorthree} and then substituting $\sum_{j=1}^n d_j$ for $\left\| \bm{d} \right\|_1$ in Eq.\ \eqref{17112015}, we have
\begin{align}
\begin{split}  \label{17112015b}
\sum_{j=1}^n   \ingroupPara  \frac{d_j}{\sqrt{\left\| \bm{d} \right\|_1}}
& =  \frac{1}{2} \sum_{j=1}^n   \ingroupPara  \frac{\E d_j}{\sqrt{\E \left\| \bm{d} \right\|_1}}
\\ & \quad - \sum_{j=1}^n  \left[ \frac{1}{2} \sum_{l=1}^n   \ingroupParaL  \frac{\E d_l}{\E \left\| \bm{d} \right\|_1} -  \ingroupPara\right]
\frac{d_j}{\sqrt{\E \left\| \bm{d} \right\|_1}}
\\& \quad + \mathcal{O}_P \left(\errorModtotal\right).
\end{split}
\intertext{\underline{Step \ref{centerdj}:}
Applying $d_i = d_i^w + d_i^b$ leads to  the identity}
\begin{split}
  & =  \frac{1}{2} \sum_{j=1}^n   \ingroupPara  \frac{\E d_j}{\sqrt{\E \left\| \bm{d} \right\|_1}} + \mathcal{O}_P \left(\errorModtotal\right)
\\& \quad - \sum_{j=1}^n  \left[ \frac{1}{2} \sum_{l=1}^n   \ingroupParaL  \frac{\E d_l}{\E \left\| \bm{d} \right\|_1} -  \ingroupPara\right]
 \frac{d_j^w}{\sqrt{\E \left\| \bm{d} \right\|_1 }}
\\& \quad - \sum_{j=1}^n  \left[ \frac{1}{2} \sum_{l=1}^n   \ingroupParaL  \frac{\E d_l}{\E \left\| \bm{d} \right\|_1} -  \ingroupPara\right]
\frac{d_j^b}{\sqrt{\E \left\| \bm{d} \right\|_1}}.
\end{split} \label{15122015d}	
\end{align}
We define non-random factors $\beta_j$ and $\alpha_j$ as in Eqs.\ \eqref{beta} and~\eqref{alpha}; i.e.,
\begin{align*}
 \beta_j & = \left[ \frac{1}{2} \sum_{l=1}^n   \ingroupParaL  \frac{\E d_l}{\E \left\| \bm{d} \right\|_1} -  \ingroupPara\right]  \frac{1}{\sqrt{\E \left\| \bm{d} \right\|_1}}
 \quad \text{and} \quad
 \alpha_j  = \frac{1}{2}  + \beta_j.
\end{align*}
Combining the results from Eqs.\ \eqref{15122015b} and~\eqref{15122015d}, we may rewrite $\widehat{Q}$ in terms of $\alpha_j$ and $\beta_j$ as
\begin{align*}
	\widehat{Q} & = \sum_{i<j} \pi_i \pi_j \ingroup
- \frac{1}{2} \sum_{j=1}^n   \ingroupPara  \frac{\E d_j}{\sqrt{\E \left\| \bm{d} \right\|_1}}
+ \sum_{j=1}^n \alpha_j d_j^w
+ \sum_{j=1}^n  \beta_j d_j^b + \mathcal{O}_P \left(\errorModtotal\right).
\end{align*}

After centering $d_j^w$ and $d_j^b$ about their respective means, we obtain
\begin{align} \label{12112015d}
\begin{split}
	\widehat{Q}  & = \sum_{j=1}^n \alpha_j \left[d_j^w - \E d_j^w \right]
 + \sum_{j=1}^n  \beta_j \left[d_j^b-\E d_j^b\right]
 + \sum_{j=1}^n \alpha_j \E d_j^w
 + \sum_{j=1}^n \beta_j \E d_j^b
\\ & \quad
+ \sum_{i<j} \pi_i \pi_j \ingroup
- \frac{1}{2} \sum_{j=1}^n   \ingroupPara  \frac{\E d_j}{\sqrt{\E \left\| \bm{d} \right\|_1}} + 			
		\mathcal{O}_P \left(\errorModtotal\right).
\end{split}
\end{align}

\underline{Step \ref{nonrandomtems}}
We now address the non-random terms in modularity. We treat the non-random terms in the two lines of  Eq.\ \eqref{12112015d} separately; i.e.,
\begin{itemize}
	\item [a)] $\sum_{j=1}^n \alpha_j \E d_j^w
 + \sum_{j=1}^n \beta_j \E d_j^b$;
	\item [b)] $\sum_{i<j} \pi_i \pi_j \ingroup
- \frac{1}{2} \sum_{j=1}^n   \ingroupPara  \frac{\E d_j}{\sqrt{\E \left\| \bm{d} \right\|_1}}$.\\
\end{itemize}

\underline{Term a) :}

From the definition of $\alpha_j$ and $\beta_j$, we obtain
\begin{align} \label{12112015b}
	 a) &
	  = \frac{1}{2} \sum_{j=1}^n  \E d_j^w + \sum_{j=1}^n \beta_j \E d_j  \notag
\\ & = \sum_{j=1}^n \sum_{i < j} \pi_i \pi_j \ingroup
			+  \sum_{j=1}^n \left[  \frac{1}{2} \sum_{l=1}^n
					 \ingroupParaL  \frac{\E d_l}{\E \left\| \bm{d} \right\|_1}
					 - \ingroupPara\right]
			 \frac{\E d_j }{\sqrt{\E \|\bm{d}\|_1} } \notag
\\			 & = \sum_{j=1}^n \sum_{i < j} \pi_i \pi_j \ingroup
+  \left[ \frac{1}{2}  \sum_{l=1}^n \ingroupParaL  \frac{\E d_l}{\sqrt{\E \|\bm{d}\|_1} }\right]
			 \frac{\sum_{j=1}^n \E d_j }{\E \left\| \bm{d} \right\|_1}
	\notag
\\ 	& \quad 			-  \sum_{j=1}^n \ingroupPara
			 \frac{\E d_j }{\sqrt{\E \|\bm{d}\|_1} }	 \notag
\\ & = \sum_{j=1}^n \sum_{i < j} \pi_i \pi_j \ingroup
			- \frac{1}{2}  \sum_{j=1}^n \ingroupPara
			 \frac{\E d_j }{\sqrt{\E \|\bm{d}\|_1} }	\notag
\\ & = b).			 		
\end{align}

\underline{Term b) :}

Via straightforward calculations, one can show that
\begin{align}
 b)
 & = \sum_{j=1}^n \sum_{i < j} \pi_i\pi_j \ingroup - \frac{1}{2} \sum_{j=1}^n \sum_{i \neq j} \frac{\pi_i\left(\pi_j \left\|\bm{\pi}\right\|_1 - \pi_j^2\right)}{\left\| \bm{\pi} \right\|_1 \sqrt{ 1 - \frac{ \left\| \bm{\pi} \right\|_2^2}{\left\| \bm{\pi} \right\|_1^2}} }  \ingroup \notag
\\ \label{10122015d} \begin{split}
 & = \sum_{j=1}^n \sum_{i < j} \pi_i\pi_j \ingroup
\\ & \qquad - \frac{1}{2} \sum_{j=1}^n \sum_{i \neq j} \left[\pi_i\pi_j
- \frac{\pi_i \pi_j^2}{\left\| \bm{\pi} \right\|_1 	 }  \right] \left( 1 - \frac{ \left\| \bm{\pi} \right\|_2^2}{\left\| \bm{\pi} \right\|_1^2}\right)^{-\frac{1}{2}} \ingroup.
\end{split}
\end{align}
We know from Eq.\ \eqref{twonormoveroneneormsqu} that from Assumption \ref{nonodeincontrol} it follows that $ \left\| \bm{\pi} \right\|_2^2 / \left\| \bm{\pi} \right\|_1^2 = \mathcal{O}\left(1/n\right)$. As a consequence, we can apply a convergent Taylor expansion to $f(x)=(1-x)^{-1/2}$ at 0 to obtain
\begin{align} \label{15122015e}
\left( 1 - \frac{ \left\| \bm{\pi} \right\|_2^2}{\left\| \bm{\pi} \right\|_1^2}\right)^{-\frac{1}{2}} = 1 + \frac{1}{2} \frac{ \left\| \bm{\pi} \right\|_2^2}{\left\| \bm{\pi} \right\|_1^2} + \mathcal{O}{\left[\left(\frac{ \left\| \bm{\pi} \right\|_2^2}{\left\| \bm{\pi} \right\|_1^2}\right)^2\right]}.
\end{align}
As a consequence, it follows that we may express Eq.\ \eqref{10122015d} as
\begin{align}
b) & = \sum_{j=1}^n \sum_{i < j} \pi_i\pi_j \ingroup - \frac{1}{2} \sum_{j=1}^n \sum_{i \neq j} \pi_i\pi_j \ingroup \notag
 \\ &  -  \sum_{j=1}^n \sum_{i < j} \Biggl[  \frac{1}{2} \pi_i\pi_j \frac{ \left\| \bm{\pi} \right\|_2^2}{\left\| \bm{\pi} \right\|_1^2}
+ \pi_i\pi_j \;  \mathcal{O}\Biggl[\left(\frac{ \left\| \bm{\pi} \right\|_2^2}{\left\| \bm{\pi} \right\|_1^2}\right)^2\Biggr]   \Biggr] \ingroup \label{17112015d}
\\
&  + \frac{1}{2} \sum_{j=1}^n \sum_{i \neq j} \Biggl[ \frac{\pi_i \pi_j^2}{\left\| \bm{\pi} \right\|_1 	 }  + \frac{1}{2} \frac{\pi_i \pi_j^2}{\left\| \bm{\pi} \right\|_1 	 }  \frac{ \left\| \bm{\pi} \right\|_2^2}{\left\| \bm{\pi} \right\|_1^2}
+  \frac{\pi_i \pi_j^2}{\left\| \bm{\pi} \right\|_1 	 }  \;  \mathcal{O}\Biggl[\left(\frac{ \left\| \bm{\pi} \right\|_2^2}{\left\| \bm{\pi} \right\|_1^2}\right)^2\Biggr]
   \Biggr] \ingroup. \label{17112015e}
\end{align}

We identify the first terms in Eqs.\ \eqref{17112015d} and~\eqref{17112015e} as the terms of leading order. We will show in Step \ref{errorterm} that the remaining terms satisfy
\begin{align}
\label{15122015f}
\begin{split}
 &-  \sum_{j=1}^n \sum_{i < j} \Biggl[ \pi_i\pi_j \;  \mathcal{O}\Biggl[\left(\frac{ \left\| \bm{\pi} \right\|_2^2}{\left\| \bm{\pi} \right\|_1^2}\right)^2\Biggr]   \Biggr] \ingroup
\\ & \quad + \frac{1}{2} \sum_{j=1}^n \sum_{i \neq j} \Biggl[ \frac{1}{2} \frac{\pi_i \pi_j^2}{\left\| \bm{\pi} \right\|_1 	 }  \frac{ \left\| \bm{\pi} \right\|_2^2}{\left\| \bm{\pi} \right\|_1^2}
+  \frac{\pi_i \pi_j^2}{\left\| \bm{\pi} \right\|_1 	 }  \;  \mathcal{O}\Biggl[\left(\frac{ \left\| \bm{\pi} \right\|_2^2}{\left\| \bm{\pi} \right\|_1^2}\right)^2\Biggr]
   \Biggr] \ingroup
  \\& =\mathcal{O}\left(\errorModtotal\right),
\end{split}
\end{align}
where we remind the reader that $\errorModtotal$ is the error term defined in Eq.\ \eqref{errorModTotal}.

Finally, considering the leading-order terms in Eqs.\ \eqref{17112015d} and~\eqref{17112015e}, it then follows from the identity
\begin{equation*}
\sum_{j=1}^n \sum_{i \neq j} \pi_i \pi_j^2 \ingroup = \sum_{j=1}^n \sum_{i < j} \pi_i \pi_j \left(\pi_i + \pi_j\right) \ingroup
\end{equation*}
 that
\begin{align}
b) & =  \frac{1}{2} \sum_{j=1}^n  \sum_{i<j} \pi_i \pi_j \left[ \frac{\pi_i + \pi_j}{\left\|\bm{\pi}\right\|_1}
- \frac{\left\|\bm{\pi}\right\|_2^2}{ \left\|\bm{\pi}\right\|_1^2} \right]\ingroup   +
   \mathcal{O}\left(\errorModtotal\right). \label{06102015}
\end{align}

We may then combine terms a) and b) using Eqs.\ \eqref{12112015b} and~\eqref{06102015}, whence
\begin{align}
a) + b) = &\sum_{j=1}^n \sum_{i<j} \pi_i \pi_j \left[\frac{\pi_i + \pi_j}{\left\|\bm{\pi}\right\|_1}
- \frac{\left\|\bm{\pi}\right\|_2^2}{ \left\|\bm{\pi}\right\|_1^2} \right]\ingroup +  \mathcal{O}\left(\errorModtotal\right). \notag
\end{align}
In order to gain interpretability, we rearrange the term $a) + b)$ even further:
\begin{align}
  \quad = & \sum_{j=1}^n \sum_{i<j} \E A_{ij} \left[\frac{\pi_i \left\|\bm{\pi}\right\|_1 + \pi_j \left\|\bm{\pi}\right\|_1 - \left\|\bm{\pi}\right\|_2^2}{\left\|\bm{\pi}\right\|_1^2}
 \right]\ingroup +  \mathcal{O}\left(\errorModtotal\right) \notag
 \\  \quad = &\sum_{j=1}^n \sum_{i<j} \E A_{ij} \left[\frac{\pi_i \left\|\bm{\pi}\right\|_1 + \pi_j \left\|\bm{\pi}\right\|_1 - \left\|\bm{\pi}\right\|_2^2}{\E \left\|\bm{d}\right\|_1}
 \right]\left[1 - \frac{\left\|\bm{\pi}\right\|_2^2}{\left\|\bm{\pi}\right\|_1^2}\right] \ingroup  +  \mathcal{O}\left(\errorModtotal\right) \notag
 \\  \quad = &\sum_{j=1}^n \sum_{i<j} \E A_{ij} \left[\frac{\pi_i \left\|\bm{\pi}\right\|_1 + \pi_j \left\|\bm{\pi}\right\|_1 - \left\|\bm{\pi}\right\|_2^2}{\E \left\|\bm{d}\right\|_1}
 \right] \ingroup  \notag
\\  & \quad - \sum_{j=1}^n \sum_{i<j} \E A_{ij} \left[\frac{\pi_i \left\|\bm{\pi}\right\|_1 + \pi_j \left\|\bm{\pi}\right\|_1  - \left\|\bm{\pi}\right\|_2^2}{\E \left\|\bm{d}\right\|_1}
 \right]\frac{\left\|\bm{\pi}\right\|_2^2}{\left\|\bm{\pi}\right\|_1^2} \ingroup  +  \mathcal{O}\left(\errorModtotal\right) \notag
 \\  \label{aPlusb} \begin{split} \quad = & \sum_{j=1}^n \sum_{i<j}  \frac{\E A_{ij} \left(\E d_i + \E d_j - \left\|\bm{\pi}\right\|_2^2\right)}{\E \left\|\bm{d}\right\|_1} \ingroup
 \\  & \quad + \sum_{j=1}^n \sum_{i<j}  \frac{\E A_{ij} \left(\pi_i + \pi_j \right)}{\E \left\|\bm{d}\right\|_1} \ingroup
\\  & \quad - \sum_{j=1}^n \sum_{i<j} \E A_{ij} \left[\frac{\pi_i \left\|\bm{\pi}\right\|_1 + \pi_j \left\|\bm{\pi}\right\|_1  - \left\|\bm{\pi}\right\|_2^2}{\E \left\|\bm{d}\right\|_1}
 \right]\frac{\left\|\bm{\pi}\right\|_2^2}{\left\|\bm{\pi}\right\|_1^2} \ingroup    + \mathcal{O}\left(\errorModtotal\right).
\end{split}
\intertext{We will show in Step \ref{errorterm} that}
 \begin{split} & \sum_{j=1}^n \sum_{i<j}  \frac{\E A_{ij} \left(\pi_i + \pi_j \right)}{\E \left\|\bm{d}\right\|_1} \ingroup
 \\ & \quad - \sum_{j=1}^n \sum_{i<j} \E A_{ij} \left[\frac{\pi_i \left\|\bm{\pi}\right\|_1 + \pi_j \left\|\bm{\pi}\right\|_1  - \left\|\bm{\pi}\right\|_2^2}{\E \left\|\bm{d}\right\|_1}
 \right]\frac{\left\|\bm{\pi}\right\|_2^2}{\left\|\bm{\pi}\right\|_1^2} \ingroup
 \\ & = \mathcal{O}\left(\errorModtotal\right).
 \end{split}
 \label{errorMod2}
 \end{align}
 Recall from the definition of $\muMod$ in Eq.\ \eqref{muMod} that
  \begin{align}
   \muMod = & \sum_{j=1}^n
\sum_{i<j}  \frac{\E A_{ij} \left(\E d_i + \E d_j - \left\|\bm{\pi}\right\|_2^2\right)}{\E \left\|\bm{d}\right\|_1} \ingroup. \notag
\end{align}
Then, as a consequence of Eqs.\ \eqref{aPlusb} and~\eqref{errorMod2}, we see that
\begin{align}
 a) + b)= & \muMod  + \mathcal{O}\left(\epsilon\right). \label{12112015c}
\end{align}
Inserting the results from Eq.\ \eqref{12112015c} into Eq.\ \eqref{12112015d} and under the assumption that all error terms are controlled (see Step \ref{errorterm} below), we obtain the result of this lemma; i.e.,
\begin{align} \label{ModPluError}
	\widehat{Q} & = 	\sum_{j=1}^n \alpha_j \left[d_j^w - \E d_j^w \right]
   - \sum_{j=1}^n  \beta_j \left[d_j^b -\E d_j^b\right]  + \muMod + \mathcal{O}\left(\epsilon\right).
\end{align}

\underline{Step \ref{errorterm}:} We now define and address the five error terms cited above; we call these $\errorMod, \errorModTwo, \ldots, \errorModFive$.\\

\underline{Term $\errorMod$:} Recalling Eq.\ \eqref{IntroErrorMod}, we define
\begin{align*}
\errorMod & = \frac{1}{\min_l \sqrt{ \E d_l} }\sum_{j=1}^n   \ingroupPara  \left(\hat{\pi}_j - \pi_j\right)
\\ &= \frac{1}{\min_l \sqrt{ \E d_l} }\sum_{j=1}^n   \ingroupPara  \left(\frac{d_j}{\sqrt{\left\| \bm{d} \right\|_1}} - \pi_j\right).
\end{align*}
First, we apply a Taylor expansion to $\left(\left\| \bm{d} \right\|_1/ \E \left\| \bm{d} \right\|_1\right)^{-1/2} = f(x)=x^{-1/2}$ at 1, leading to
\begin{align*}
	\frac{1}{\sqrt{\left\| \bm{d} \right\|_1}} =
	\frac{1}{\sqrt{\E \left\| \bm{d} \right\|_1}}
	\left[1+ \mathcal{O}_P\left(\sqrt{\frac{\V \left\| \bm{d} \right\|_1}{\left(\E \left\| \bm{d} \right\|_1\right)^2}}\right)\right],
\end{align*}
and then control the remainder using Chebyshev's inequality. As a consequence, we obtain from Assumption \ref{over-dispersed} $\left(\V A_{ij} = \Theta\left(\E A_{ij}\right)\right)$ that
\begin{align*}
 \errorMod & = \frac{1}{\min_l \sqrt{ \E d_l} } \sum_{j=1}^n   \ingroupPara  \left( \frac{d_j \left[1+ \mathcal{O}_P\left(1/\sqrt{\E \left\| \bm{d} \right\|_1}\right)\right]}{\sqrt{\E \left\| \bm{d} \right\|_1}}  - \pi_j\right).
\intertext{
From Chebyshev's inequality and Assumptions \ref{sparse} and \ref{over-dispersed}, we know that $d_j = \E d_j +\mathcal{O}_P\left(\sqrt{\E d_j}\right)= \E d_j \left[1+\mathcal{O}_P(1/\sqrt{\E d_j})\right]$.
It follows that}
& = \frac{1}{\min_l \sqrt{ \E d_l} } \sum_{j=1}^n   \ingroupPara  \left( \frac{\E d_j \left[1+ \mathcal{O}_P\left(1/\sqrt{\E d_j} \right)\right]}{\sqrt{\E \left\| \bm{d} \right\|_1}}  - \pi_j\right)
\\ & = \frac{1}{\min_l \sqrt{ \E d_l} } \sum_{j=1}^n   \ingroupPara  \left( \pi_j
\frac{ \left[1+ \mathcal{O}_P\left( 1/\sqrt{\E d_j} \right)\right]\left[1- \pi_j/ \left\|\bm{\pi}\right\|_1\right]}{ \left[1- \left\|\bm{\pi}\right\|_2^2/ \left\|\bm{\pi}\right\|_1^2\right]^{1/2}}
  - \pi_j\right).	
 \intertext{ Since $\left\|\bm{\pi}\right\|_2^2/ \left\|\bm{\pi}\right\|_1^2 = \mathcal{O}\left(1/n\right)$ (Eq.\ \eqref{twonormoveroneneormsqu}, following from Assumption \ref{nonodeincontrol}), we can apply a convergent Taylor expansion to $f(x)=(1-x)^{-1/2}$ at 0 (as in Eq.~\eqref{15122015e}). Furthermore, the remainder term $\left(\left\|\bm{\pi}\right\|_2^2/ \left\|\bm{\pi}\right\|_1^2\right)^2$ in this Taylor expansion satisfies $\left(\left\|\bm{\pi}\right\|_2^2/ \left\|\bm{\pi}\right\|_1^2\right)^2 = \mathcal{O}\left(1/n^2\right) = \mathcal{O}\left(1/\sqrt{\E d_j} \right)$ (Assumptions \ref{nonodeincontrol} and \ref{notcomplete}). Hence, we obtain}
  & = \frac{1}{\min_l \sqrt{ \E d_l} }\sum_{j=1}^n   \ingroupPara  \Biggl( \pi_j
  \left[1+ \mathcal{O}_P\left( \frac{1}{\sqrt{\E d_j}} \right)\right]
   \left[1 - \frac{\pi_j}{ \left\|\bm{\pi}\right\|_1}\right]
\\ & \qquad \qquad \qquad \qquad \qquad \qquad \qquad \left[1 + \frac{1}{2} \left(\frac{\left\|\bm{\pi}\right\|_2^2}{ \left\|\bm{\pi}\right\|_1^2}\right)
 \right] - \pi_j\Biggr)
  \end{align*}
 \begin{align}
 \notag   & = \frac{1}{\min_l \sqrt{ \E d_l} }\sum_{j=1}^n   \ingroupPara  \left( - \frac{\pi_j^2}{ \left\|\bm{\pi}\right\|_1} + \frac{1}{2} \pi_j \frac{\left\|\bm{\pi}\right\|_2^2}{ \left\|\bm{\pi}\right\|_1^2}
 \right) \left[1+ \mathcal{O}_P\left( \frac{1}{\sqrt{\E d_j}} \right)\right]
\\ \label{errorone}  & = \frac{1}{\min_l \sqrt{ \E d_l} }\sum_{j=1}^n   \sum_{i \neq j} \pi_i \pi_j  \left( -\frac{\pi_j}{ \left\|\bm{\pi}\right\|_1} + \frac{1}{2} \frac{\left\|\bm{\pi}\right\|_2^2}{ \left\|\bm{\pi}\right\|_1^2}
 \right) \ingroup \left[1+ \mathcal{O}_P\left( \frac{1}{\sqrt{\E d_j}} \right)\right]
 \\ \label{16122015a} & = \frac{1}{\min_l \sqrt{ \E d_l} }\sum_{j=1}^n   \sum_{i \neq j} \pi_i \pi_j  \ingroup \cdot \mathcal{O}_P\left(\frac{1}{n}\right).\quad \text{(Assumption \ref{nonodeincontrol}, Eq.\ \eqref{twonormoveroneneormsqu})}\\ \notag
\end{align}

\underline{Term $\errorModTwo$:}
We now analyze the second error term. Recalling Eq.\ \eqref{seconderror}, define
\begin{align}
 \notag \errorModTwo &= \sum_{j=1}^n   \ingroupPara  \frac{d_j}{\sqrt{\E \left\| \bm{d} \right\|_1}} \frac{1}{\E \left\| \bm{d} \right\|_1}
 \intertext{From Chebyshev's inequality and Assumption \ref{over-dispersed} it follows that}
 \label{errorModtwo} & = \sum_{j=1}^n   \ingroupPara  \frac{\E d_j}{\sqrt{\E \left\| \bm{d} \right\|_1}} \frac{1}{\E \left\| \bm{d} \right\|_1}\left(1 + \mathcal{O}_P\left(\frac{1}{\sqrt{\E d_j}}\right)\right)
\end{align}
This expression is smaller than $\errorModThree$ as defined in Eq.\ \eqref{09012016}.

\underline{Term $\errorModThree$:}
We now analyze the third error term. Recalling Eq.\ \eqref{errorthree}, define
\begin{align}
 \notag \errorModThree & =
\frac{1}{2} \sum_{j=1}^n   \ingroupPara  \frac{\E d_j }{\sqrt{\E \left\| \bm{d} \right\|_1}} \left(\frac{ \left\| \bm{d} \right\|_1}{\E \left\| \bm{d} \right\|_1} - 1 \right) \frac{ 1 }{ \sqrt{\E d_j} }
\intertext{Applying Chebyshev's inequality leads to} \notag &= \frac{1}{2} \sum_{j=1}^n   \ingroupPara  \frac{\E d_j }{\sqrt{\E \left\| \bm{d} \right\|_1}} \cdot \mathcal{O}_P\left( \frac{ 1 }{ \sqrt{\E \left\| \bm{d} \right\|_1} } \right) \cdot \frac{1}{ \sqrt{\E d_j}}
\\ \label{09012016} & = \sum_{j=1}^n   \ingroupPara  \mathcal{O}_P\left(\frac{\E d_j}{\E \left\| \bm{d} \right\|_1 \sqrt{\E d_j} } \right)
\\ \notag & = \sum_{j=1}^n   \ingroupPara \mathcal{O}_P\left( \frac{ \sqrt{\E d_j}}{\E \left\| \bm{d} \right\|_1 }\right)
\\ \notag &= \sum_{j=1}^n   \ingroupPara  \pi_j \sqrt{\frac{ \pi_j \left\| \bm{\pi} \right\|_1}{\pi_j^2 \left\| \bm{\pi} \right\|_1^4 }}
\mathcal{O}_P\left(\sqrt{\frac{1- \pi_j/\left\| \bm{\pi} \right\|_1}{1-\left\| \bm{\pi} \right\|_2^2/\left\| \bm{\pi} \right\|_1^2}}\right)
\\ \notag & = \sum_{j=1}^n   \ingroupPara  \pi_j \sqrt{\frac{ \pi_j \left\| \bm{\pi} \right\|_1}{\pi_j^2 \left\| \bm{\pi} \right\|_1^4 }} \mathcal{O}_P\left(\sqrt{1 + \frac{1}{n}}\right)\quad \text{(Assumption \ref{nonodeincontrol}, Eqs.\ \eqref{twonormoveroneneormsqu},~\eqref{15122015e})}
\end{align}
\begin{align}
 \notag & = \sum_{j=1}^n   \ingroupPara  \pi_j \sqrt{\frac{ 1}{\pi_j \left\| \bm{\pi} \right\|_1 \left\| \bm{\pi} \right\|_1^2 }}
\mathcal{O}_P\left(\sqrt{1 + \frac{1}{n}}\right)
\\ \notag & = \sum_{j=1}^n   \ingroupPara  \pi_j \sqrt{\frac{1 - \pi_j/\left\| \bm{\pi} \right\|_1 }{\E d_j \left\| \bm{\pi} \right\|_1^2 }}
\mathcal{O}_P\left(\sqrt{1 + \frac{1}{n}}\right)
\\ \label{16122015c} & = \frac{\sum_{j=1} \sum_{i<j} \pi_i \pi_j \ingroup }{ \min_l \sqrt{\E d_l} \; \left\| \bm{\pi} \right\|_1} \mathcal{O}_P\left(\sqrt{1 + \frac{1}{n}}\right). \quad \text{(Assumption \ref{nonodeincontrol})}
\end{align}

\underline{Term $\errorModFour$:}
We now analyze the fourth error term. Recalling Eq.\ \eqref{15122015f}, define
\begin{align}
\errorModFour &=  -  \sum_{j=1}^n \sum_{i < j} \Biggl[ \pi_i\pi_j \;  \mathcal{O}\Biggl[\left(\frac{ \left\| \bm{\pi} \right\|_2^2}{\left\| \bm{\pi} \right\|_1^2}\right)^2\Biggr]   \Biggr] \ingroup
\\ & \quad + \frac{1}{2} \sum_{j=1}^n \sum_{i \neq j} \Biggl[ \frac{1}{2} \frac{\pi_i \pi_j^2}{\left\| \bm{\pi} \right\|_1 	 }  \frac{ \left\| \bm{\pi} \right\|_2^2}{\left\| \bm{\pi} \right\|_1^2}
+  \frac{\pi_i \pi_j^2}{\left\| \bm{\pi} \right\|_1 	 }  \;  \mathcal{O}\Biggl[\left(\frac{ \left\| \bm{\pi} \right\|_2^2}{\left\| \bm{\pi} \right\|_1^2}\right)^2\Biggr]
   \Biggr] \ingroup
\\ \label{16012016} \begin{split} &=
- \mathcal{O}\left[ \left(\frac{ \left\| \bm{\pi} \right\|_2^2}{\left\| \bm{\pi} \right\|_1^2}\right)^2 \right]
 \sum_{j=1}^n \sum_{i < j}  \pi_i\pi_j  \ingroup
\\  & \quad + \mathcal{O}\left[ \frac{ \left\| \bm{\pi} \right\|_2^2}{\left\| \bm{\pi} \right\|_1^2} \right] \sum_{j=1}^n \sum_{i \neq j} \pi_i \pi_j \frac{ \pi_j}{\left\| \bm{\pi} \right\|_1 	 }  \ingroup \end{split}
\\ \label{16122015d} & = \mathcal{O}\left(\frac{1}{n^2}\right) \sum_{j=1}^n \sum_{i<j} \pi_i \pi_j \ingroup. \quad \text{(Assumption \ref{nonodeincontrol}, Eq.\ \eqref{twonormoveroneneormsqu})}
\end{align}

\underline{Term $\errorModFive$:}
We now analyze the fifth error term. Recalling Eq.\ \eqref{errorMod2}, define
\begin{align}
 \begin{split} \label{19012016b} \errorModFive = &
 \sum_{j=1}^n \sum_{i<j}  \frac{\E A_{ij} \left(\pi_i + \pi_j \right)}{\E \left\|\bm{d}\right\|_1} \ingroup
 \\ & \quad - \sum_{j=1}^n \sum_{i<j} \E A_{ij} \left[\frac{\pi_i \left\|\bm{\pi}\right\|_1 + \pi_j \left\|\bm{\pi}\right\|_1  - \left\|\bm{\pi}\right\|_2^2}{\E \left\|\bm{d}\right\|_1}
 \right]\frac{\left\|\bm{\pi}\right\|_2^2}{\left\|\bm{\pi}\right\|_1^2} \ingroup
 \end{split}
 \\ \notag = &\sum_{j=1}^n \sum_{i<j}  \frac{\E A_{ij} \left(\pi_i + \pi_j \right)}{\E \left\|\bm{d}\right\|_1} \ingroup \left[1-\frac{\left\|\bm{\pi}\right\|_2^2}{\left\|\bm{\pi}\right\|_1}\right]
 \\ \notag & \qquad + \sum_{j=1}^n \sum_{i<j} \frac{\E A_{ij}}{\E \left\|\bm{d}\right\|_1} \ingroup
 \left(\frac{\left\|\bm{\pi}\right\|_2^2}{\left\|\bm{\pi}\right\|_1} \right)^2
\\ \notag \leq &  \frac{2 \max_l \pi_l}{\E \left\|\bm{d}\right\|_1} \sum_{j=1}^n \sum_{i<j} \E A_{ij} \ingroup
\left[1 + \mathcal{O}\left(\max_l \pi_l\right)\right]
\quad \text{(Assumption \ref{nonodeincontrol}, Eq.\ \eqref{twonormoveroneneormsqu})}
\\ \notag & \qquad + \sum_{j=1}^n \sum_{i<j} \frac{\E A_{ij}}{\E \left\|\bm{d}\right\|_1} \ingroup
 \left(\max_l \pi_l\right)^2
\end{align}
\begin{align}
 \notag = &  \frac{2 \max_l \pi_l + \mathcal{O}\left(\max_l \pi_l^2\right)}{ \left\|\bm{\pi}\right\|_1^2}
\left[1-\frac{\left\|\bm{\pi}\right\|_2^2}{\left\|\bm{\pi}\right\|_1^2}\right]^{-1}
\sum_{j=1}^n \sum_{i<j} \E A_{ij} \ingroup.
\intertext{Applying a convergent Taylor expansion to $f(x)= (1-x)^{-1}$ at 0 with $x = \left\|\bm{\pi}\right\|_2^2/\left\|\bm{\pi}\right\|_1^2$ (Assumption \ref{nonodeincontrol} and Eq.\ \eqref{twonormoveroneneormsqu}), we obtain}
 \notag = &  \frac{2 \max_l \pi_l + \mathcal{O}\left(\max_l \pi_l^2\right)}{\left\|\bm{\pi}\right\|_1^2} \left[1+\mathcal{O}\left(\frac{1}{n}\right)\right] \sum_{j=1}^n \sum_{i<j} \E A_{ij} \ingroup
\\ \label{16122015e} = & \mathcal{O} \left( \frac{1}{n \, \left\|\bm{\pi}\right\|_1} + \frac{1}{n^2}\right) \sum_{j=1}^n \sum_{i<j} \pi_i \pi_j \ingroup. \quad \text{(Assumption \ref{nonodeincontrol})}
\end{align}

As a consequence of Eqs.\ \eqref{16122015a}--\eqref{16122015e}, we now know that the error terms $\errorMod, \errorModTwo, \ldots, \errorModFive$ in our analysis of modularity satisfy
\begin{align*}
 &\errorMod + \errorModTwo + \errorModThree + \errorModFour + \errorModFive
\\ & = \mathcal{O}_P \left( \frac{1}{n \, \min_l \sqrt{\E d_l}} +\frac{1}{\left\|\bm{\pi}\right\|_1 \,\min_l \sqrt{\E d_l} } + \frac{1}{n^2} +\frac{1}{\left\|\bm{\pi}\right\|_1 n} \right) \sum_{j=1}^n \sum_{i<j} \pi_i \pi_j \ingroup.
\intertext{From Assumption \ref{notcomplete}, it follows that $\min_l \sqrt{\E d_l} = o\left(\sqrt{n^2}\right) = o(n)$. Hence, }
 & = \mathcal{O}_P \left( \frac{1}{n \, \min_l \sqrt{\E d_l}} +\frac{1}{\left\|\bm{\pi}\right\|_1 \,\min_l \sqrt{\E d_l} } \right) \sum_{j=1}^n \sum_{i<j} \pi_i \pi_j \ingroup.
\end{align*}
 Recall from Eq.\ \eqref{errorModTotal} that
\begin{align*}
	 \errorModtotal & = \frac{\sum_{j=1}^n \sum_{i<j} \pi_i \pi_j \ingroup}{\min\left(n, \left\|\bm{\pi}\right\|_1\right) \, \min_l \sqrt{\E d_l} }.
\end{align*}
It follows that
\begin{align*}
 &\errorMod + \errorModTwo + \errorModThree + \errorModFour + \errorModFive
  = \mathcal{O}_P\left(\epsilon\right).
\end{align*}

As a consequence, we conclude the required result of Lemma \ref{Modularity_reform}; i.e.,
\begin{align*}
 		\widehat{Q} &= \muMod + \left(\sum_{j=1}^n \alpha_j \left[d_j^w - \E d_j^w\right]  + \sum_{j=1}^n \beta_j \left[d_j^b - \E d_j^b\right]\right)  + \mathcal{O}_P\left(\epsilon\right).
\end{align*}
\end{proof}
We now derive the asymptotic distribution of modularity $\widehat{Q}$. Recalling the definitions of $\alpha,\beta$ in Eqs.\ \eqref{beta}, \eqref{alpha}, we define a sequence of random variables via
\begin{align} \label{mod_reform_slutsky}
 	X_n= \sum_{j=1}^n \alpha_j \left[d_j^w - \E d_j^w\right]  + \sum_{j=1}^n \beta_j \left[d_j^b - \E d_j^b\right].
\end{align}
In Lemma \ref{billinslayformod} below we show the asymptotic behavior of $X_n$. The Lemma parallels Lemma~2 in the main text.
\begin{appxlemma} \label{billinslayformod}
 		Consider Assumptions \ref{nonodeincontrol}--\ref{skewed}, and suppose that the number $K$ of communities grows strictly more slowly than $n$, so that $K/n \rightarrow 0$. Then, as $n \rightarrow \infty$,
\begin{align*}
\left(\V X_n\right) ^{-\frac{1}{2}} X_n \stackrel{d}{\rightarrow} \No(0,1).
\end{align*}
\end{appxlemma}

\begin{proof}
First we write $X_n$ as a sum of independent, zero-mean random variables:
\begin{align}
	X_n &= \sum_{j=1}^n \alpha_j \left[d_j^w - \E d_j^w\right]  + \sum_{j=1}^n \beta_j \left[d_j^b - \E d_j^b\right] \notag
\\	&= \sum_{j=1}^n \sum_{i \neq j} \alpha_j \left[A_{ij} - \E A_{ij}\right] \ingroup  + \sum_{j=1}^n \sum_{i \neq j} \beta_j \left[A_{ij} - \E A_{ij}\right] \NOTingroup \notag
\\	&= \sum_{j=1}^n \sum_{i < j} \left(\alpha_i + \alpha_j\right) \left[A_{ij} - \E A_{ij}\right] \ingroup  \notag
\\ & \qquad + \sum_{j=1}^n \sum_{i < j} \left(\beta_i + \beta_j\right) \left[A_{ij} - \E A_{ij}\right] \NOTingroup \notag
\\	&= \sum_{j=1}^n \sum_{i < j} \left[\left(\alpha_i + \alpha_j\right)  \ingroup  +  \left(\beta_i + \beta_j\right)  \NOTingroup \right] \left[A_{ij} - \E A_{ij}\right] \notag
\\	&= \sum_{j=1}^n \sum_{i < j} \left[\left(1+ \beta_i + \beta_j\right)  \ingroup  +  \left(\beta_i + \beta_j\right)  \NOTingroup \right] \left[A_{ij} - \E A_{ij}\right] \notag
\\	&= \sum_{j=1}^n \sum_{i < j} \underbrace{\left[  \ingroup  +  \beta_i + \beta_j \right]}_{c_{ij}} \left[A_{ij} - \E A_{ij}\right] \label{Xn}
\\	&= \sum_{j=1}^n \sum_{i < j} c_{ij} \left[A_{ij} - \E A_{ij}\right] \label{19012016}.
\end{align}
To apply the Lindeberg--Feller Central Limit Theorem to this sum, we show:
\begin{enumerate}
	\item $\V \left(c_{ij} A_{ij}\right) < \infty$; \label{one}
	\item The Lyapunov condition for exponent $1$ is satisfied; i.e., \label{two}
				\begin{align*}
					\frac{\sum_{j = 1}^n \sum_{i<j} \E \left[\left( c_{ij} A_{ij} - \E \left(c_{ij} A_{ij}\right)\right)^3\right]}{\left[\V \left(\sum_{j = 1}^n \sum_{i<j} c_{ij} A_{ij}\right)\right]^{3/2}} \rightarrow 0.
				\end{align*}
\end{enumerate}
Since both conditions are strongly influenced by $c_{ij}$, we first show that $c_{ij} = \mathcal{O}\left(1\right)$. From Eq.\ \eqref{Xn} and the definitions of $\alpha,\beta$ in Eqs.\ \eqref{beta}, \eqref{alpha}, we see that
\begin{align*}
	c_{ij} & - \ingroup
	\\     & =  \left[ \sum_{l=1}^n \ingroupParaL  \frac{\E d_l}{\E \left\| \bm{d} \right\|_1} - \ingroupPara -\ingroupParaI \right] \frac{1}{\sqrt{\E \left\| \bm{d}\right\|_1}}
	\\     & =  \left[ \sum_{l=1}^n \frac{ \left( \left\|\bm{\pi}\right\|_1^{g(l),\emptyset} - \pi_l \right) \pi_l }{ \left\| \bm{\pi} \right\|_1} \left( \frac{ 1 - \frac{ \pi_l}{ \left\| \bm{\pi} \right\|_1} }{ 1 - \frac{ \left\| \bm{\pi} \right\|_2^2}{\left\| \bm{\pi} \right\|_1^2} } \right) - \ingroupPara -\ingroupParaI \right] \frac{\left( 1 - \frac{ \left\| \bm{\pi} \right\|_2^2}{\left\| \bm{\pi} \right\|_1^2}\right)^{-\frac{1}{2}}}{\left\| \bm{\pi} \right\|_1}.
\end{align*}
From Assumption \ref{nonodeincontrol} and Eq.\ \eqref{08102015}, we know that $\left\| \bm{\pi} \right\|_2^2/\left\| \bm{\pi} \right\|_1^2 \leq \max_{i} \pi_i \left\| \bm{\pi} \right\|_1 / \left\| \bm{\pi} \right\|_1^2 = \mathcal{O}\left(1/n\right)$. Hence, we can apply a convergent Taylor expansion to $f(x)=(1-x)^{-\alpha}, \alpha = 1/2, 1$ at $x=0$. We obtain
\begin{align}	
     & = \left[ \frac{ \sum_{k=1}^K \left( \left\|\bm{\pi}\right\|_1^{k,\emptyset} \right)^2 - \left\| \bm{\pi} \right\|_2^2 }{\left\| \bm{\pi} \right\|_1} - \ingroupPara -\ingroupParaI \right] \frac{ \left[ 1 + \mathcal{O}\left( \frac{ \max_i \pi_i}{ \left\| \bm{\pi} \right\|_1} \right) \right] }{\left\| \bm{\pi} \right\|_1} \label{05012016}
\\    & = \left[ \frac{ \sum_{k=1}^K \left( \left\|\bm{\pi}\right\|_1^{k,\emptyset} \right)^2}{\left\| \bm{\pi} \right\|_1} - \left\|\bm{\pi}\right\|_1^{g(j),\emptyset} -\left\|\bm{\pi}\right\|_1^{g(i),\emptyset} \right]
\frac{ \left[ 1 + \mathcal{O}\left( \frac{ \max_i \pi_i}{ \left\| \bm{\pi} \right\|_1} \right) \right] }{\left\| \bm{\pi} \right\|_1} \notag
\\ & \quad + \left[ \frac{\pi_j}{\left\| \bm{\pi} \right\|_1} + \frac{\pi_i}{\left\| \bm{\pi} \right\|_1} - \frac{ \left\| \bm{\pi} \right\|_2^2 }{\left\| \bm{\pi} \right\|_1^2} \right]  \left[ 1 + \mathcal{O}\left( \frac{ \max_i \pi_i}{ \left\| \bm{\pi} \right\|_1} \right) \right] \notag
\intertext{Since $\left\| \bm{\pi} \right\|_2^2/\left\| \bm{\pi} \right\|_1^2 \leq \max_{i} \pi_i \left\| \bm{\pi} \right\|_1 / \left\| \bm{\pi} \right\|_1^2 = \mathcal{O}\left(1/n\right)$, it follows further that}
    & = \left[ \frac{ \sum_{k=1}^K \left( \left\|\bm{\pi}\right\|_1^{k,\emptyset} \right)^2}{\left\| \bm{\pi} \right\|_1^2} - \frac{\left\|\bm{\pi}\right\|_1^{g(j),\emptyset}}{\left\| \bm{\pi} \right\|_1} - \frac{\left\|\bm{\pi}\right\|_1^{g(i),\emptyset}}{\left\| \bm{\pi} \right\|_1} \right]
\left[ 1 + \mathcal{O}\left( \frac{ 1}{ n} \right) \right]
 + \mathcal{O}\left( \frac{ 1}{n}\right). \label{17122015a}
\end{align}
The first term in Eq.\ \eqref{17122015a} is $\mathcal{O}(1)$, and thus we conclude $c_{ij} = \mathcal{O}(1)$. This in turn allows us to combine the relative and additive error terms. Furthermore we see that $c_{ij}$ is, up to an additive error term of order at most $1/n$, a function only of $g(i)$ and $g(j)$:
\begin{equation}
 c_{ij} = \ingroup + \sum_{k=1}^K \left( \frac{ \left\|\bm{\pi}\right\|_1^{k,\emptyset} }{ \left\| \bm{\pi} \right\|_1 } \right)^2 - \frac{ \left\|\bm{\pi}\right\|_1^{g(i),\emptyset} }{ \left\| \bm{\pi} \right\|_1 } - \frac{ \left\|\bm{\pi}\right\|_1^{g(j),\emptyset} }{ \left\| \bm{\pi} \right\|_1 } + \mathcal{O}\left( \frac{1}{n} \right). \label{30112015aabb}
\end{equation}

  We are now ready to address the two conditions sufficient for the Lindeberg-Feller Central Limit Theorem.

\underline{Condition \ref{one}:}
\begin{align*}
	\V \left(c_{ij} A_{ij}\right) &= c_{ij}^2 \V \left(A_{ij}\right)
	\\ &= c_{ij}^2 \Theta \left(\pi_i \pi_j\right) \quad \text{(Assumption \ref{over-dispersed})}
	\\ &< \infty.  \quad  \left(\text{Eq.\ \eqref{30112015aabb}: } c_{ij} = \mathcal{O}(1); \; \pi_i, \pi_j \in \R_{>0} \right)
\end{align*}

\underline{Condition \ref{two}:}
\begin{align*}
	&\frac{\sum_{j = 1}^n \sum_{i<j} \E \left[\left( c_{ij} A_{ij} - \E \left(c_{ij} A_{ij}\right)\right)^3\right]}{\left[\V \left(\sum_{j = 1}^n \sum_{i<j} c_{ij} A_{ij}\right)\right]^{3/2}}
\\ &	= \frac{\sum_{j = 1}^n \sum_{i<j} c_{ij}^3 \E \left[\left(  A_{ij} - \E \left( A_{ij}\right)\right)^3\right]}{\left[
\sum_{j = 1}^n \sum_{i<j} c_{ij}^2 \V  A_{ij}\right]^{3/2}}
\\ &	= \mathcal{O}\left(1\right) \cdot \frac{\sum_{j = 1}^n \sum_{i<j} c_{ij}^2 \E \left[\left(  A_{ij} - \E \left( A_{ij}\right)\right)^3\right]}{\left[
\sum_{j = 1}^n \sum_{i<j} c_{ij}^2 \V  A_{ij}\right]^{3/2}} \quad  \left(\text{Eq.\ \eqref{30112015aabb}: } c_{ij} = \mathcal{O}(1) \right)
\\ &	= \mathcal{O}\left(\frac{\sum_{j = 1}^n \sum_{i<j} c_{ij}^2 \V A_{ij} }{\left[
\sum_{j = 1}^n \sum_{i<j} c_{ij}^2 \V  A_{ij}\right]^{3/2}}\right) \quad \text{(Assumption \ref{skewed})}
\\ &	= \mathcal{O} \left(\frac{1 }{\left[
\sum_{j = 1}^n \sum_{i<j} c_{ij}^2 \V  A_{ij}\right]^{1/2}}\right).   \left(\text{Eq.\ \eqref{30112015aabb}: } c_{ij} = \mathcal{O}(1) \right).
\end{align*}

For Condition \ref{two}, it remains to show that $ \sum_{j = 1}^n \sum_{i<j} c_{ij}^2 \V  A_{ij} \rightarrow \infty$:
\begin{align}
 \notag	\sum_{j = 1}^n & \sum_{i<j} c_{ij}^2 \V  A_{ij}
 = \sum_{j = 1}^n \sum_{i<j} c_{ij}^2 \Theta\left( \pi_i \pi_j\right) \quad \text{(Assumption \ref{over-dispersed})}
 \\ \notag &= \frac{1}{2} \left[ \sum_{i = 1}^n \sum_{j = 1}^n c_{ij}^2 \Theta\left( \pi_i \pi_j\right) - \sum_{i = 1}^n c_{ii}^2 \Theta\left( \pi_i^2 \right) \right]
 \\ &= \frac{1}{2} \left[ \sum_{k = 1}^K \sum_{t = 1}^K c_{tk}^2 \Theta\left( \left\|\bm{\pi}\right\|^{k,\emptyset}_1  \left\|\bm{\pi}\right\|^{t,\emptyset}_1 \right) + \mathcal{O}\left( \left\|\bm{\pi}\right\|_2^2 \right) \right]. \quad  \left(\text{Eq.\ \eqref{30112015aabb}: } c_{ij} = \mathcal{O}(1) \right) \label{10122015a}
\end{align}
Recall from Eq.\ \eqref{30112015aabb} that $c_{ij}$ can be written as a function of $g(i)$ and $g(j)$:
\begin{align*}
c_{tk} & =  \delta_{t=k}	  	
	 +\underbrace{   \frac{1}{\left\| \bm{\pi} \right\|_1}
	      \left[ \sum_{l=1}^K  \frac{\left(\left\|\bm{\pi}\right\|^{l,\emptyset}_1 \right)^2    }{ \left\| \bm{\pi} \right\|_1}
	      - \left\|\bm{\pi}\right\|^{t,\emptyset}_1 - \left\|\bm{\pi}\right\|^{k,\emptyset}_1 \right]}_{B}
	      + \mathcal{O}\left( \frac{1}{n} \right)
 \\ \Rightarrow c_{tk}^2 & = \delta_{t=k} + 2\delta_{t=k}B + B^2 + \mathcal{O}\left( \frac{1}{n} \right). \quad  \left(\text{Eq.\ \eqref{30112015aabb}: } c_{ij} = \mathcal{O}(1) \right)
\end{align*}
Then, substituting $a_k$ for $\left\|\bm{\pi}\right\|^{k,\emptyset}_1$ in Eq.\ \eqref{10122015a} (so that $ \left\| \bm{a} \right\|_1 = \left\| \bm{\pi} \right\|_1 $), we obtain
\begin{align}
 \sum_{k=1}^K \sum_{t=1}^K c_{tk}^2 \Theta\left( a_k a_t \right) & =  \sum_{k=1}^K \sum_{t=1}^K
		 \left[\delta_{k=t} + 2 \delta_{k=t} B + B^2 + \mathcal{O}\left( \frac{1}{n} \right) \right]
		    \Theta\left( a_k a_t \right) \notag
\\ 	& = \sum_{k=1}^K (1+2B) \Theta\left( a_k^2 \right)
+ \sum_{k=1}^K \sum_{t=1}^K \left[ B^2 + \mathcal{O}\left( \frac{1}{n} \right) \right]
		    \Theta\left( a_k a_t \right) \label{30112015c}	.
\end{align}
We now address the two terms on the right-hand side of Eq.\ \eqref{30112015c} separately:
\begin{align}
	\sum_{k=1}^K & (1+2B) a_k^2
	= \left\| \bm{a} \right\|_2^2 + \frac{2}{\left\| \bm{a} \right\|_1} \sum_{k=1}^K \left( \sum_{l=1}^K
	        \frac{a_l^2}{ \left\| \bm{a} \right\|_1}
	       - 2 a_k \right) a_k^2\notag
\\ 	       & \qquad \qquad \,\,\,\,\, = \left\| \bm{a} \right\|_2^2 +  2\frac{\left\| \bm{a} \right\|_2^4}{\left\| \bm{a} \right\|_1^2}
	       - 4 \frac{\left\| \bm{a} \right\|_3^3}{\left\| \bm{a} \right\|_1} . \label{30112015}
\\\sum_{k=1}^K & \sum_{t=1}^K \left[ B^2 + \mathcal{O}\left( \frac{1}{n} \right) \right]
		    a_k a_t \notag
\\		   &  = \sum_{k=1}^K \sum_{t=1}^K   \left\{\frac{1}{\left\| \bm{a} \right\|_1}
 \left[ \sum_{l=1}^K
	        \frac{\left(a_l\right)^2}{ \left\| \bm{a} \right\|_1}
	       - a_k - a_t \right]\right\}^2
		    a_k a_t	+ \mathcal{O}\left( \frac{\left\| \bm{a} \right\|_1^2}{n} \right) \notag
\\ 	   &  = \frac{1}{\left\| \bm{a} \right\|_1^2} \sum_{k=1}^K \sum_{t=1}^K   \left[\frac{\left\| \bm{a} \right\|_2^2}{ \left\| \bm{a} \right\|_1}
	       - \left(a_k + a_t\right) \right]^2
		    a_k a_t + \mathcal{O}\left( \frac{1}{n} \right) + \mathcal{O}\left( \frac{\left\| \bm{a} \right\|_1^2}{n} \right) \notag
\end{align}
\begin{align}
 	   &  = \frac{1}{\left\| \bm{a} \right\|_1^2} \sum_{k=1}^K \sum_{t=1}^K   \left[ \left(\frac{\left\| \bm{a} \right\|_2^2}{ \left\| \bm{a} \right\|_1}\right)^2
	       - 2 \frac{\left\| \bm{a} \right\|_2^2}{ \left\| \bm{a} \right\|_1}
	       \left(a_k + a_t\right)
	       +\left(a_k + a_t\right)^2 \right]
		    a_k a_t	+ \mathcal{O}\left( \frac{\left\| \bm{a} \right\|_1^2}{n} \right) \notag	
 	\\ 	   &  = \frac{1}{\left\| \bm{a} \right\|_1^2}
	\left[ \left\| \bm{a} \right\|_2^4
	- 2  \left\| \bm{a} \right\|_2^4
	+ \sum_{k=1}^K \sum_{t=1}^K
	      \left(a_k^2 + 2 a_k a_t + a_t^2\right)
		    a_k a_t \right] + \mathcal{O}\left( \frac{\left\| \bm{a} \right\|_1^2}{n} \right) \notag
\\ 	   &  = \frac{1}{\left\| \bm{a} \right\|_1^2}
	\left[ \left\| \bm{a} \right\|_2^4
	- 2  \left\| \bm{a} \right\|_2^4
	+ 2 \left\| \bm{a} \right\|_3^3  \left\| \bm{a} \right\|_1
	+ 2 \left\| \bm{a} \right\|_2^4 \right] + \mathcal{O}\left( \frac{\left\| \bm{a} \right\|_1^2}{n} \right) \notag
\\ 	   &  = \frac{1}{\left\| \bm{a} \right\|_1^2}
	\left[ \left\| \bm{a} \right\|_2^4
	+ 2 \left\| \bm{a} \right\|_3^3  \left\| \bm{a} \right\|_1 \right] + \mathcal{O}\left( \frac{\left\| \bm{a} \right\|_1^2}{n} \right). \label{30112015b} 	 \end{align}

Thus, substituting Eqs.\ \eqref{30112015} and~\eqref{30112015b} into Eq.\ \eqref{10122015a}, we obtain
\begin{align}
\sum_{j = 1}^n & \sum_{i<j} c_{ij}^2 \V  A_{ij}
 = \Theta\left( \left\| \bm{a} \right\|_2^2 + 3\frac{\left\| \bm{a} \right\|_2^4}{\left\| \bm{a} \right\|_1^2}
	       - 2 \frac{\left\| \bm{a} \right\|_3^3}{\left\| \bm{a} \right\|_1} \right) + \mathcal{O}\left( \frac{\left\| \bm{a} \right\|_1^2}{n} + \left\| \bm{\pi} \right\|_2^2 \right) \label{01122015}
\\ & = \left\|\bm{a}\right\|_2^2 \left[ \Theta\left(  1
 	+ 3 \frac{\left\| \bm{a} \right\|_2^2}{\left\| \bm{a} \right\|_1^2}
	       - 2 \frac{\left\|\bm{a}\right\|_2\left\| \bm{a} \right\|_3^3}{\left\| \bm{a} \right\|_1\left\|\bm{a}\right\|_2^3} \right) + \mathcal{O}\left( \frac{\left\| \bm{a} \right\|_1^2}{\left\| \bm{a} \right\|_2^2} \left\{ \frac{1}{n} + \frac{\left\| \bm{\pi} \right\|_2^2}{\left\| \bm{a} \right\|_1^2} \right\} \right)    \right]. \notag
\intertext{Since $\left\| \bm{a} \right\|_1 = \left\| \bm{\pi} \right\|_1$ and $\left\| \bm{a} \right\|_1^2/\left\| \bm{a} \right\|_2^2 \leq K$, it follows that}	
  & = \left\|\bm{a}\right\|_2^2 \left[\Theta\left(  1
 	+ 3 \frac{\left\| \bm{a} \right\|_2^2}{\left\| \bm{a} \right\|_1^2}
	       - 2 \frac{\left\|\bm{a}\right\|_2\left\| \bm{a} \right\|_3^3}{\left\| \bm{a} \right\|_1\left\|\bm{a}\right\|_2^3} \right) + \mathcal{O}\left( K \left\{ \frac{1}{n} + \frac{\left\| \bm{\pi} \right\|_2^2}{\left\| \bm{\pi} \right\|_1^2} \right\} \right)     \right]\notag
\\  & = \left\|\bm{a}\right\|_2^2 \left[ \Theta\left(  1
 	+ 3 \frac{\left\| \bm{a} \right\|_2^2}{\left\| \bm{a} \right\|_1^2}
	       - 2 \frac{\left\|\bm{a}\right\|_2\left\| \bm{a} \right\|_3^3}{\left\| \bm{a} \right\|_1\left\|\bm{a}\right\|_2^3} \right) + \mathcal{O}\left( \frac{K}{n} \right) \right] \quad \left( \text{Assumption \ref{nonodeincontrol}} \right)\label{14122015}
  \\ & \geq  \left\|\bm{a}\right\|_2^2 \left[\Theta\left( 1
 	+ 3 \frac{\left\| \bm{a} \right\|_2^2}{\left\| \bm{a} \right\|_1^2}
	       - 2 \frac{\left\|\bm{a}\right\|_2}{\left\| \bm{a} \right\|_1}\right) + \mathcal{O}\left( \frac{K}{n} \right)   \right]\notag
  \\ & = \left\|\bm{a}\right\|_2^2 \left[ \Theta\left( \left[ \sqrt{3} \frac{\left\|\bm{a}\right\|_2}{\left\| \bm{a} \right\|_1} - \frac{1}{\sqrt{3}} \right]^2
 	+ \frac{2}{3} \right) + \mathcal{O}\left( \frac{K}{n} \right) \right] \notag
  \\ & = \Theta\left( \left\|\bm{a}\right\|_2^2 \right) . \quad \left( K = o(n) \right) \label{1612215m}
\end{align}
Furthermore, from Eq.\ \eqref{14122015} we obtain that
\begin{align}
\sum_{j = 1}^n \sum_{i<j} c_{ij}^2 \V  A_{ij}
 & \leq \left\|\bm{a}\right\|_2^2 \left[ \Theta\left( 1
 	+ 3 \frac{\left\| \bm{a} \right\|_2^2}{\left\| \bm{a} \right\|_1^2} \right)
	+ \mathcal{O}\left( \frac{K}{n} \right) \right] \label{1612215l}
\intertext{and thus, since $\left\| \bm{a} \right\|_2^2 \leq \left\| \bm{a} \right\|_1^2$, we conclude from Eqs.\ \eqref{1612215m} and~\eqref{1612215l} that whenever $K = o(n)$,}
  \sum_{j = 1}^n \sum_{i<j} c_{ij}^2 \V  A_{ij} & = \Theta\left( \left\|\bm{a}\right\|_2^2 \right) . \label{VarXn16}
\end{align}

Now, since by hypothesis $ \left\| \bm{\pi} \right\|_1 \to \infty $, and by construction $ \left\| \bm{a} \right\|_1 = \left\| \bm{\pi} \right\|_1 $, we see immediately that
\begin{align*}
 \left\|\bm{a}\right\|_2^2 &\geq \frac{\left\|\bm{\pi}\right\|_1^2}{K} \quad \left(K \left\|\bm{a}\right\|_2^2 \geq \left\|\bm{a}\right\|_1^2\right)
 \\ &= \omega \left(\frac{n}{K}\right) \quad \text{(Assumption \ref{sparse})}
 \\ &= \omega \left(1\right). \quad \left( K = o(n) \right)
\end{align*}
Thus the Lyapunov condition is satisfied, and we obtain the claimed result that
\begin{align*}
\left(\V X_n\right) ^{-\frac{1}{2}} X_n \stackrel{d}{\rightarrow} \No(0,1)
\end{align*}
via the Lindeberg--Feller Central Limit Theorem.
\end{proof}

Combining Lemma \ref{Modularity_reform} and Eq.\ \eqref{mod_reform_slutsky}, we obtain that modularity $\widehat{Q}$ satisfies
\begin{align}
 \notag		\widehat{Q} &= \muMod + X_n  + \mathcal{O}_p\left(\errorModtotal\right)
 \\	\label{16122015h}	\Rightarrow \quad  \left(\V X_n\right)^{- \frac{1}{2}} \left(\widehat{Q} - \muMod\right) &=  \left(\V X_n\right)^{- \frac{1}{2}} X_n  +  \left(\V X_n\right)^{- \frac{1}{2}}\mathcal{O}_p\left(\errorModtotal\right).
\end{align}
We know from Lemma \ref{billinslayformod} that
\begin{align*}
	\left(\V X_n\right)^{- \frac{1}{2}} X_n \stackrel{d}{\rightarrow}  \No\left(0,1\right).
\end{align*}
Now, we will show that
\begin{align*}
	\left(\V X_n\right)^{- \frac{1}{2}}\errorModtotal \stackrel{n}{\rightarrow} 0.
\end{align*}
As in Lemma \ref{billinslayformod},
define
\begin{align*}
a_k=\left\|\bm{\pi}\right\|^{k,\emptyset}_1 = \sum_{i=1}^n \pi_i \delta_{g(i)=k},
\end{align*}
whence
\begin{align*}
\sum_{j=1}^n \sum_{i<j} \pi_i \pi_j \ingroup \leq \frac{1}{2} \left\|\bm{a}\right\|_2^2.
\end{align*}
Using this notation, we have from Eqs.\ \eqref{errorModTotal} and~\eqref{VarXn16} that \begin{align*}
0 \leq \errorModtotal \leq \frac{\left\|\bm{a}\right\|_2^2}{\min\left(n, \left\|\bm{\pi}\right\|_1\right) \, \min_l \sqrt{ \E d_l }}
\end{align*}
and $\V X_n = \Theta\left(\left\|\bm{a}\right\|_2^2\right)$, respectively. It follows that
\begin{align}
\notag	\left(\V X_n\right)^{- \frac{1}{2}}\errorModtotal & = \mathcal{O}\left( \left\|\bm{a}\right\|_2^{-1} \frac{\left\|\bm{a}\right\|_2^2}{\min\left(n, \left\|\bm{\pi}\right\|_1\right) \, \min_l \sqrt{ \E d_l }} \right)
\\ \notag	& = \mathcal{O}\left(  \sqrt{\frac{\left\|\bm{a}\right\|_2^2}{\min\left(n^2, \left\|\bm{\pi}\right\|_1^2\right) \, \min_l  \E d_l } } \right)
\\ \notag	& = \mathcal{O} \left(  \sqrt{\frac{\left\|\bm{\pi}\right\|_1}{\min\left(n^2, \left\|\bm{\pi}\right\|_1^2\right) \min_l \pi_l} } \right) \quad \text{(Assumption \ref{nonodeincontrol})}
\\	\notag & = o \left(  \sqrt{\frac{\left\|\bm{\pi}\right\|_1}{\min\left(n^{3/2}, n^{-1/2}\left\|\bm{\pi}\right\|_1^2\right)} } \right)
\quad \text{(Assumption \ref{sparse})}
\\ \label{16122015i}	& = o \left(1 \right).
\quad \text{(Assumption \ref{sparse} and \ref{notcomplete})}
\end{align}

We are now ready to complete the proof of Theorem \ref{DensityM}. Observe from~\eqref{sigMod} and~\eqref{Xn} that $\sigMod$ as defined in the statement of Theorem \ref{DensityM} satisfies
\begin{align*}
\sigMod^2 = \V X_n.
\end{align*}
Combining the results from Eqs.\ \eqref{16122015h},~\eqref{16122015i} and Lemma \ref{billinslayformod} using Slutsky's Theorem, we conclude the overall result of this theorem; i.e.,
\begin{align*}
	 \frac{\widehat{Q} - \muMod}{\sigMod} \stackrel{d}{\rightarrow} \No\left(0,1\right).
\end{align*}
\end{proof}
 	
\section{Proof of Theorem 3}
To add interpretability to the coefficients $\bm{\alpha}=0.5 + \bm{\beta}$ and $\bm{\beta}$ for the decomposition of modularity in Theorem 3 in the main text, we change their formulation from the one in Lemma \ref{Modularity_reform} in the proof of Theorem \ref{DensityM} (see Eq. \eqref{betaRepeat} below)
to $\beta_j^*$ in Eq. \ref{betaStar} below. By doing so, we add an error term that asymptotically wears off. More formally, we obtain the following corollary.

\begin{appxcorollary} \label{APP:LemmabetaStar}
Consider Assumptions \ref{nonodeincontrol}--\ref{over-dispersed} ($\pi_i/ \left\|\bm{\pi}\right\|_1= \mathcal{O}\left(1/ n\right)$, $\pi_i = \omega\left(1/ \sqrt{n}\right)$, $\pi_i = o\left(\sqrt{n}\right), \E A_{ij}=\Theta\left(\V A_{ij}\right)$). Then, the following identity holds:
\begin{align}
\notag 		\widehat{Q}-\muMod
 	&= \left(\sum_{j=1}^n \alpha_j^* \left[d_j^w - \E d_j^w\right]  + \sum_{j=1}^n \beta_j^* \left[d_j^b - \E d_j^b\right]\right)  + \mathcal{O}_P\left(\epsilon\right)
 	\intertext{with $\alpha_j^* = 0.5 + \beta_j^*$ and}
 	 \beta_j^* &=\frac{\sum_{l=1}^n \E d_l^w}{2 \sum_{l=1}^n \E d_l} - \frac{ \E d_j^w}{\E d_j}. \label{betaStar}
\end{align}
\end{appxcorollary}

\begin{proof}
Recall from Lemma \ref{Modularity_reform} in the proof of Theorem \ref{DensityM} that
\begin{align}
 	\notag	\widehat{Q} &= \muMod + \left(\sum_{j=1}^n \alpha_j \left[d_j^w - \E d_j^w\right]  + \sum_{j=1}^n \beta_j \left[d_j^b - \E d_j^b\right]\right)  + \mathcal{O}_P\left(\epsilon\right)
 		\intertext{where}
 	\beta_j & = \left[ \frac{1}{2} \sum_{l=1}^n   \ingroupParaL  \frac{\E d_l}{\E \left\| \bm{d} \right\|_1} -  \ingroupPara\right]  \frac{1}{\sqrt{\E \left\| \bm{d} \right\|_1}}. \label{betaRepeat}
\end{align}
We first address how $\beta_j$ and $\beta_j^*$ relate:
\begin{align}
 \notag	\beta_j & =  	\left[ \frac{1}{2} \sum_{l=1}^n   \ingroupParaL  \frac{\E d_l}{\E \left\| \bm{d} \right\|_1} - \ingroupPara\right] \frac{1}{\sqrt{\E \left\| \bm{d} \right\|_1}}
	\\ \notag & =  	\left[ \frac{\sum_{l=1}^n \sum_{m<l} \E A_{lm} \ingrouplm}{\sqrt{\E \left\| \bm{d} \right\|_1}}  \frac{\left\|\bm{\pi}\right\|_1 \left(1 - \pi_l/\left\|\bm{\pi}\right\|_1\right)}{\sqrt{\E \left\| \bm{d} \right\|_1}} - \ingroupPara\right] \frac{1}{\sqrt{\E \left\| \bm{d} \right\|_1}}
		\\ \notag  & =  	\left[ \frac{\sum_{l=1}^n \sum_{m<l} \E A_{lm} \ingrouplm}{2 \sum_{l=1}^n \sum_{m<l} \E A_{lm}}  \frac{\left\|\bm{\pi}\right\|_1 \left(1 - \pi_l/\left\|\bm{\pi}\right\|_1\right)}{\left\|\bm{\pi}\right\|_1 } - \frac{\ingroupPara}{\left\|\bm{\pi}\right\|_1}\right]
\\ & \notag \qquad \cdot		\frac{1}{\sqrt{1 - \left\|\bm{\pi}\right\|_2^2/\left\|\bm{\pi}\right\|_1^2}}
\end{align}
From Assumption \ref{nonodeincontrol} and Eq.\ \eqref{08102015}, we know that $\left\| \bm{\pi} \right\|_2^2/\left\| \bm{\pi} \right\|_1^2 \leq \max_{i} \pi_i \left\| \bm{\pi} \right\|_1 / \left\| \bm{\pi} \right\|_1^2 = \mathcal{O}\left(1/n\right)$. Hence, we can apply a convergent Taylor expansion to $f(x)=(1-x)^{-1/2}$ at $x=0$. We obtain
\begin{align}
 \notag		 & =  	\left[ \frac{\sum_{l=1}^n \sum_{m<l} \E A_{lm} \ingrouplm}{2\sum_{l=1}^n \sum_{m<l} \E A_{lm}}
		  - \frac{\ingroupPara}{\left\|\bm{\pi}\right\|_1}\right]	
		 \left[ 1 + \mathcal{O}\left( \frac{ \max_i \pi_i}{ \left\| \bm{\pi} \right\|_1} \right) \right]
		 \\ \notag		 & =  	\left[ \frac{\sum_{l=1}^n \sum_{m<l} \E A_{lm} \ingrouplm}{2\sum_{l=1}^n \sum_{m<l} \E A_{lm}}
		  - \frac{\pi_j \ingroupPara}{\pi_j \left\|\bm{\pi}\right\|_1 \left(1-\pi_j/ \left\|\bm{\pi}\right\|_1\right)}\right]
		  \left(1- \frac{\pi_j}{ \left\|\bm{\pi}\right\|_1}\right)	
\\ & \notag \qquad \cdot \left[ 1 + \mathcal{O}\left( \frac{ \max_i \pi_i}{ \left\| \bm{\pi} \right\|_1} \right) \right]
		\\ 	\notag	 & =  	\left[ \frac{\sum_{l=1}^n \sum_{m<l} \E A_{lm} \ingrouplm}{2\sum_{l=1}^n \sum_{m<l} \E A_{lm}}
		  - \frac{\pi_j \ingroupPara}{\pi_j \left\|\bm{\pi}\right\|_1 \left(1-\pi_j/ \left\|\bm{\pi}\right\|_1\right)}\right]
		 \left[ 1 + \mathcal{O}\left( \frac{ \max_i \pi_i}{ \left\| \bm{\pi} \right\|_1} \right) \right]
		 		\\ \notag		 & =  	\left[ \frac{\sum_{l=1}^n \E d_l^w}{2 \sum_{l=1}^n \E d_l}
		  - \frac{ \E d_j^w}{\E d_j}\right]
		 \left[ 1 + \mathcal{O}\left( \frac{ \max_i \pi_i}{ \left\| \bm{\pi} \right\|_1} \right) \right]
\\ 	\notag	 & =  	\left[ \frac{\sum_{l=1}^n \E d_l^w}{2 \sum_{l=1}^n \E d_l}
		  - \frac{ \E d_j^w}{\E d_j}\right]
		 \left[ 1 + \mathcal{O}\left( \frac{ 1}{ n} \right) \right] \quad \text{(Assumption \ref{nonodeincontrol})}.
\\ 		 & =  	\beta_j^*
		 \left[ 1 + \mathcal{O}\left( \frac{ 1}{ n} \right) \right]. \label{betaStar2}	
\end{align}
We now will substitute Eq. \eqref{betaStar2} into the result of Lemma \ref{Modularity_reform}. Therefore, first recall from Lemma \ref{Modularity_reform} that
\begin{align*}
 		&\widehat{Q}-\muMod
 	\\&= \left(\sum_{j=1}^n \alpha_j \left[d_j^w - \E d_j^w\right]  + \sum_{j=1}^n \beta_j \left[d_j^b - \E d_j^b\right]\right)  + \mathcal{O}_P\left(\epsilon\right)
\\ &= \left(\sum_{j=1}^n \left(0.5 + \beta_j\right) \left[d_j^w - \E d_j^w\right]  + \sum_{j=1}^n \beta_j \left[d_j^b - \E d_j^b\right]\right)  + \mathcal{O}_P\left(\epsilon\right).
 		\intertext{From Eq. \eqref{betaStar2}, it follows that}
 &= \sum_{j=1}^n \left(0.5 + \beta_j^*\right) \left[d_j^w - \E d_j^w\right]  + \sum_{j=1}^n \beta_j^* \left[d_j^b - \E d_j^b\right]  + \mathcal{O}_P\left(\epsilon\right) 	
 \\ & \quad + \mathcal{O}	\left( \frac{1}{n}\sum_{j=1}^n \beta_j^* \left[d_j - \E d_j\right]\right).
 \end{align*}
 We now address the error term:
 \begin{align*}
  \frac{1}{n}\sum_{j=1}^n \beta_j^* \left[d_j - \E d_j\right]
  & = \frac{1}{n}\sum_{j=1}^n \beta_j^* \sqrt{\E d_j}  \quad \text{(Chenyshev's inequality)}
\\ &= \mathcal{O}_P	\left( \frac{1}{n}\sum_{j=1}^n  \left(\frac{\sum_{l=1}^n \E d_l^w}{2 \sum_{l=1}^n \E d_l} + \frac{ \E d_j^w}{\E d_j}\right) \sqrt{\E d_j} \right)
  \end{align*}
\begin{align*}
 &=  \mathcal{O}_P	\left( \frac{1}{n}\sum_{j=1}^n \left(\frac{\sum_{l=1}^n \E d_l^w}{2 \sum_{l=1}^n \E d_l} \frac{\E d_j}{ \sqrt{\E d_j}} + \frac{\E d_j^w}{\sqrt{\E d_j}}\right) \right)
\\ &=  \mathcal{O}_P	\left( \frac{1}{n \min_l \sqrt{\E d_l}}
\left(\frac{\sum_{l=1}^n \E d_l^w \sum_{j=1}^n \E d_j }{2 \sum_{l=1}^n \E d_l }
+ \sum_{j=1}^n \E d_j^w\right) \right)
\\ &=  \mathcal{O}_P	\left( \frac{1}{n \min_l \sqrt{\E d_l}}
 \sum_{j=1}^n \sum_{i \neq j} \pi_i \pi_j \ingroup \right)
 \\ &=  \mathcal{O}_P	\left( \epsilon\right).
\end{align*}

As a consequence, we conclude the required result of Corollary \ref{APP:LemmabetaStar}; i.e.,
\begin{align*}
 		\widehat{Q} &= \muMod + \left(\sum_{j=1}^n \alpha_j^* \left[d_j^w - \E d_j^w\right]  + \sum_{j=1}^n \beta_j^* \left[d_j^b - \E d_j^b\right]\right)  + \mathcal{O}_P\left(\epsilon\right).
\end{align*}
\end{proof}

\section{Approximation of the bias of modularity}
We state in the main text that the shift of modularity $\muMod$ in  Theorem~\ref{DensityM} Eq.~\eqref{muMod} is equal to the approximate bias $b'$ to leading order; with
\begin{align*}
b' &= \sum_{j=1}^n \sum_{i<j} \left( \E A_{ij} - \frac{ \E d_i d_j}{\E \left\|\bm{d}\right\|_1} \right) \ingroup.
\end{align*}
More formally, we obtain the following Lemma.
\begin{appxlemma}
 		Consider Assumptions \ref{nonodeincontrol} and \ref{sparse}. Then it holds for $\muMod$ in Eq.\ \eqref{muMod} that
\begin{align*} 		
 	  \muMod &= \sum_{j=1}^n \sum_{i<j}  \left(\E A_{ij} - \frac{ \E d_i d_j}{\E \left\|\bm{d}\right\|_1}
 \left[1 + \mathcal{O}\left(\frac{ 1}{n^{3/2}}\right) \right] \right)\ingroup.
\end{align*}
\end{appxlemma}

\begin{proof}
Recall from Theorem \ref{DensityM} Eq.\ \eqref{muMod} that
\begin{align*}
	\muMod & = \sum_{j=1}^n \sum_{i<j}  \frac{\E A_{ij} \left(\E d_i + \E d_j - \left\|\bm{\pi}\right\|_2^2\right)}{\E \left\|\bm{d}\right\|_1} \ingroup
\\		&= \sum_{j=1}^n \sum_{i<j}  \frac{
\pi_i^2 \pi_j \left(\left\|\bm{\pi}\right\|_1- \pi_i\right)
+ \pi_i \pi_j^2 \left(\left\|\bm{\pi}\right\|_1- \pi_j\right)
- \pi_i \pi_j \left\|\bm{\pi}\right\|_2^2}{\E \left\|\bm{d}\right\|_1} \ingroup
\\		&= \sum_{j=1}^n \sum_{i<j} \Biggl( \frac{ \pi_i \pi_j \left\|\bm{\pi}\right\|_1^2
- \pi_i \pi_j \left\|\bm{\pi}\right\|_1^2 + \V A_{ij} - \V A_{ij}}{\E \left\|\bm{d}\right\|_1}
\\ & \qquad \quad
+ \frac{ \pi_i^2 \pi_j \left(\left\|\bm{\pi}\right\|_1- \pi_i\right)
+ \pi_i \pi_j^2 \left(\left\|\bm{\pi}\right\|_1- \pi_j\right)
- \pi_i \pi_j \left\|\bm{\pi}\right\|_2^2}{\E \left\|\bm{d}\right\|_1} \Biggr)\ingroup
  \end{align*}
\begin{align*}
		&= \sum_{j=1}^n \sum_{i<j} \Biggl( \frac{ \pi_i \pi_j \left(\left\|\bm{\pi}\right\|_1^2 - \left\|\bm{\pi}\right\|_2^2\right)
 + \V A_{ij} - \V A_{ij}}{\E \left\|\bm{d}\right\|_1}
\\ & \qquad \quad
- \frac{ \pi_i \pi_j \left\|\bm{\pi}\right\|_1^2
- \pi_i \pi_j \pi_i \left\|\bm{\pi}\right\|_1 + \pi_i^3 \pi_j
- \pi_i \pi_j \pi_j \left\|\bm{\pi}\right\|_1 + \pi_i \pi_j^3}{\E \left\|\bm{d}\right\|_1} \Biggr) \ingroup
\\		&= \sum_{j=1}^n \sum_{i<j} \Biggl( \frac{ \pi_i \pi_j \left(\left\|\bm{\pi}\right\|_1^2 - \left\|\bm{\pi}\right\|_2^2\right)
 + \V A_{ij} - \V A_{ij}}{\E \left\|\bm{d}\right\|_1}
\\ & \qquad \quad
- \frac{ \pi_i \pi_j \left(\left\|\bm{\pi}\right\|_1 - \pi_i\right)
\left(\left\|\bm{\pi}\right\|_1 - \pi_j\right) + \pi_i^3 \pi_j
 + \pi_i \pi_j^3}{\E \left\|\bm{d}\right\|_1} \Biggr)\ingroup
 \\		&= \sum_{j=1}^n \sum_{i<j} \left( \E A_{ij} - \frac{ \E d_i \E d_j
 + \V A_{ij}
 + \pi_i^3 \pi_j  + \pi_i \pi_j^3  - \V A_{ij} }{\E \left\|\bm{d}\right\|_1}\right) \ingroup.
\end{align*}
Recall from Eq.\ \eqref{Cov_degrees} that $\operatorname{cov} \left( d_i , d_j \right)  = \V  A_{ij}$ for $ i \neq j$. Furthermore, it holds that $\E d_i d_j = \E d_i \E d_j + \operatorname{cov} \left( d_i , d_j \right)$. Hence,
\begin{align*}
		&= \sum_{j=1}^n \sum_{i<j}  \left(\E A_{ij} - \frac{ \E d_i d_j
 + \pi_i^3 \pi_j  + \pi_i \pi_j^3  - \V A_{ij} }{\E \left\|\bm{d}\right\|_1} \right)\ingroup
\\ 		&= \sum_{j=1}^n \sum_{i<j} \left( \E A_{ij} - \frac{ \E d_i d_j}{\E \left\|\bm{d}\right\|_1}
 \left[1 + \frac{ \pi_i^3 \pi_j  + \pi_i \pi_j^3  - \V A_{ij} }{\E d_i d_j} \right]\right) \ingroup.
\end{align*}
We now define and analyze the error term:
\begin{align*}
 \epsilon_3 &=\frac{ \pi_i^3 \pi_j  + \pi_i \pi_j^3  - \V A_{ij} }{\E d_i d_j}
 \\ &= \frac{ \pi_i^3 \pi_j  + \pi_i \pi_j^3  - \V A_{ij} }{\E d_i \E d_j+ \V A_{ij}}
 \\  &= \frac{ \pi_i^3 \pi_j  + \pi_i \pi_j^3  - \V A_{ij} }{\pi_i \pi_j \left(\left\|\bm{\pi}\right\|_1 - \pi_i\right)
\left(\left\|\bm{\pi}\right\|_1 - \pi_j\right) + \V A_{ij}}
\\  &= \Theta\left(\frac{ \pi_i^3 \pi_j  + \pi_i \pi_j^3  - \pi_i \pi_j }{\pi_i \pi_j \left\|\bm{\pi}\right\|_1^2}\right) \quad \text{(Assumption \ref{nonodeincontrol})}
\\  &= \Theta\left(\frac{ \pi_i^2 + \pi_j^2  - 1}{ \left\|\bm{\pi}\right\|_1^2}\right)
\\  &= \mathcal{O}\left(\frac{ 1}{\min \left\{n^2, \left\|\bm{\pi}\right\|_1^2\right\}}\right) \quad \text{(Assumption \ref{nonodeincontrol})}
\\  &= \mathcal{O}\left(\frac{ 1}{n^{3/2}}\right). \quad \text{(Assumption \ref{sparse})}
\end{align*}
The required result follows; i.e.,
\begin{align*}	
  \muMod &= \sum_{j=1}^n \sum_{i<j}  \left(\E A_{ij} - \frac{ \E d_i d_j}{\E \left\|\bm{d}\right\|_1}
 \left[1 + \mathcal{O}\left(\frac{ 1}{n^{3/2}}\right) \right]\right) \ingroup. \quad \text{(Assumption \ref{sparse})}
\end{align*}		
\end{proof}

\providecommand*\hyphen{-}

\end{document}